\documentclass[11pt]{amsart}
\usepackage[british]{babel}  
\usepackage[top=2cm]{geometry}               		
\usepackage{mathrsfs}
\usepackage{bbm}					
\usepackage[parfill]{parskip} 
\usepackage{breqn}
\usepackage{bm}	
\usepackage{amsmath,amssymb,amsthm,eucal,verbatim}
\usepackage{hyperref,amsbsy}
\usepackage{graphicx}
\usepackage{epstopdf}

\usepackage{enumitem}

\usepackage{mathtools}
\theoremstyle{plain}
\newtheorem{thm}{Theorem}[section]

\newtheorem{cor}{Corollary}[section]
\newtheorem{lem}{Lemma}[section]
\newtheorem*{thm*}{Theorem}
\newtheorem*{lem*}{Lemma}
\newtheorem*{cor*}{Corollary}

\newtheorem*{rem*}{Remark}
\newtheorem{prop}{Proposition}[section]

\theoremstyle{definition}

\newtheorem*{LGF}{The Landau-Gonek Formula}
\newtheorem*{hyp}{Hypothesis $\mathscr{H}_{\alpha}$}
\newtheorem*{pcc}{Montgomery's Pair Correlation Conjecture}

\newcommand{\be}{\begin{equation}}
\newcommand{\ee}{\end{equation}}
\renewcommand{\a}{\alpha}
\newcommand{\s}{\sigma}
\renewcommand{\b}{\beta}
\renewcommand{\r}{\rho }
\newcommand{\g}{\gamma }

\renewcommand{\t}{\tau}
\renewcommand{\d}{\delta}   

\newcommand{\sgn}{\operatorname{sgn}}

\makeatletter
\newcommand*{\bigs}[1]{{\hbox{$\left#1\vbox to10\p@{}\right.\n@space$}}}
\makeatother

\makeatletter
\newcommand*{\bigss}[1]{{\hbox{$\left#1\vbox to11\p@{}\right.\n@space$}}}
\makeatother

\newcommand{\sdfrac}[2]{\mbox{\small$\displaystyle\frac{#1}{#2}$}}

\newcommand{\sdsqrt}[1]{\mbox{\small$\displaystyle\sqrt{#1}$}}

\title[\resizebox{4.8in}{!}{On The Logarithm of the Riemann zeta-function Near the Nontrivial Zeros}]{On The Logarithm of the Riemann zeta-function \\ Near the Nontrivial Zeros}
\author{Fatma Cicek}
\address{Indian Institute of Technology Gandhinagar, Palaj, Gujarat 382355, India}
\email{fcicek@iitgn.ac.in}

\calclayout

\begin{document}

\maketitle

\begin{abstract}
     Assuming the Riemann hypothesis and Montgomery's Pair Correlation Conjecture, we investigate the distribution of the  sequences  
$(\log|\zeta(\r+z)|)$ and $(\arg\zeta(\r+z)).$
    Here $\r=\frac12+i\g$ runs over the nontrivial zeros of the zeta-function, $0<\g \leq T,$ $T$ is a large real number, and $z=u+iv$ is a nonzero complex number of modulus
    $\ll1/\log T.$ Our approach proceeds via a study of  the integral moments of these sequences.     
    If we let $z$ tend to $0$ and further assume that all the zeros $\r$ are simple, we can replace the pair correlation conjecture with a weaker spacing hypothesis on the zeros and deduce that  the sequence $(\log ( |\zeta^\prime(\r)|/\log T))$ has an approximate Gaussian distribution with mean $0$ and variance $\tfrac12\log\log T.$ This gives
 an alternative proof of an old result of Hejhal and improves it by providing a rate of convergence to the distribution. 
\end{abstract}

\section{Introduction}

The study of the distribution of the logarithm of the Riemann zeta-function was precipitated by the work of Bohr and Jessen~\cite{BJ} in the early 1930s. 
They showed that for a fixed $\frac12 < \s \leq 1$ and  any rectangle $\mathcal{R}$ in the complex plane with sides parallel to the coordinate axes, the quantity 
    \[
    \frac{1}{T}\mu\Big\{T < t \leq 2T: \log\zeta(\s+it) \in \mathcal{R}\Big\}
    \]
converges to a value $\mathbb{F}_\s(\mathcal{R})$ as $T\to \infty,$ where $\mu$ denotes the Lebesgue measure and $\mathbb{F}_\s$ denotes a probability distribution function on $\mathbb{C}.$ This result is one  of the many lovely connections between probability theory and analytic number theory. Notice that this means, for example, that $ \log|\zeta(s)|$
and $ \arg\zeta(s)$ are usually bounded on the line $\Re s =\s$ when $\frac12 < \s \leq 1.$
In contrast to this,  $\log|\zeta(\frac12+it)|$ and  $\arg\zeta(\frac12+it)$ are typically much smaller or much larger, for the work of Selberg~\cite{Selberg1944, S1946Archiv} and Tsang~\cite{Tsang}
shows that for large $T$ and fixed real numbers $a, b$ with $a< b$ we have
     \be\label{CLT real}
    \begin{split}
    \frac 1T\text{meas}\Big\{T <t \leq 2T:\frac{\log|\zeta(\tfrac12+it)|}{\sqrt{\tfrac12\log\log T}}\in [a, b]\Big\}
    =\frac{1}{\sqrt{2\pi}}\int_a^b e^{-x^2/2}\mathop{dx}+O\bigg(\frac{(\log\log\log T)^2}{\sqrt{\log\log T}}\bigg)
    \end{split}
    \ee
     and
    \be\label{CLT im}
    \begin{split}
    \frac 1T\text{meas}\Big\{T <t \leq 2T:\frac{\arg\zeta(\tfrac12+it)}{\sqrt{\tfrac12\log\log T}}\in [a, b]\Big\}
    =\frac{1}{\sqrt{2\pi}}\int_a^b e^{-x^2/2}\mathop{dx}+O\bigg(\frac{\log\log\log T}{\sqrt{\log\log T}}\bigg).
    \end{split}
    \ee
Here, the function $\arg\zeta(\tfrac12+it)$ is defined as follows. 
If $t$ is not the ordinate of a zero, then starting with $\arg\zeta(2)=0,$ 
$\arg\zeta(\tfrac12+it)$ is defined by continuous variation over the line segment from  $2$ to $2+it,$ and then from $2+it$ to $\tfrac12+it.$
If $t$ is the ordinate of a zero, then we define 
\be\label{defn of arg}
\arg\zeta(\tfrac12+it)=\lim_{\epsilon \to 0} \frac{\arg\zeta(\tfrac12+i(t+\epsilon))+\arg\zeta(\tfrac12+i(t-\epsilon))}{2}.
\ee
Notice that the main term on the right-hand side of \eqref{CLT real} and \eqref{CLT im} is the distribution function of a random variable with the standard Gaussian distribution.
Indeed, these results are famously known as Selberg's central limit theorem. 

Selberg later generalized his theorem to functions in the so-called Selberg class
(see \cite{Selberg92}). 
He also gave a number of applications of \eqref{CLT real} and \eqref{CLT im}  to such problems as determining the proportion of $a$-points of linear combinations of functions in the Selberg class in various regions of the critical strip and the proportion of zeros of such combinations on the critical line.

Later in a related work, Hejhal~\cite{H} proved that, if the Riemann hypothesis (RH) is true, then the function $\displaystyle\log(|\zeta'(\tfrac12+it)|/\log t)$ has an approximate Gaussian distribution on the interval $[T, 2T]$ 
with mean $0$ and variance $\frac{1}{2}\log\log T.$
Indeed, this was later shown to hold unconditionally by Selberg in unpublished work \cite{SUnpublished}. In the same paper Hejhal further proved a discrete version of this result. 
To describe his work, we need to introduce some notation and a hypothesis. 

Let  $N(T)$ denote the number of nontrivial zeros $\r=\b+i\g$ of $\zeta(s)$ with $0<\b < 1, 0<\g \leq T.$ By the Riemann-von Mangoldt formula,
    \[
    N(T)=\frac{T}{2\pi}\log\frac{T}{2\pi}-\frac{T}{2\pi}+
    \frac{1}{\pi}\arg\zeta (\tfrac12+it)+\frac78+O\Big(\frac{1}{T}\Big). 
    \]
Following \eqref{defn of arg}, if $T$ is the ordinate of a zero, then  we set 
$N(T)=\lim_{\epsilon \to 0}\frac{N(T+\epsilon)+N(T-\epsilon)}{2}.$

For $\a$ a positive real number consider the following zero-spacing hypothesis (which inherently assumes RH).  
   \begin{hyp}\label{hypothesis}
    We have
        \[
        \limsup_{T\to\infty}\frac{1}{N(T)}\#\Big\{0 < \g \leq T:0\leq \g^+-\g \leq \frac{C}{\log T}\Big\} \ll C^\a
        \]
     uniformly for  $0<C<1.$
     Here $\tfrac12+i\g^+$ is the immediate successor of $\tfrac12+i\g$ with the convention that 
     $\g^+=\g$ if and only if  $ \tfrac12+i\g$ is a multiple zero.
\end{hyp}
Notice that if $C$ is any positive number, then the left-hand side is $\ll \min\{C^\a, 1\} N(T).$

Hejhal~\cite{H} proved that if one assumes RH, Hypothesis $\mathscr H_\a$ for some fixed $\a\in(0,1],$ and that all the zeros of the zeta-function are simple, then  as $T\to \infty,$
  \be\label{Hejhal discrete CLT}
    \begin{split}
    \frac{1}{N(2T)-N(T)}
    \# \bigg\{T < \g \leq 2T:  \frac{\log(|\zeta'(\r)|/\log T )}
    {\sqrt{\tfrac12\log\log T}}   \in [a, b]\bigg\}  
     \sim \frac{1}{\sqrt{2\pi}}\int_a^b e^{-x^2/2}\mathop{dx}.
    \end{split}
    \ee
Here, Hypothesis $\mathscr{H}_\a$ ensures that there are not too many dense clusters of zeros of the zeta-function, which is in turn necessary for controlling some of the error terms arising in the proof of \eqref{Hejhal discrete CLT}. This and similar hypotheses have been used by a number of authors, for example, see~\cite{BH}, \cite{Kirila} and~\cite{L}.    

We also remark that $\mathscr{H}_1,$ which implies  $\mathscr{H}_\a$ for every 
$\a\in(0,1].$ $\mathscr{H}_1$ is believed to be true since it 
is implied by the following well-known conjecture of Montgomery~\cite{Montgomery73}.

   \begin{pcc}\label{pcc}
    Let $\a <\b$ be real numbers and define $\delta_0=1$ if $\a\leq 0 < \beta,$ and $\delta_0=0$ otherwise. Then we have
        \[
        \frac{1}{N(T)} \sum_{\substack{0 < \g, \g' \leq T, \\ \tfrac{2\pi\a}{\log T}\leq \g-\g' \leq \tfrac{2\pi\b}{\log T}}} 1 \sim \int_\a^\b \left(1-\frac{\sin^2(\pi x)}{(\pi x)^2}+\delta_0\right) \mathop{dx}
        \]
   as $T\to\infty.$    
   \end{pcc}

Our goal in this paper is to prove a suitable discrete analogue of Selberg's central limit theorem given in \eqref{CLT real} and \eqref{CLT im}.
We also obtain a more precise version of Hejhal's result in \eqref{Hejhal discrete CLT}. 
 
 
    \begin{thm}\label{distr of Re log zeta}
    Assume the Riemann hypothesis and Montgomery's Pair Correlation Conjecture. 
    Let $z=u+iv$ be a complex number with $0 < u \ll \tfrac{1}{\log T}$ and $v=O\bigs( \tfrac{1}{\log X}\bigs),$ where 
    $\displaystyle X=T^{\frac{1}{16\Psi(T)^6}}$ with $\Psi(T)=\sum_{p\leq T} p^{-1}$ and $T$ is sufficiently large.
    Then
        \begin{align*}
        \frac{1}{N(T)}\#\bigg\{0 < \g \leq T: 
        \frac{\log|\zeta(\r+z)|-M_X(\r, z)}{\sqrt{\tfrac12\log\log T}} \in [a, b]&\bigg\} \\
        =\frac{1}{\sqrt{2\pi}}&\int_a^b e^{-x^2/2}\mathop{dx}
        +O\bigg(\frac{(\log\log\log T)^2}{\sqrt{\log\log T}}\bigg),
        \end{align*}
       where 
       \be\label{mean}
       \begin{split}
       M_X(\r, z)= m(\r+iv)\Big(\log\Big(\sdfrac{eu\log X}{4}\Big)-\sdfrac{u\log X}{4}\Big).
      \end{split}
      \ee
      Here, $m(\r+iv)$ denotes the multiplicity of the zero at $\r+iv$ if it is a zero of $\zeta(s),$ otherwise $m(\r+iv)=0.$
       \end{thm}
       We shall see later in the proof of Proposition \ref{zero spacing eta v} where we use Montgomery's Pair Correlation Conjecture that when $v\neq 0,$ $\r+iv$ is not usually a zero of the zeta-function, and so $M_X(\r, z)=0.$
    

    \begin{thm}\label{distr of Im log zeta}
    Assume the Riemann hypothesis. 
    Let $z=u+iv$ be a complex number with $0 < u \leq \tfrac{1}{\log X}$ and $v=O\bigs(\tfrac{1}{\log X}\bigs),$ where 
    $\displaystyle X=T^{\frac{1}{16\Psi(T)^6}}$ with $\Psi(T)=\sum_{p\leq T} p^{-1}$ and $T$ is sufficiently large.
    Then
        \[
        \frac{1}{N(T)}\#\bigg\{0 < \g \leq T:
         \frac{\arg\zeta(\r+z)}{\sqrt{\tfrac12\log\log T}} \in [a, b]\bigg\} 
        =\frac{1}{\sqrt{2\pi}}\int_a^b e^{-x^2/2}\mathop{dx}
        +O\bigg(\frac{\log{\log\log T}}{\sqrt{\log\log T}}\bigg).
        \]
    \end{thm}
    Note that unlike Theorem~\ref{distr of Re log zeta}, Theorem~\ref{distr of Im log zeta} does not require the assumption of Montgomery's Pair Correlation Conjecture. 

    Also, Theorem~\ref{distr of Re log zeta} is uniform in $u.$ Letting $v=0$ and then $u\to 0^+$ in the statement of the theorem, we immediately deduce the following corollary.
    
    
    \begin{cor} \label{cor: log zeta'}
     Assume the Riemann hypothesis and Montgomery's Pair Correlation Conjecture, and assume in addition that all zeros of the zeta-function are simple. Then for sufficiently large $T,$ 
        \[
        \begin{split}
        \frac{1}{N(T)}\#\bigg\{0 < \g \leq T: 
        \frac{\log(|\zeta'(\r)|/\log T)}{\sqrt{\tfrac12\log\log T}}\in [a, b]\bigg\}& \\ 
        =\frac{1}{\sqrt{2\pi}}&\int_a^b e^{-x^2/2}\mathop{dx}
        +O\bigg(\frac{(\log\log\log T)^2}{\sqrt{\log\log T}}\bigg).
        \end{split}
        \]
     \end{cor}   
In fact, we can slightly weaken the hypotheses of the corollary.  


    \begin{thm} \label{thm: log zeta'}
    Assume the Riemann hypothesis and Hypothesis $\mathscr H_\a$ for some $\a\in (0,1].$ If all zeros of the zeta-function are simple, then for sufficiently large $T$
        \[ 
        \begin{split} 
        \frac{1}{N(T)}\#\bigg\{0 < \g \leq T: 
        \frac{\log(|\zeta'(\r)|/\log T)}{\sqrt{\tfrac12\log\log T}}\in [a, b]\bigg\}&  \\
        =\frac{1}{\sqrt{2\pi}}&\int_a^b e^{-x^2/2}\mathop{dx}
        +O\bigg(\frac{(\log\log\log T)^2}{\sqrt{\log\log T}}\bigg).
        \end{split}
        \]
     \end{thm}   
Observe that this theorem improves Hejhal's result \eqref{Hejhal discrete CLT} by providing an error term. 

It does not seem that Hypothesis $\mathscr H_\a$ (together with RH)  is sufficient to prove Theorem \ref{distr of Re log zeta}. 
However, though we shall not do so, we feel it is worth remarking that  one can prove Theorem \ref{distr of Re log zeta} under the assumption of the following alternative zero-spacing hypothesis (and RH):

There is an $\a \in (0, 1]$ such that for every real number $\t,$ we have   
        \[
        \limsup_{T\to\infty}\frac{1}{N(T)}\#\Big\{0 < \g \leq T:0\leq \g^+_\t-(\g+\t) \leq \sdfrac{C}{\log T}\Big\} \ll C^\a
        \]
 uniformly for  $0<C<1.$
     Here $\tfrac12+i\g^+_\t$ is the zero that immediately follows $\tfrac12+i(\g+\t).$
     Moreover, one has $\g^{+}_{\t}  = \g+\t$ if and only if  $ \tfrac12+i(\g+\t)$ is a multiple zero.

A word is in order concerning our use of RH and Montgomery's Pair Correlation Conjecture or Hypothesis 
$\mathscr H_\a.$
The proofs of our theorems depend on the calculation of  moments  of the form     \[
    \sum_{0<\g\leq T} A(\r)^j \mkern4.5mu\overline{\mkern-4.5mu B(\r)}{}\,^k
    \]
where  $A(s)$ and $B(s)$ are Dirichlet polynomials and $\rho$ runs over the zeros of the zeta-function. 
A formula of Landau and Gonek~\cite{GonekLandaulemma1, GonekLandaulemma2}  allows one to estimate sums of the type
    \[
    \sum_{0<\g\leq T} A(\r)^j B(1-\r)^k,
    \]
which are of the above type, provided that RH is true.
In addition to RH, we need to assume Montgomery's Pair Correlation Conjecture 
in Theorem \ref{distr of Re log zeta}, or Hypothesis $\mathscr H_\a$
in Theorem~\ref{thm: log zeta'}, in order to control one of the error terms in our moment calculations.

The remainder of this  paper is organized into five sections. 
In Section~\ref{sec:approximate formula for zeta} we use  Dirichlet polynomials over the primes
to approximate the real and imaginary parts of $\log \zeta(\r+z).$
In Section~\ref{sec:lemmas} we present a number of technical lemmas. 
In Section~\ref{moments} we calculate some discrete moments related to the real part of $\log\zeta(\r+z).$
Section~\ref{sec:proof} is where we complete the proof of Theorem~\ref{distr of Re log zeta}.
The proof of  Theorem~\ref{distr of Im log zeta} is very similar and easier, so we do not include it.
Finally, we prove Theorem~\ref{thm: log zeta'} in 
Section~\ref{proof of thm Re log zeta'}. 
 
Throughout the paper, we assume RH
and take $T$ to be a sufficiently large positive real number. 
We suppose that $c, A$ and $D$ always denote positive constants, and they may be different at each occurrence. $c$ denotes an absolute constant, while $A$ and $D$ depend on some parameters. 
The variables $p$ and $q,$ indexed or not, are reserved to denote prime numbers, and the variables $j, k, \ell, m$ and $n$ always denote nonnegative integers.


\section*{Acknowledgements}

    The author gives sincere thanks to her doctoral advisor Steven M. Gonek for introducing the problem in this paper
    and also for providing guidance and support during the process of its study.
    Professor Gonek also read an earlier version of this paper and made many useful suggestions which significantly improved the exposition. 
    
    In the beginning stages of this work, the author was partially supported by the NSF grant DMS-1200582 through her advisor.


\section{Approximate Formulas} \label{sec:approximate formula for zeta}

Our goal in this section is to prove approximate formulas for the real and imaginary parts of $\log\zeta(\r+z),$
where $\r=\tfrac12+i\g$ denotes a typical nontrivial zero of $\zeta(s)$ with multiplicity $m(\r),$ and $z=u+iv$ denotes a complex-valued shift such that $0 < u=\Re z \leq \tfrac{1}{\log X}$ and $v=\Im z= O\Big(\tfrac{1}{\log X}\Big).$
Here, $4\leq X\leq t^2$ for a sufficiently large number $t.$
We also set $s=\s+it.$

The Riemann zeta-function has the two well known expressions  
     \[
    \zeta(s)=\sum_{n=1}^\infty \frac{1}{n^s} =\prod_{p}\Big(1-\frac{1}{p^s}\Big)^{-1}
    \qquad \text{for} \quad \s>1.
    \]
From the Euler product one finds that
    \[
    \log\zeta(s)=\sum_{n=1}^\infty \frac{\Lambda(n)}{n^s\log n} \qquad \text{for} \quad \s>1,
    \]
    where $\Lambda(n)$ is  the  von Mangoldt function. 
This last series is absolutely convergent in the half-plane $\s>1.$
It is not difficult to show that if  we truncate it at $X^2$ and only include the primes, 
the resulting Dirichlet polynomial $\sum_{p\leq X^2} \frac{1}{p^s}$ provides a good approximation to $\log\zeta(s)$ in this region.
In Lemma \ref{Re log zeta} and Lemma \ref{arg zeta}, we show that we may similarly use a Dirichlet polynomial to approximate $\log\zeta(\r+z).$
 
Before stating our lemmas we require some additional notation. Let 
        \be\label{P defn}
        \CMcal{P}_X(\g+v)=\sum_{p\leq X^2}\frac{1}{p^{1/2+i(\g+v)}}.
        \ee
Also let
         \be\label{Lambda_X}
        \begin{split}
        \Lambda_X(n)=\Lambda(n)w_X(n),
        \end{split}
        \ee 
where        
        \be\label{w_X}
        \begin{split}
        w_X(n)=
        \begin{cases}
        1&\quad \text{if} \quad 1\leq n \leq X,\\
         \tfrac{\log{(X^2/n)}}{\log{X}} &\quad \text{if} \quad  X< n \leq X^2.
        \end{cases}
        \end{split}
        \ee
Finally, we set
       \[
       \s_1=\frac12+\frac{4}{\log X},
       \]
and       
       \be\label{eta}
        \eta_{\g+v} =\min_{\substack{\g' \neq \g+v}} |\g'-(\g+v)|,
        \ee
where $\g'$ runs over all ordinates of the nontrivial zeros of the zeta-function.        
Notice, in particular, that $\eta_\g$ is the distance from $\g$ to the  nearest ordinate of a zero other than $\r.$


   \begin{lem} \label{Re log zeta}
    Let $4\leq X\leq T^2,$ 
    and $\displaystyle \s_1=\tfrac12+\tfrac{4}{\log X}.$
    Let $z=u+iv$ denote a complex number with $0 <  u \leq \tfrac{1}{\log X}$ and $\displaystyle  v=O\bigs(\tfrac{1}{\log X}\bigs).$ Then 
        \be\label{eq:Re log zeta}  
        \log|\zeta(\r+z)|  
        \, = M_X(\r, z)+\Re\CMcal{P}_X(\g+v) 
        +O\bigg(\sum_{i=1}^{4}r_i(X, \g +v)\bigg).
        \ee
        Here
        $M_X(\r, z)$ is as defined in \eqref{mean}, and $\CMcal{P}_X(\g+v)$ is as defined in \eqref{P defn}. We also have
        \begin{align*}
         r_1(X, \g+v)=\bigg|\sum_{p\leq X^2}&\frac{1-w_X(p)}{p^{1/2+i(\g+v)}}\bigg| \, ,  \quad
        r_2(X, \g+v)=\bigg|\sum_{p\leq X}\frac{w_X(p^2)}{p^{1+2i(\g+v)}}\bigg|\, , \\ 
         r_3(X, \g+v)=\, &\frac{1}{\log X} \int_{1/2}^\infty X^{\tfrac12-\s}\bigg|\sum_{p\leq X^2}\frac{\Lambda_X(p)\log{(Xp)}}{p^{\s+i(\g+v)}}\bigg|\mathop{d\s}, \\
        \shortintertext{and}\hspace{1.8 cm}
          r_4(X, \g+v)&=\bigg(1+\log^{+}\Big(\sdfrac{1}{\eta_{\g+v}\log X }\Big)\bigg)\sdfrac{E(X,\g+v)}{\log X}\, ,
        \end{align*}  
where
        \[
        \quad E(X, \g+v) =\bigg|\sum_{n\leq X^2} 
        \sdfrac{\Lambda_X(n)}{n^{\s_1+i(\g+v)}}\bigg|+\log(\g+v).
        \]   
\end{lem}

A similar result holds for $\arg\zeta(\r+z).$


\begin{lem} \label{arg zeta}
    With the same notation as in Lemma~\ref{Re log zeta}, we have
        \be\label{Im part 1}
        \arg\zeta(\r+z) 
        =\Im{\CMcal{P}_X(\g+v)}
        +O\bigg(\sum_{i=1}^{3}r_i(X, \g +v)\bigg)
        +O\bigg(\frac{E(X, \g+v) }{\log X}\bigg). 
        \ee
\end{lem} 

The statement of Lemma \ref{Re log zeta} is uniform in $u.$
Subtracting $M_X(\r,z)$ from both sides of \eqref{eq:Re log zeta}, letting $v=0$ and then $u$ tend to $0$ from the right, we immediately obtain


\begin{cor} \label{Re log zeta`}
    With the same notation as in Lemma \ref{Re log zeta}, 
        \[
        \log\bigg|\frac{\zeta^{(m(\r))}(\r)}{(m(\r))!}\bigg|-m(\r)\log\Big(\sdfrac{e\log X}{4}\Big)
        =\Re\CMcal{P}(\g)
        +O\bigg(\sum_{i=1}^{4}r_i(X, \g)\bigg).
        \]
\end{cor}

We start with the proof of Lemma \ref{Re log zeta}. This is more complicated than the proof of Lemma \ref{arg zeta},
where some of the intermediate results below will be reused.


\begin{proof}[Proof of Lemma \ref{Re log zeta}]
    Let $ 4\leq X\leq t^2$ and $t\geq 2.$ We write
        \[
        s_1=\s_1+it \quad \text{and} \quad s=\s+it,
        \]
    where $s$ is not a zero of $\zeta(s).$
    By (14.21.4) in \cite{T}, 
        \be\label{Z'/Z}
        \frac{\zeta^{'}}{\zeta}(s) 
        =  - \sum_{n\leq X^2} \frac{\Lambda_X(n)}{n^s} 
        +O\bigs(X^{\frac12-\s}E(X, t)\bigs) \quad \text{for} \quad  \s\geq \s_1,
        \ee
        where
        \be\label{E} 
        E(X, t) =\Big|\sum_{n\leq X^2} \frac{\Lambda_X(n)}{n^{\s_1+it}}\Big|+\log t. 
        \ee
    We also have for  $s$ not equal to any zero $\r^\prime =\frac12+i\g^\prime$ that
        \be\label{Z'/Z 2}
        \frac{\zeta^{'}}{\zeta}(s) 
        =\sum_{\r^\prime} \Big(\frac{1}{s-\r'} +\frac1{\r'}\Big)   +O(\log t).
        \ee
    Taking real parts of both sides and setting $s=s_1,$ we find that
        \[
        \Re \frac{\zeta^{'}}{\zeta}(s_1)
         =\sum_{\r^\prime} \frac{\s_1-1/2}{(\s_1-1/2)^2 +(t-\g^\prime)^2}
         +O(\log t).
        \]
    Using this and \eqref{Z'/Z} with $s=s_1,$ we see that 
        \be\label{Sum<E}
        \sum_{\r^\prime} \frac{\s_1-1/2}{(\s_1-1/2)^2 +(t-\g^\prime)^2} 
        \ll  E(X, t).
        \ee
        
    Now suppose that $\r=\frac 1 2+i\g$ is a fixed zero with $0<\g \leq T,$ and choose a number $z=u+iv$ with $0 <u \leq  \s_1-\tfrac12$ and $ v=O\bigs(\tfrac{1}{\log X}\bigs).$
   Then we have
        \be\label{logzeta rho+z}
        \begin{split}
        \log\zeta(\r+z)
        &= -\int_{\frac12+u}^\infty \frac{\zeta^{'}}{\zeta}(\s+i(\g+v))\mathop{d\s}  \\
        &=-\int_{\s_1}^\infty \frac{\zeta^{'}}{\zeta}(\s+i(\g+v)) \mathop{d\s}    
        -\Big(\s_1-\sdfrac12-u\Big) \frac{\zeta^{'}}{\zeta}(\s_1+i(\g+v))  \\
        &\hskip.5in+\int_{1/2+u}^{\s_1} \Big(\frac{\zeta^{'}}{\zeta}(\s_1+i(\g+v)) -\frac{\zeta^{'}}{\zeta}(\s+i(\g+v))\Big)\mathop{d\s}  \\[3pt]
        &= J_1 +J_2 + J_3.   
         \end{split}
         \ee
    By \eqref{Z'/Z}  
        \begin{align}\label{J_1}
        J_1
        =&\sum_{n\leq X^2} \frac{\Lambda_X(n)}{n^{\s_1+i(\g+v)} \log n}
        +O\bigg(\frac{E(X, \g+v) }{\log X}\bigg) .
        \end{align}
    Again by \eqref{Z'/Z} with $\s=\s_1,$ 
        \be\label{J_2}
        J_2 \ll \Big(\s_1-\sdfrac12-u\Big)  E(X, \g+v)  \ll \frac{E(X,  \g+v)}{\log X} .
        \ee
     By \eqref{logzeta rho+z}, the last two estimates imply
        \be\label{Re part}
        \log|\zeta(\r+z)| 
        =  \Re\, \sum_{n\leq X^2} \frac{\Lambda_X(n)}{n^{\s_1+i(\g+v)} \log n}
        +\Re\, J_3 
        +O\bigg(\frac{E(X, \g+v)}{\log X}\bigg).
         \ee
    In regards to $\Re  J_3,$ we have by \eqref{Z'/Z 2}
         \begin{align*} 
        \Re\bigg(\frac{\zeta^{'}}{\zeta}(\s_1&+i(\g+v)) -\frac{\zeta^{'}}{\zeta}(\s+i(\g+v)) \bigg)  \\
        =&  \sum_{\g^\prime} \bigg( \frac{\s_1-1/2}{ (\s_1-1/2)^2+ (\g+v-\g^\prime)^2}-\frac{\s-1/2}{(\s-1/2)^2+ (\g+v-\g^\prime)^2}  \bigg) 
         +O(\log \g)
         \end{align*}
         for $\frac12\leq \s \leq \s_1$ and $s\neq \r.$
        We separate out the terms  $\g'$ corresponding to $\g+v,$ if any, from the sum. There are $m(\r+iv)$ of them, so we find that
         \begin{align*}
        &\bigg| \Re \bigg(\frac{\zeta^{'}}{\zeta}(\s_1+ i(\g+v)) -\frac{\zeta^{'}}{\zeta}(\s+ i(\g+v)) \bigg) 
       -m(\r+iv)\bigg(\frac{1}{ \s_1-1/2}-\frac{1}{\s-1/2}\bigg)\bigg| \\
        \leq & \sum_{\g'\neq \g+v}   \frac{ \big|(\s_1-1/2)
        \big((\s-1/2)^2+(\g+v-\g')^2 \big)-(\s-1/2) \big((\s_1-1/2)^2+(\g+v-\g')^2 \big)\big|}
        {\big((\s_1-1/2)^2+(\g+v-\g')^2 \big) \big((\s-1/2)^2+(\g+v-\g')^2 \big)}    \\
        &\quad +O(\log \g) \\
        = & \sum_{\g'\neq \g+v}   \frac{(\s_1-\s)\bigs|-(\s_1-1/2)(\s-1/2)+(\g+v-\g')^2\bigs|}{\big((\s_1-1/2)^2+(\g+v-\g')^2 \big) \big((\s-1/2)^2+(\g+v-\g')^2 \big)}\\
        &\quad +O(\log \g).
         \end{align*} 
        Integrating the first and the last term of the inequalities over $\s \in [\frac12+u, \s_1],$ we deduce by triangle inequality that
        \begin{align}\label{ReJ3}
         &\qquad \bigg|\Re \,J_3
         -m(\r+iv)\bigg(\frac{\s_1-1/2-u}{\s_1-1/2}- \frac12\log \frac{(\s_1-1/2)^2}{u^2}\bigg)\bigg| \notag \\
         \leq 
          &\sum_{\g'\neq \g+v} \int_{1/2+u}^{\s_1}  \frac{(\s_1-\s)(\s_1-1/2)(\s-1/2)}
         {\big((\s-1/2)^2+(\g+v-\g')^2 \big)\big((\s_1-1/2)^2+(\g+v-\g')^2 \big)} \mathop{d\s} \\
         + \sum_{\g'\neq \g+v}  &\frac{1}{(\s_1-1/2)^2+(\g+v-\g')^2}\int_{1/2+u}^{\s_1}
          \frac{(\s_1-\s)(\g+v-\g')^2}{(\s-1/2)^2+(\g+v-\g')^2} \mathop{d\s}
         +O \bigg( \frac{\log{\g}}{\log X} \bigg) \notag .
        \end{align}
    The term being subtracted on the left-hand side is the function $M_X(\r, z)$ in \eqref{mean}, namely,  
        \[
        M_X(\r, z)=m(\r+iv)\Big(\log\Big(\sdfrac{eu\log X}{4}\Big)-\sdfrac{u\log X}{4}\Big).
        \]
      We now study the second sum on the right-hand side of  \eqref{ReJ3}. The integral is at most $(\s_1-1/2)^2.$ Thus the second sum is
        \be\label{second sum J3}
       \ll \sum_{\g'\neq \g+v}    \frac{(\s_1-1/2)^2 }{(\s_1-1/2)^2 +(\g+v-\g')^2} . 
        \ee
    For the first sum, note that 
        \[
        (\s_1-\s)(\s_1-1/2)(\s-1/2)  \ll (\s_1-1/2)^2(\s-1/2) \qquad \text{for} \quad \s\in[1/2+u,\s_1].
        \]    
    Thus, the first sum is
        \be\label{first sum lemma}
        \ll  \sum_{\g'\neq \g+v} 
         \frac{(\s_1-1/2)^2}
         {(\s_1-1/2)^2+(\g+v-\g')^2} \int_{1/2+u}^{\s_1}\frac{\s-1/2}    
            {(\s-1/2)^2+(\g+v-\g')^2}\mathop{d\s}. 
        \ee
       Since $\s_1-\tfrac12= \tfrac{4}{\log X},$ the integral here is
           \[
           \frac12\log{\Big(1+\sdfrac{16-u^2\log^2{X}}{u^2\log^2{X}+(\g+v-\g')^2\log^2{X}}\Big)}.
           \]  
    For $x>0,$ set $\log^+ x =\max\{\log x, 0\}.$ It is easy to check that 
    $\log(1+x)\leq 1+\log^+x,$ and that $\log^+(x/y)\leq \log^+x+\log^+(1/y).$
    Using these inequalities, we see that the above expression is
           \[
           \leq 1+ \log^+(16-u^2\log^2{X})+\log^+\Big(\sdfrac{1}{u^2\log^2{X}+(\g+v-\g')^2\log^2{X}}\Big). 
           \]
    The first two terms are $O(1)$ since $\displaystyle 0<u\leq\tfrac{1}{\log X}.$
    To estimate the third, observe that $\log^+\bigs(\tfrac{1}{x+y}\bigs)\leq \log^+\bigs(\tfrac{1}{y}\bigs)$   for
    $0< x\leq 1$ and $y>0.$ Then by the definition of $\eta_{\g+v}$ in \eqref{eta}, 
    the third term is 
            \[
            \ll \log^+\Big(\sdfrac{1}{(\eta_{\g+v}\log X)^2}\Big) 
            \ll \log^+\Big(\sdfrac{1}{\eta_{\g+v}\log X}\Big). 
            \]    
   Combining this with \eqref{first sum lemma}, we obtain from \eqref{ReJ3} and \eqref{second sum J3} that 
        \[
        \begin{split}
        \big|\Re \,J_3 &-M_X(\r, z)\big|  \\
         &\ll
        \bigg(1+\log^{+}\Big(\frac{1}{\eta_{\g+v}\log X}\Big)\bigg)  \sum_{\g' \neq \g+v}\frac{(\s_1-1/2)^2}{(\s_1-1/2)^2 +(\g+v-\g')^2 }
         +O\bigg( \frac{\log{\g}}{\log X}\bigg) .
        \end{split}
        \]
    Now, by \eqref{Sum<E}   
        \[
        \qquad \sum_{\g' \neq \g+v}   \frac{(\s_1-1/2)^2 }{(\s_1-1/2)^2 +(\g+v-\g')^2 } 
        \ll  \Big(\s_1-\sdfrac12\Big) E(X, \g+v) \ll \frac{E(X, \g+v)}{\log X}, 
        \]
     so 
     \[
     \Re \,J_3
     =M_X(\r, z)
     +\bigg(1+\log^{+}\Big(\frac{1}{\eta_{\g+v}\log X}\Big)\bigg)\frac{E(X, \g+v)}{\log X}.
     \]
     Going back to \eqref{Re part}, we have shown that
        \be \label{Re part 2}
        \begin{split}
        \log&|\zeta(\r+z)|  \\
        &=M_X(\r, z)  
        +\Re\sum_{n\leq X^2}\sdfrac{\Lambda_X(n)}{n^{\s_1+i(\g+v)} \log n} 
        +O\bigg(\Big(1+\log^{+}\Big(\sdfrac{1}{\eta_{\g+v}\log X}\Big)\Big)\sdfrac{E(X, \g+v)}{\log X}\bigg).
        \end{split}
        \ee
    Next, following \cite[p. 35]{S1946Archiv} (or \cite[p. 35]{L})
    we write the sum
        \[
        \Re{\sum_{n\leq X^2}\sdfrac{\Lambda_X(n)}{n^{\s_1+i(\g+v)}\log n}}
        \]
    in a more convenient form. 
    By the notation in \eqref{Lambda_X}, we have $\Lambda_X(p^\ell)=w_X(p^\ell)\log p\, .$Thus, the above is
        \[
        \Re{\sum_{p\leq X^2} \frac{1}{p^{\s_1+i(\g+v)}}}
        +\Re {\sum_{p\leq X^2}\sdfrac{w_X(p)-1}{p^{\s_1+i(\g+v)}}} 
        +\Re {\sum_{p^2\leq X^2}\sdfrac{w_X(p^2)}{2p^{2(\s_1+i(\g+v))}}}
        +\Re {\sum_{\substack{p^\ell \leq X^2,\\ \ell>2}}\sdfrac{w_X(p^\ell)}{\ell p^{\ell(\s_1+i(\g+v))}}}.
        \]
     This in turn equals 
        \begin{align*}
        &\Re {\CMcal{P}_X(\g+v)}
        +O\bigg(\Big|  \sum_{p\leq X^2}\sdfrac{1-w_X(p)}{p^{1/2+i(\g+v)}}\Big|\bigg)
        +O\bigg(\Big| \Re \sum_{p\leq X^2}\Big(\sdfrac{w_X(p)}{p^{\s_1+i(\g+v)}}-\sdfrac{w_X(p)}{p^{1/2+i(\g+v)}}\Big)\Big|\bigg) \\
        &+O\bigg( \Big| \sum_{p\leq X}\sdfrac{w_X(p^2)}{p^{1+2i(\g+v)}}\Big|\bigg)
        +O\bigg(\Big| \sum_{p\leq X}\sdfrac{w_X(p^2)}{p^{2\s_1+2i(\g+v)}}-\sdfrac{w_X(p^2)}{p^{1+2i(\g+v)}}\Big|\bigg)
        +O\Big(\sum_{\substack{p^\ell\leq X^2, \\ \ell>2}}\sdfrac{1}{\ell p^{\ell/2}}\Big).
        \end{align*}
    We leave the first and the third error terms intact. In the second error term we apply the mean value theorem for integrals. This ensures that there exists a number $\s_{\ast}$ between $\frac 1 2$ and $\s_1$ such that this error term is 
        \[
        \int_{1/2}^{\s_1}\bigg( \Re \sum_{p\leq X^2}\sdfrac{w_X(p)\log{p}}{p^{u+i(\g+v)}}\bigg)\mathop{du} 
        =\Big(\s_1-\sdfrac12\Big)\, \Re\sum_{p\leq X^2}\sdfrac{\Lambda_X(p)} {p^{\s_{\ast}+i(\g+v)}}.
        \]
    Using the integral $\displaystyle \int_{\s_{\ast}}^\infty \sdfrac{\mathop{d\s}}{(Xp)^\s}= \sdfrac{1}{(Xp)^{\s_\ast} \log(Xp)},$ we rewrite this, and then estimate
        \[
        \begin{split}
          \Big(\s_1-\sdfrac12\Big)X^{\s_{\ast}-\frac 1 2}
         \int_{\s_{\ast}}^\infty X^{\frac 1 2-\s}&\bigg(\Re \sum_{p\leq X^2}\sdfrac{\Lambda_X(p)\log{(Xp)}}{p^{\s+i(\g+v)}}\bigg)\mathop{d\s} \\
        &\ll  \frac{1}{\log X} \int_{1/2}^\infty X^{\frac 1 2-\s}\bigg|\sum_{p\leq X^2}\sdfrac{\Lambda_X(p)\log{(Xp)}}{p^{\s+i(\g+v)}}\bigg|\mathop{d\s}. \\
        \end{split}
        \]
    Next, since $|w_X(p^2) /p^{2i(\g+v)}| \leq 1$ the fourth error term is 
        \[
        =\bigg|\sum_{p\leq X}\sdfrac{w_X(p^2)}{p^{2i(\g+v)}}\Big(\sdfrac{1}{p}-\sdfrac{1}{p^{2\s_1}}\Big)\bigg| 
         \ll \sum_{p\leq X}\Big(\sdfrac{1}{p}-\sdfrac{1}{p^{2\s_1}} \Big) 
         \ll  \Big(\s_1-\sdfrac12\Big)\sum_{p\leq X}\sdfrac{\log p}{p} \ll 1,
        \]
     where we used Mertens' theorem
      $\sum_{p\leq X}\frac{\log p}{p}=\log X+O(1).$   
    Clearly the fifth error term is just $O(1)$ as well.  
    We combine our estimates to find that 
        \be\label{eq:prime sum for Re}
        \begin{split} 
        \qquad \Re \sum_{n\leq X^2}\sdfrac{\Lambda_X(n)}{n^{\s_1+i(\g+v)}\log n}  
        =\Re {\CMcal{P}_X(\g+v)}
        &+O\bigg(\Big| \sum_{p\leq X^2}\sdfrac{1-w_X(p)}{p^{1/2+i(\g+v)}}\Big|\bigg)
        +O\bigg( \Big| \sum_{p\leq X}\sdfrac{w_X(p^2)}{p^{1+2i(\g+v)}}\Big|\bigg)\\
        \qquad &+O\bigg(\frac{1}{\log X} \int_{1/2}^\infty X^{\tfrac12-\s}\Big|\sum_{p\leq X^2}\sdfrac{\Lambda_X(p)\log{(Xp)}}{p^{\s+i(\g+v)}}\Big|\mathop{d\s}\bigg).
        \end{split}
        \ee
    By \eqref{Re part 2} and the above result, the proof of the lemma is complete. 
\end{proof}
 
As we mentioned earlier, the proof of Lemma \ref{arg zeta} has some similarities to the proof of Lemma \ref{Re log zeta}. We will indicate those as needed. 


\begin{proof}[Proof of Lemma \ref{arg zeta}]
     By \eqref{J_1} and \eqref{J_2}, \eqref{logzeta rho+z} gives
        \be\label{Im part}
        \arg\zeta(\r+z) 
        = \Im\, J_3 + \Im\, \sum_{n\leq X^2} \frac{\Lambda_X(n)}{n^{\s_1+i(\g+v)} \log n}
        +O\bigg(\frac{ E(X, \g+v )}{\log X} \bigg),
        \ee
     where 
     \[
      \Im {J_3} 
      =\Im \int_{1/2+u}^{\s_1} \Big(\frac{\zeta^{'}}{\zeta}(\s_1+i(\g+v)) -\frac{\zeta^{'}}{\zeta}(\s+i(\g+v))\Big)\mathop{d\s} .
     \]   
    An argument similar to the one in ~\cite[pp. 8--9]{Selberg1944} can be applied to $ \Im{J_3}.$    
    By \eqref{Z'/Z 2}, we have for $\frac12\leq \s \leq \s_1$ and $s\neq \r$
        \begin{align*} 
        \Im\bigg(\frac{\zeta^{'}}{\zeta}(\s_1&+ i(\g+v)) -\frac{\zeta^{'}}{\zeta}(\s+i(\g+v)) \bigg)  \\
        &=\sum_{\g^\prime} \bigg( \frac{\g+v-\g^\prime}{ (\s_1-1/2)^2+ (\g+v-\g^\prime)^2}
         -\frac{\g+v-\g^\prime}{(\s-1/2)^2+ (\g+v-\g^\prime)^2}  \bigg) 
         +O(\log \g)\\
         &=\sum_{\g^\prime} \bigg( \frac{(\g+v-\g^\prime)\left((\s-1/2)^2-(\s_1-1/2)^2\right)}
         {\left((\s_1-1/2)^2+ (\g+v-\g^\prime)^2\right)\left((\s-1/2)^2+ (\g+v-\g^\prime)^2\right)}\bigg) 
         +O(\log \g).
         \end{align*}
    By integrating the first and the last terms over $[\frac12+u,\s_1],$ we obtain the estimate
        \[
        \begin{split}
        | \Im{J_3}|\
        &\leq \, \sum_{\g'}\frac{(\s_1-1/2)^2}{ (\s_1-1/2)^2+(\g+v-\g')^2  } \int_{1/2}^\infty \frac{|\g+v-\g'|}{(\s-1/2)^2+(\g+v-\g')^2}{d\sigma}
        +O\bigg(\frac{\log{\g}}{\log X} \bigg) \\
        &\ll \sum_{\g'}\frac{(\s_1-1/2)^2}{(\s_1-1/2)^2+(\g+v-\g')^2} 
        +O\bigg(\frac{\log{\g}}{\log X} \bigg).
        \end{split}
        \]
     The last step followed from the convergence of the integral.   
    Then from \eqref{Sum<E}, we conclude that 
    $\displaystyle |\Im{J_3}| \, \ll \tfrac{E(X, \g+v)}{\log X}$ since
    $\log \g \ll E(X, \g+v).$
    Hence by \eqref{Im part}
        \be\label{Im part 2} 
        \arg\zeta(\r+z) 
        = \Im\, \sum_{n\leq X^2} \frac{\Lambda_X(n)}{n^{\s_1+i(\g+v)}\log n}
        +O\bigg(\sdfrac{E(X, \g+v)}{\log X} \bigg).
        \ee
     The argument we used to prove \eqref{eq:prime sum for Re} 
     similarly shows that
        \[
        \begin{split} 
        \Im \sum_{n\leq X^2}\frac{\Lambda_X(n)}{n^{\s_1+i(\g+v)}\log n} 
        =\Im{\CMcal{P}_X(\g+v)}
        &+O\bigg(\Big|\sum_{p\leq X^2}\sdfrac{1-w_X(p)}{p^{1/2+i(\g+v)}}\Big|\bigg)
        +O\bigg( \Big|\sum_{p\leq X}\sdfrac{w_X(p^2)}{p^{1+2i(\g+v)}}\Big|\bigg)\\
        &+O\bigg( \frac{1}{\log X} \int_{1/2}^\infty X^{\tfrac12-\s}\Big|\sum_{p\leq X^2}\sdfrac{\Lambda_X(p)\log{(Xp)}}{p^{\s+i(\g+v)}}\Big|\mathop{d\s}\bigg).
        \end{split}
        \]
     Substituting this into \eqref{Im part 2} completes the proof. 
\end{proof}



\section{Some Preliminary Lemmas}\label{sec:lemmas}
The fundamental result that we will use throughout this section is the Landau-Gonek formula. 

    \begin{LGF}
    Assume RH and let $x  >1, T\geq 2.$ Then
        \[
        \sum_{0<\g\leq T}x^{i\g} 
        = -\frac{T}{2\pi} \frac{\Lambda(x)}{\sqrt x} +\mathcal{E}(x, T),
        \]
    where
        \be\label{LandauLemmaError}
        \begin{split}
        \mathcal{E}(x, T) \ll \sqrt{x} \log( xT) \log\log (3x)
        + \frac{ \log x}{\sqrt{x}}  \min\bigg(T, \frac{x }{\langle x  \rangle}  \bigg)
        + \frac{\log  T}{\sqrt{x} }  \min \bigg(T, \frac{ 1}{\log x}  \bigg),
        \end{split}
        \ee
    and $\langle x \rangle$ denotes the distance from $x$ to the closest prime power other than $x$ itself. 
    \end{LGF}

\begin{proof} 
    Landau~\cite{Landau} proved a weaker version of this in the sense that
    the error term was not uniform in $x.$
    The above version was later proven by 
    Gonek~\cite{GonekLandaulemma1,GonekLandaulemma2}.
\end{proof}

The rest of the results in this section are obtained from the Landau-Gonek formula. 
The following one will be essential in computing discrete moments of Dirichlet polynomials.

  \begin{lem}\label{sumanbnlemma} 
  Assume RH. 
    Let $(a_n)$ and $(b_n)$ be sequences of complex numbers and suppose that $ M, N\leq T.$ Then
        \be\label{sumanbn}
        \begin{split}
        \sum_{0 < \g \leq T}&\bigg(\sum_{n\leq N} a_n n^{-i(\g+v)}\bigg)\bigg( \mkern4.5mu\overline{\mkern-4.5mu \sum_{m\leq M}b_mm^{-i(\g+v)}}\bigg) \\
        = & \, N(T)\sum_{n\leq \min\{M, N\}}a_n\overline{b_n}  
        -\sdfrac{T}{2\pi}\sum_{m\leq M,  n\leq N}a_n\overline{b_m}\Big(\frac{m}{n}\Big)^{iv}\bigg\{\sdfrac{\Lambda(m/n)}{\sqrt{m/n}}+\sdfrac{\Lambda(n/m)}{\sqrt{n/m}}\bigg\}
        \\
        &+O\bigg(\max\{M, N\}\log^2{T}\Big(\sum_{n\leq N}|a_n|^2+\sum_{m\leq M}|b_m|^2 \Big)\bigg) \\
        +O\bigg(&\max\{M, N\} \log T\log\log T\Big(\sum_{m\leq M}\sdfrac{|b_m|}{\sqrt m}\sum_{m<n\leq N}|a_n|\sqrt n+\sum_{n\leq N}\sdfrac{|a_n|}{\sqrt n}\sum_{n<m\leq M}|b_m|\sqrt m\,\Big)\bigg).
        \end{split}
        \ee
    In particular, 
        \[
        \begin{split}
        \sum_{0 < \g \leq T} \Big| \sum_{n\leq N} &  a_n n^{-i(\g+v)}\Big|^2 \\
        =&\, N(T) \sum_{n\leq N} |a_n|^2 
        -\frac{T}{\pi}\Re\sum_{m, n\leq N}  a_n \overline{a_m}\Big( \frac{m}{n}\Big)^{iv}\frac{\Lambda(m/n) }{\sqrt{m/n}} \\
        &+O\bigg(N\log T\log \log T\sum_{n\leq N}\sum_{n<m\leq N}|a_na_m|
        \sdsqrt{\sdfrac mn}\, \bigg) 
        +O\bigg(N\log^2{T}\sum_{n\leq N}|a_n|^2\bigg).
        \end{split}
        \]
    \end{lem}  

\begin{proof}
    The left-hand side of \eqref{sumanbn} is equal to
        \be\label{proof of sumanbn}
        \begin{split}
        \sum_{m\leq M} \sum_{n\leq N} a_n \overline{b_m} \sum_{0 < \g \leq T}\Big(\frac m n\Big)^{i(\g+v)} 
        = &\, N(T) \sum_{n\leq \min\{M, N\}} a_n\overline{b_n} \\
        &+\sum_{\substack{n<m\\ m\leq M, n\leq N}}   a_n \overline{b_m} \Big(\frac{m}{n} \Big)^{iv}
        \bigg\{-\sdfrac{T}{2\pi}\sdfrac{\Lambda(m/n)}{\sqrt{m/n}}+\mathcal{E}\Big(\sdfrac{m}{n}, T\Big)\bigg\} \\
        &+\sum_{\substack{m<n\\m\leq M, n\leq N }}   a_n \overline{b_m} \Big(\frac{m}{n}\Big)^{iv} 
        \bigg\{-\sdfrac{T}{2\pi}\sdfrac{\Lambda(n/m)}{\sqrt{n/m}}+\mathcal{E}\Big(\sdfrac{n}{m}, T\Big)\bigg\},
        \end{split}
        \ee
    where we applied the Landau-Gonek formula.
    Then by \eqref{LandauLemmaError}, 
        \be\label{error from epsilon}
        \begin{split}
          \sum_{n\leq N}\sum_{n<m\leq M}  a_n \overline{b_m} \Big(\frac{m}{n} \Big)^{iv}\mathcal{E}\Big(\sdfrac{m}{n}, T\Big) 
        \ll &\sum_{n\leq N}\sum_{n<m\leq M}  |a_n b_m| \bigg\{\sdsqrt{\frac{m}{n}} \log\Big(\sdfrac{mT}{n}\Big) \log\log\Big(\sdfrac{3m}{n}\Big)\bigg\} \\
        &+ \sum_{n\leq N}\sum_{n<m\leq M}  |a_n b_m| \frac{ \log (m/n)}{\sqrt{m/n}}  \min\bigg(T, \frac{(m/n) }{\langle m/n \rangle}\bigg)  \\
        &+ \sum_{n\leq N}\sum_{n<m\leq M}  |a_n b_m| \frac{\log T}{\sqrt{m/n}}  \min \bigg(T, \frac{1}{\log (m/n)} \bigg) .
        \end{split}
        \ee
    Since $M, N\leq T,$ the first term on the right-hand side is clearly
        \be\label{eq:first term}
        \ll \log T \log \log T\sum_{n\leq N}\frac{|a_n|}{\sqrt{n}} \sum_{n<m \leq M}\sqrt m\,|b_m|.
        \ee
    For the second term, note that since $\langle m/n \rangle\geq \frac 1n$
        \[
        \begin{split}
         \frac{ \log (m/n)}{\sqrt{m/n}} \min\bigg(T, \sdfrac{(m/n) }{\langle m/n \rangle}  \bigg)
         \ll  \sqrt{mn}\log M \ll  N\sqrt{m/n}\log T.
        \end{split}
        \]
Thus, the second error term is bounded by $N$ times our bound for the first. For the third term on the right-hand side of \eqref{error from epsilon}, we note that
        \[
         \frac{\log T}{\sqrt{m/n} }  \min \bigg(T, \frac{1}{\log (m/n)}\bigg) 
         \ll   \frac{\log T}{\sqrt{m/n} \log (m/n)}
        \]
    since $\log{(m/n)\gg \frac{1}{T}}$ for $M, N\leq T.$ Then the third term is
        \[
        \begin{split}
        &\ll \log T \sum_{n\leq N} \; \sum_{n<m\leq M}\sqrt{\frac{n}{m}}\, \frac{|a_n|^2+ |b_m|^2}{ \log (m/n)} \\
        &=\log T \sum_{n\leq N} \;  \sum_{n<m\leq M}\sqrt{\frac{n}{m}}\frac{|a_n|^2}{ \log (m/n)}
        +\log T \sum_{n\leq N} \; \sum_{n<m\leq M}\sqrt{\frac{n}{m}}\frac{|b_m|^2}{\log (m/n)} \\[1ex]
        &:=S_1+S_2\, .
        \end{split}
        \]
    For $S_1,$ we separate the sum over $m$ as follows. 
        \[
       \bigg( \, \sum_{n<m \leq \min\{2n, M\}}+ 
        \sum_{2n< m\leq M}\, \bigg)\frac{1}{\sqrt m  \log (m/n)} .
        \]
    If $n<m \leq \min\{2n, M\},$ then $\log(m/n) \gg (m-n)/n.$ For the remaining $m$ which satisfy $2n< m \leq  M,$ we simply have $\log(m/n) \gg 1 .$ Using these bounds, we see that
        \begin{align*}
        S_1 &\ll \log T \sum_{n\leq N} \sqrt n\, |a_n|^2 \,\bigg(\sum_{n< m \leq \min\{2n, M\}}\frac{n}{\sqrt{m}(m-n)}+ \sum_{2n< m \leq  M}\frac{1}{\sqrt m}\bigg) \\
        &\ll \log T \sum_{n\leq N}n\, |a_n|^2  \sum_{n< m \leq \min\{2n, M\}}\frac{1}{m-n} + \sqrt M\sum_{n\leq N}\sqrt n\, |a_n|^2 \\[1.6ex]
        &\ll N \log N \log T \sum_{n\leq N} |a_n|^2 +\sqrt{NM}\log T\sum_{n\leq N} |a_n|^2 
        \ll N\log^2{T} \sum_{n\leq N} |a_n|^2 .
        \end{align*}
    The other term we need to consider is
        \[
        S_2= \log T\sum_{n\leq N} \;   \sum_{n<m\leq M}\sqrt{\frac{n}{m}}\frac{|b_m|^2}{\log(m/n)}.
        \]
    If we change the order of summation and use $\sqrt n\leq \sqrt m,$ then this is at most
        \begin{equation*}
         \log T\sum_{m\leq M} |b_m| ^2  \sum_{\substack{n<m, \\n\leq N}}\sdfrac{1}{\log (m/n)}.
        \end{equation*}
    Further, since $\log(m/n) \geq (m-n)/m$ for $n<m,$ we have
        \begin{equation*}
        S_2 
        \ll \log T\sum_{m\leq M} m|b_m| ^2 \, \sum_{\substack{n<m,\\ n\leq N}} \frac{1}{m-n}
        \ll M\log M\log T\sum_{m\leq M} |b_m| ^2 
        \ll M\log^2{T} \sum_{m\leq M} |b_m| ^2.
         \end{equation*}
    When we combine our estimates for $S_1$ and $S_2$ with \eqref{eq:first term} in \eqref{error from epsilon}, we find
        \begin{align*}
        \sum_{n\leq N} & \sum_{n< m\leq M}a_n \overline{b_m} \Big(\frac{m}{n}\Big)^{iv} \mathcal{E}\Big(\sdfrac{m}{n}, T\Big) \\
        \ll   \, N\log  T \log \log T \sum_{n \leq N}\frac{|a_n|}{\sqrt n} & \sum_{n <m\leq M }\sqrt m\, |b_m|
        \,+\, N\log^2{T}\sum_{n\leq  N} |a_n| ^2 +M\log^2{T}\sum_{m\leq M} |b_m| ^2  .
        \end{align*}
    In a similar way, one shows that
        \begin{align*}
        \sum_{m\leq M}&\sum_{m<n\leq N} a_n \overline{b_m} \Big(\frac{n}{m}\Big)^{iv}\mathcal{E}\Big(\sdfrac{n}{m}, T\Big)  \\
        \ll M \log T \log \log T \sum_{m \leq M}\frac{|b_m|}{\sqrt m} &\sum_{m <n \leq N}\sqrt{n}\, |a_n| 
        \,+\,\max\{M, N\}\log^2{T}\bigg(\sum_{n\leq N} |a_n|^2+\sum_{m \leq  M} |b_m|^2 \bigg).
        \end{align*}
    The lemma now follows from \eqref{proof of sumanbn}. 
\end{proof}
The following is an easy consequence of the previous lemma. 
We state it separately because the following form will be more convenient for our purposes. 

\begin{cor} \label{cor:moments}
Assume RH. 
Let $(a_n)$ and $(b_n)$ be two sequences of complex numbers. Suppose that $ N^j, N^{k-j}\leq T,$ where $j$ and $k$ are nonnegative integers. Then
    \be \notag
    \begin{split} 
    &\sum_{0 < \gamma \leq T}\bigg(\, \sum_{n\leq N}a_nn^{-i(\g+v)}\bigg)^j \bigg(\, \mkern4.5mu\overline{\mkern-4.5mu \sum_{m\leq N}b_mm^{-i(\g+v)}}\,\bigg)^{k-j}  \\
     =&\, N(T)\sum_{n\leq \min\{N^j, N^{k-j}\}}A_n\overline{B_n}
    -\frac{T}{2\pi}\sum_{\substack{n\leq N^j, \\ m\leq N^{k-j}}}A_n\overline{B_m}\Big(\frac mn\Big)^{iv}
    \bigg\{\frac{\Lambda(m/n)}{\sqrt{m/n}}+\frac{\Lambda(n/m)}{\sqrt{n/m}}\bigg\} \\
    &+O\bigg(N^k\log T \log\log T \Big(\sum_{n\leq N^j}\sdfrac{|A_n|}{\sqrt n}\sum_{n<m\leq N^{k-j}}\sqrt m\,|B_m|
    +\sum_{m\leq N^{k-j}}\sdfrac{|B_m|}{\sqrt m}\sum_{m<n\leq N^j}\sqrt n\, |A_n|\Big)\bigg)\\
    &+O\bigg(N^k \log^2{T}\Big(\sum_{n\leq N^j}|A_n|^2\,+\sum_{m\leq N^{k-j}}|B_m|^2 \Big)\bigg),
    \end{split}
    \ee
where 
    \[
    A_n=\sum_{n=n_1\dots n_j} a_{n_1}\dots a_{n_j} \quad \text{and} \quad B_m=\sum_{m=m_1\dots m_{k-j}}b_{m_1}\dots b_{m_{k-j}}.
    \]
\end{cor}

\begin{proof}
    This is the result of Lemma \ref{sumanbnlemma} applied to the sequences 
    $(A_n)$ and $(B_n).$
\end{proof}

The following lemma can be viewed as the discrete version of Lemma 3 in \cite{Soundararajan09}.


\begin{lem}\label{Soundmomentlemma3}
Assume RH.
    Let $k$ be a positive integer and suppose $\displaystyle 1<Y\leq (T/\log{T})^{\frac{1}{3k}}.$
    For any complex-valued sequence $(a_p)_p$ indexed by the primes, we have
        \[
        \sum_{0 < \g \leq T}\bigg| \sum_{p\leq Y}\frac{a_p}{p^{1/2+i\g}}\bigg|^{2k}
        \ll k! N(T)\Big(\sum_{p\leq Y}\sdfrac{|a_p|^2}{p}\Big)^{k}.
        \]
\end{lem}
\begin{proof}
    We begin by using the multinomial theorem to write 
        \[
        \Big(\sum_{p\leq Y}\sdfrac{a_p}{p^{1/2+i\g}}\Big)^{k}= \sum_{n\leq Y^k}\sdfrac{A_n}{n^{1/2+i\g}},
        \]
    where 
        \[
        A_n=\frac{k!}{{\a_1}!\dots {\a_r}!}a_{p_1}^{\a_1}\dots a_{p_r}^{\a_r} \qquad \text{for} \quad n=p_1^{\a_1}\dots p_r^{\a_r}.
        \]
    Here the $p_i$ are distinct primes, each of which is less than $Y,$ and the powers $\a_i\geq 1$ satisfy the condition $\a_1+\dots +\a_r=k.$ For $n$ that cannot be written as $p_1^{\a_1}\dots p_r^{\a_r}$ for such $p_i$'s and $\a_i$'s, we set $A_n=0.$ Now, by the second assertion of Lemma \ref{sumanbnlemma}, 
        \be\label{eq:proof of Soundmomentlemma3}
        \begin{split}
        \sum_{0 < \g \leq T}\bigg| \sum_{p\leq Y}\frac{a_p}{p^{1/2+i\g}}\bigg|^{2k}
        =&\, \sum_{0 < \g \leq T}\bigg(\sum_{n\leq Y^k}A_n n^{-1/2-i\g}\bigg) \bigg(\mkern4.5mu\overline{\mkern-4.5mu\, \sum_{m\leq Y^k} A_mm^{-1/2-i\g}}\bigg) \\
        =&\, N(T)\sum_{n\leq Y^k}\sdfrac{|A_n|^2}{n} 
        -\frac{T}{\pi}\Re \sum_{m\leq Y^k}\sum_{n\leq Y^k}\frac{A_n\overline{A_m}}{\sqrt{nm}}\frac{\Lambda(m/n)}{\sqrt{m/n}}\\
        &+O\bigg(Y^k\log T\log\log T \sum_{n\leq Y^k}\sum_{n<m\leq Y^k}\sdfrac{|A_n A_m|}{n}\,\bigg)
\\
        &+O\bigg( Y^k\log^2{T}\sum_{n\leq Y^k} \sdfrac{|A_n|^2}{n}\bigg).
        \end{split}
        \ee
Note that by the definition of $A_n,$
          \be\label{proof of Sound lemma}
            \begin{split}
       \sum_{n\leq Y^k}\frac{|A_n|^2}{n}
        &=\sum_{\substack{\a_i\geq 1, \sum \a_i =k \\ p_i\leq Y}}\bigg(\frac{k!}{{\a_1}!\dots {\a_r}!}\bigg)^2 \frac{|a_{p_1}|^{2\a_1}\dots|a_{p_r}|^{2\a_r}}{p_1^{\a_1}\dots p_r^{\a_r}} \\
        &\leq  k!\sum_{\substack{\a_i\geq 1, \sum \a_i =k,  \\ p_i\leq Y}}\frac{k!}{{\a_1}!\dots {\a_r}!} \frac{|a_{p_1}|^{2\a_1}\dots|a_{p_r}|^{2\a_r}}{p_1^{\a_1}\dots p_r^{\a_r}} 
        =k!\Big(\sum_{p\leq Y}\frac{|a_p|^2}{p}\,\Big)^k.
        \end{split}
        \ee
     Thus the first main term on the right-hand side of \eqref{eq:proof of Soundmomentlemma3} is
        \[
        \leq k!N(T) \Big(\sum_{p\leq Y}\sdfrac{|a_p|^2}{p}\Big)^{k}.
        \]
    The second main term on the right-hand side of \eqref{eq:proof of Soundmomentlemma3} vanishes
    because the number of prime divisors of both $m$ and $n$ is $k,$ and so if $m\neq n,$ then their ratio cannot be a nontrivial prime power. 

 Next, note that $Y^k\log^2{T} \ll N(T)$ by the choice of $Y.$ Then by \eqref{proof of Sound lemma}, the second error term is smaller than $ k! N(T) \Big(\sum_{p\leq Y}|a_p|^2 /p\Big)^{k}.$ 

    Finally, we estimate the first error term on the right hand-side of \eqref{eq:proof of Soundmomentlemma3} using the arithmetic mean-geometric mean inequality. 
        \begin{align*}
        Y^k\log T\log\log T\sum_{n\leq Y^k}\sum_{n<m\leq Y^k}\sdfrac{|A_n A_m|}{n}
        &\leq Y^k \log T\log\log T\sum_{m\leq Y^k}m \sum_{n\leq Y^k}\sdfrac{|A_n A_m|}{nm} \\
        &\leq  Y^{2k}\log T\log\log T\sum_{m\leq Y^k}\sum_{n\leq Y^k}\bigg(\sdfrac{|A_n|^2}{2n^2}+\sdfrac{|A_m|^2}{2m^2}\bigg)\\
        &\leq Y^{3k}\log T\log\log T\sum_{n\leq Y^k}\sdfrac{|A_n|^2}{n^2} \\
        &\leq  k!Y^{3k}\log^2 T\Big(\sum_{p\leq Y}\sdfrac{|a_p|^2}{p^2}\Big)^k.
        \end{align*}
    Note that the last step follows by an argument similar to the one we used in \eqref{proof of Sound lemma}. The bound we obtained here is even smaller than our bound for the second error term. 
    Hence, all of the four terms in \eqref{eq:proof of Soundmomentlemma3} are of smaller size than the $O$-term in the statement of the lemma. 
\end{proof}


\section{Moment Calculations} \label{moments}

In Lemma \ref{Re log zeta}, we saw that the sum
    \[
    \Re\CMcal{P}_X(\g+v)=\Re \sum_{p\leq X^2} \frac{1}{p^{1/2+i(\g+v)}} 
    \]
can be used to approximate the function $\log{|\zeta(\r+z)|}-M_X(\r, z)$, where
    \[
    M_X(\r, z)=m(\r+iv)\Big(\log\Big(\sdfrac{eu\log X}{4}\Big)-\sdfrac{u\log X}{4}\Big).
    \]
We will now calculate discrete integral moments of $ \Re\CMcal{P}_X(\g+v)$,
which will then allow us to compute such moments of the function 
$\log{|\zeta(\r+z)|}$ $-M_X(\r, z)$ under the assumption of RH and Montgomery's Pair Correlation Conjecture. 

We remind the reader that $z=u+iv$ denotes a complex number with
\[
     0 < u\leq \frac{1}{\log X} \quad \text{and} 
    \quad  v=O\Big(\frac{1}{\log X}\Big).
    \] 
We also take $\displaystyle X\leq T^{\tfrac{1}{8k}}$.
It will be useful in this chapter to use the notation
    \be\label{Psi}
    \Psi=\sum_{p\leq X^2} \frac 1p,
    \ee
and express some of our terms in terms of $\Psi$.
Note that by Mertens' Theorem, we have
    \[
    \Psi
    =\log\log X+O(1).
    \]
    
Our main theorem is the following.


\begin{thm}\label{moments of Re log zeta}
    Assume RH and Montgomery's Pair Correlation Conjecture.  
    Suppose that $k$ is a positive integer with 
    $k \ll \log\log\log T$, and let $T^{\tfrac{\d}{8k}} \leq X\leq T^{\tfrac{1}{8k}}$ for $0<\d \leq 1$ fixed.
    If $k$ is even, then
        \begin{align*}
         \sum_{0 < \g \leq T} \bigs(\log{|\zeta(\r+z)|}-M_X(\r, z)\bigs)^k & \\
        = \beta_kN(T) & \Psi^{\tfrac k2} 
        +O\Big(D^kk^{\tfrac{3k+2}{2}}\beta_k N(T)\Psi^{\tfrac{k-1}{2}}\Big).
        \end{align*}
    If $k$ is odd, then we have
        \begin{align*}
        \sum_{0 < \g \leq T} \bigs(\log{|\zeta(\r+z)|}-M_X(\r, z)\bigs)^k 
        =O\Big(D^k k^{\tfrac{3k+1}{2}}\beta_{k+1} N(T)\Psi^{\tfrac{k-1}{2}}\Big).
        \end{align*}
    Here, $D$ is a constant depending on $\d$ and $\Psi$ is as defined in \eqref{Psi}. We also have
        \be\label{beta}
        \beta_r=\frac{r!}{2^r (r/2)!} \quad \text{for an even positive integer } r.
        \ee  
\end{thm}

    The coefficients $\beta_r$ are closely related to the Gaussian distribution. The moments of a random variable $Z$ that has Gaussian distribution with mean $0$ and variance $V$ are as follows:
        \[
        \begin{split}
        \mathbb{E}[Z^r]=
        \begin{cases}
       \beta_r  (2V)^{r/2} \quad &r \text{ : even,} \\
        0  \quad &r \text{ : odd.}
        \end{cases}
        \end{split}
        \]
        
    It is well known that if a random variable $Z'$ has the same moments, 
    then $Z'$ has the same distribution as $Z$ (see \cite[p. 413]{Billingsley}). 
    Hence the above theorem states that for $z$ chosen as above, the sequence
    $\bigs(\displaystyle \log{|\zeta(\r+z)|}-M_X(\r, z) \bigs)$ 
    has an approximate Gaussian distribution with mean $0$ and variance  $\tfrac12\Psi.$

\subsection{Moments of $\Re{\CMcal{P}_X(\g+v)}$}
We prove the following result for the moments of  
    \[
    \Re\CMcal{P}_X(\g+v)=\Re \sum_{p\leq X^2} \frac{1}{p^{1/2+i(\g+v)}}.
    \]
Even though this proposition will not be used later, it is still of interest since we obtain an explicit main term for odd moments of the real part of the polynomial. Note, however, that Theorem \ref{moments of Re log zeta} does not provide an explicit main term for odd $k$.

\begin{prop}\label{moments of Re Dirichlet polyl v}
Assume RH. 
Let $\CMcal{P}_X(\g+v)=\sum_{p\leq X^2} p^{-1/2-i(\g+v)}$ where
$X\leq T^{\frac{1}{8k}}$ and $k \ll \sqrt[6]{\log\log T}.$ Then for even $k,$
        \[
        \sum_{0 < \g \leq T}\bigs(\Re{\CMcal{P}_X(\g+v)}\bigs)^k
         = \beta_kN(T)\Psi^{\frac k2}
        +O\Big(k^2\beta_kN(T)\Psi^{\frac{k-4}{2}}\Big).
         \]
     If $k$ is odd, then 
         \[
         \begin{split}
         \sum_{0 < \g \leq T}\bigs(\Re{\CMcal{P}_X(\g+v)}\bigs)^k
         =-\frac{\beta_{k+1}}{\pi}\frac{\sin(2v\log X)-\sin(v\log 2)}{v}T
         &\Psi^{\frac{k-1}{2}} \\
         &+O\Big(k^2\beta_{k+1}T\log X\Psi^{\frac{k-3}{2}}\Big). 
         \end{split}
         \]

\end{prop}

\begin{proof}
     Expanding the $k$th moment of $\Re\CMcal{P}_X(\g+v)$ by means of the identity $\displaystyle \Re{z}=\frac{z+\overline{z}}{2}$ and the binomial theorem, we see that
        \[
        \begin{split}
        \sum_{0 < \g \leq T}\bigs(\Re{\CMcal{P}_X(\g+v)}\bigs)^{k}
        &=\frac{1}{2^k}\sum_{j=0}^k\binom{k}{j}\sum_{0 < \g \leq T}\CMcal{P}_X(\g+v)^j \overline{\CMcal{P}_X(\g+v)}^{k-j}\\
        &=\frac{1}{2^k}\sum_{j=0}^k\binom{k}{j} S_j(v).
        \end{split}
        \]
Thus, it suffices to estimate the sums 
        \[
        S_j(v):=
        \sum_{0 < \g \leq T}\CMcal{P}_X(\g+v)^j \overline{\CMcal{P}_X(\g+v)}^{k-j}
        \]
for $j=0,1,\dots, k$.        
    We write
        \[
        \CMcal{P}_X(\g+v)^j=\sum_{\substack{n=p_1\dots p_j,\\ p_i\leq X^2}}\frac{a_j(n)}{n^{1/2+i(\g+v)}} \quad \text{and} \quad
        \CMcal{P}_X(\g+v)^{k-j}=\sum_{\substack{m=q_1\dots q_{k-j},\\ q_i\leq X^2}}\frac{a_{k-j}(m)}{m^{1/2+i(\g+v)}},
        \]
    where $a_{r}(p_1\dots p_{r})$ denotes the number of permutations of the primes $p_1,\dots, p_{r}$. It is clear that 
    $a_r(p_1\dots p_{r}) \leq r!$, where the equality holds
    if and only if the primes $p_1,\dots ,p_{r}$ are all distinct, in other words, the product $p_1\dots p_{r}$ is square-free.
    Further note that $a_0(n)$ equals $1$ if $n=1$,
    and equals $0$ otherwise.
    
    To avoid a cumbersome notation, throughout the proof of this proposition we suppose that the number $n$ is product of $j$ primes, each of which is at most $X^2$, and $m$ is product of $k-j$ primes, each of size at most $X^2$. 
    
    By Corollary \ref{cor:moments}, 
    we have the following expression for $S_j(v)$:
        \be\label{eq:S_j(v)}
        \begin{split}
        S_j(v)
        =N(T)&\sum_{n}\frac{a_{j}(n)a_{k-j}(n)}{n} 
        -\frac{T}{2\pi}\sum_{m, n}\frac{a_j(n)a_{k-j}(m)}{\sqrt{mn}}\Big(\frac{m}{n}\Big)^{iv}\bigg\{\frac{\Lambda(m/n)}{\sqrt{m/n}}+\frac{\Lambda(n/m)}{\sqrt{n/m}}\bigg\} \\
        &+O\bigg( X^{2k}\log T\log\log T\Big(\sum_n\frac{a_j(n)}{n}\sum_{m>n}a_{k-j}(m)+\sum_m\frac{a_{k-j}(m)}{m}\sum_{n>m}a_j(n)\Big)\bigg)\\
        &+O\bigg(X^{2k}\log^2{T}\Big(\sum_m\frac{a_{k-j}(m)^2}{m}+\sum_n\frac{a_j(n)^2}{n}\Big)\bigg);
        \end{split}
        \ee
   here we have suppressed the conditions of summation. It will be useful later on to note that
         \be\label{a_j sum}
         \Psi^j
         =\sum_n \frac{a_j(n)}{n}
         =\sum_{n \text{ sq-free}} \frac{a_j(n)}{n}
         +\sum_{n \text{ not sq-free}} \frac{a_j(n)}{n}.
         \ee
    The second sum on the right is zero if $j=0$ or $1$.  
    If $j \geq 2$, then
        \[
        \sum_{n \text{ not sq-free}} \frac{a_j(n)}{n}
        =\sum_{q\leq X^2} \frac{1}{q^2} 
        \sum_{n_1=\frac{n}{q^2}} \frac{a_j(q^2 n_1)}{n_1} 
        \leq \binom{j}{2} \sum_{n_1} \frac{a_{j-2}(n_1)}{n_1}
        \ll j^2 \Psi^{j-2}.
        \]
    Note the last estimate still holds when $j=0$ or $1$. 
    Hence, 
        \be\label{not sq-free}
        \sum_{n \text{ not sq-free}} \frac{a_j(n)}{n}
        \ll j^2 \Psi^{j-2}.
        \ee  
    Combining this with \eqref{a_j sum}, we then see that
        \be\label{sq-free}
        \sum_{n \text{ sq-free}} \frac{a_j(n)}{n}
        = \Psi^j+O(j^2 \Psi^{j-2}). 
        \ee
        
     Returning to \eqref{eq:S_j(v)}, we see that
        \be \label{estimate1}
        \sum_n\frac{a_j(n)}{n}\sum_{m>n}a_{k-j}(m) 
        \ll k! X^{2(k-j)}\sum_n\frac{a_j(n)}{n}
        \leq k!X^{2k}\Psi^k,
        \ee
    and
        \be \label{estimate2}
        \sum_n\frac{a_j(n)^2}{n}
        \ll k!\sum_n\frac{a_j(n)}{n} 
        \leq k!\Psi^k.
        \ee
    Hence, both of the error terms in \eqref{eq:S_j(v)} are 
    $\ll k!\Psi^k\sqrt T\log^2{T}$, and we can write
        \be\label{eq:S_j(v) rewritten}
        \begin{split}
        S_j(v)
        =&\, N(T)\sum_{n}\frac{a_j(n)a_{k-j}(n)}{n} 
        -\frac{T}{2\pi}\sum_{m, n}\frac{a_j(n)a_{k-j}(m)}{\sqrt{mn}}\Big(\frac{m}{n}\Big)^{iv}
        \bigg\{\frac{\Lambda(m/n)}{\sqrt{m/n}}+\frac{\Lambda(n/m)}{\sqrt{n/m}}\bigg\}\\
        &+O\bigs(k!\Psi^k \sqrt T\log^2{T}\bigs) \\
        =&\, S_{j,1}(v)+S_{j,2}(v)+O\bigs(k!\Psi^k \sqrt T\log^2{T}\bigs).
        \end{split}
        \ee
    Next, we estimate the terms $S_{j,1}(v)$ and $S_{j,2}(v)$.
    First observe that $S_{j,1}(v)$ vanishes unless $k$ is even and $j=\frac{k}{2}$. In that case, using \eqref{not sq-free} and \eqref{sq-free},
    we obtain
        \be\label{eq:Skby2(v)}
        \begin{split}
        S_{\frac k2,1}(v) 
        =& \, N(T)\sum_{n}\frac{a_{k/2}^2(n)}{n} \\
        =& \, (k/2)!N(T)\sum_{n\text{ sq-free}}\frac{a_{k/2}(n)}{n}
        +O\bigg((k/2)!N(T)\sum_{n\text{ not sq-free}}\frac{a_{k/2}(n)}{n}\bigg). \\
        =&\, (k/2)!N(T)\Psi^{k/2}+O\bigs(k^2(k/2)!N(T)\Psi^{k/2-2}\bigs).
        \end{split}
        \ee
    Now consider the term $S_{j,2}(v)$ in \eqref{eq:S_j(v) rewritten}. 
    In order for this term not to vanish, 
    one of the ratios $m/n$ and $n/m$ must be a prime power. 
    This condition puts a restriction on $j$:
    if $m/n=q^{\ell}$, then $j=\frac{k-\ell}{2}$, and if $n/m=q^{\ell}$, then $j=\frac{k+\ell}{2}$ . 
    Furthermore, the terms with $\ell\geq 2$ in $S_{j,2}(v)$
    contribute
        \[
       \ll T\sum_{\substack{m=nq^2,\\q\leq X^2}}\frac{a_j(n)a_{k-j}(nq^2)}{n}\frac{\log q}{q^2}
        +T\sum_{\substack{n=mq^2,\\q\leq X^2}}\frac{a_j(mq^2)a_{k-j}(m)}{m}\frac{\log q}{q^2}.
        \] 
    In the first term note that $\sum_{q\leq X^2} \tfrac{\log q}{q^2} \ll 1$, and $\displaystyle a_{k-j}(nq^2)\leq (k-j)!$. 
 Hence, this term is
        \[
        \ll (k-j)! T\sum_{n} \frac{a_j(n)}{n}
        \leq k! T\Psi^k.
        \]
  Similarly, the second error term is also $\ll k! T \Psi^k$. 
  Therefore
        \be\label{eq: Landau term of S_j(v)}
        \begin{split}
       S_{j,2}(v)
        =& \,-\frac{T}{2\pi}\sum_{q\leq X^2}\sum_{n=mq}\frac{a_j(n)a_{k-j}(m)}{\sqrt{mn}}\frac{\log q}{q^{1/2+iv}}
        -\frac{T}{2\pi}\sum_{q\leq X^2}\sum_{m=nq}\frac{a_j(n)a_{k-j}(m)}{\sqrt{mn}}\frac{\log q}{q^{1/2-iv}} \\
         &+O\bigs(k! T\Psi^k\bigs).
        \end{split}
        \ee 
    Both the main terms here correspond to the case $\ell=1$,
    so we must have $j=\frac{k-1}{2}$ or $j=\frac{k+1}{2}$.
    In particular, $k$ must be odd.
    Note that when $j=\frac{k-1}{2}$, the first sum vanishes, and when 
    $j=\frac{k+1}{2}$, the second sum vanishes. 
    Furthermore,
        \be\label{conjugates}
        S_{\frac{k-1}{2},2}(v)= \overline{S_{\frac{k+1}{2},2}(v)}.
        \ee
    Thus, it suffices to estimate $S_{\frac{k-1}{2},2}(v)$. We have    
        \[
        \begin{split}
        S_{\frac{k-1}{2},2}(v)
        = -\frac{T}{2\pi}\sum_{q\leq X^2}
        \sum_{n}\frac{a_{(k-1)/2}(n)a_{(k+1)/2}(nq)}{n}\frac{\log q}{q^{1+iv}}
        +O\bigs(k! T\Psi^k\bigs).
        \end{split}
        \]
    We rewrite this as 
        \be\label{eq:Sk-1by2(v)}
        \begin{split}
         S_{\frac{k-1}{2},2}(v)
        =&\, -\frac{T}{2\pi}\sum_{q\leq X^2}
        \Big(\sum_{\substack{n\\ q \nmid n}}+\sum_{\substack{n \\ q \mid n}}\Big)
        \frac{a_{(k-1)/2}(n)a_{(k+1)/2}(nq)}{n}\frac{\log q}{q^{1-iv}}
        +O\bigs(k! T\Psi^k\bigs) \\
        =&\quad T_1+T_2+O\bigs(k! T\Psi^k\bigs).
        \end{split}
        \ee
    Note that if $q\nmid n$, then
        \[
        a_{(k+1)/2}(nq)=\tfrac{(k+1)}{2}a_{(k-1)/2}(n) .
        \]
    Thus
        \be\label{eq:Sk-1by2(v) main term}
        T_1
        =-\frac{(k+1)T}{4\pi}\sum_{q\leq X^2}\frac{\log q}{q^{1-iv}}
        \sum_{\substack{n\\ q\nmid n}} \frac{a_{(k-1)/2}^2(n)}{n}.
        \ee
    We separate the sum over $n$ into sums for which n which does or does not satisfy $a_{(k-1)/2}(n)=\left((k-1)/2\right)!$.
    We then find that
        \be\label{q nmid n}
        \sum_{n: q\nmid n}\frac{a_{(k-1)/2}^2(n)}{n} 
        =\left((k-1)/2\right)!
        \sum_{\substack{q\nmid n \\n \text{ sq-free}}}\frac{a_{(k-1)/2}(n)}{n} 
        +O\bigg(\left((k-1)/2\right)!
        \sum_{\substack{q\nmid n \\ n \text{ not sq-free}}}\frac{a_{(k-1)/2}(n)}{n}\bigg).
        \ee
     By \eqref{not sq-free}, the error term is 
     $O\bigs(k^2\left((k-1)/2\right)! \Psi^{\frac{k-5}{2}}\bigs)$.
     To treat the main term, we set
         \[
         \Psi'=\sum_{\substack{p\leq X^2\\ p\neq q}} \frac 1p,
         \]
      and note that the same analysis that led to \eqref{sq-free}
      shows that
        \be\label{sq free q nmid n}
       \sum_{\substack{q\nmid n \\ n \text{ sq-free}}}\frac{a_j(n)}{n}
       =(\Psi ')^j+O\bigs(j^2(\Psi ')^{j-2}\bigs).
        \ee 
    Now, since $(1-x)^j=1+O(jx)$ for $0\leq x \leq 1$,
    we see that $(\Psi ')^j=(\Psi-1/q)^j=\Psi^j+O\big(j \Psi^{j-1}\big)$.
    Thus,
        \[
        \sum_{\substack{q\nmid n \\ n \text{ sq-free}}}\frac{a_j(n)}{n}
        =\Psi^j+O\big(j^2 \Psi^{j-1}\big). 
        \]
    Combining this and our estimate for the error term in \eqref{q nmid n}, 
    we see that
       \be\label{q nmid n second}
       \sum_{\substack{n \\ q\nmid n}}\frac{a_{(k-1)/2}^2(n)}{n} 
        =\left((k-1)/2\right)!\Psi^{\frac{k-1}{2}}
        +O\bigs(k^2\left((k-1)/2\right)!\Psi^{\frac{k-3}{2}}\bigs).
        \ee
     We will insert this into \eqref{eq:Sk-1by2(v) main term}. 
     To handle the sum over $q$ we use the prime number theorem which, under RH, says that
        \[
        \sum_{q\leq x} \log q =x+O\bigs(x^{\frac12}\log^2{x}\bigs).
        \]
    Partial summation then gives
        \[
        \sum_{q\leq X^2} \frac{\log q}{q^{1-iv}}=\frac{X^{2iv}-2^{iv}}{iv}+O(1).
        \]
    Substituting this and \eqref{q nmid n second} into \eqref{eq:Sk-1by2(v) main term}, we find that
        \be\label{T_1 final}
        T_1
        =-\frac{T}{2\pi}((k+1)/2)! \frac{X^{2iv}-2^{iv}}{iv}
        \Psi^{\frac{k-1}{2}} 
        +O\bigs(k^2((k+1)/2)!T\log X\Psi^{\frac{k-3}{2}}\bigs).
        \ee
    We now proceed to the estimation of the sum $T_2$ in 
    \eqref{eq:Sk-1by2(v)}. 
   Writing $n=qn_1$, we have 
        \[
        T_2
        = -\frac{T}{2\pi} \sum_{q\leq X^2} \frac{\log q}{q^{1-iv}}
        \sum_{n_1} \frac{a_{(k-1)/2}(n_1q)a_{(k+1)/2}(n_1q^2)}{n_1 q},
        \]
     where $n_1$ runs over integers with exactly $(k-3)/2$ prime factors.
     Since  $a_{(k+1)/2}(n_1q^2)\leq \left((k+1)/2\right)!$, it easily follows that
       \[
       T_2 
       \ll T\left((k+1)/2\right)!
        \sum_{q\leq X^2}\frac{\log q}{q}
        \sum_{n_1} \frac{a_{(k-1)/2}(n_1q)}{n_1 q}.
        \]
    Now $a_{(k-1)/2}(n_1q) \leq \frac{k-1}{2} a_{(k-3)/2}(n_1)$. 
    Hence,
        \[
        T_2 
        \ll 
        k \left((k+1)/2\right)! \sum_{q\leq X^2}\frac{\log q}{q^2}
        \sum_{n_1} \frac{a_{(k-3)/2}(n_1)}{n_1}
        \ll 
        k \left((k+1)/2\right)! T \Psi^{\frac{k-3}{2}}.
        \]
    Combining this estimate with \eqref{T_1 final} in \eqref{eq:Sk-1by2(v)},
     we see that 
        \be\label{S_(k-1)/2}
        S_{\frac{k-1}{2},2}(v)
        =-\frac{T}{2\pi}((k+1)/2)!
        \frac{X^{2iv}-2^{iv}}{iv}\Psi^{\frac{k-1}{2}} 
        +O\bigs(k^2((k+1)/2)!T\log X \Psi^{\frac{k-3}{2}}\bigs)
        +O\bigs(k!T\Psi^k\bigs).
        \ee
     By \eqref{eq:S_j(v) rewritten}, the same estimate holds for $S_{\frac{k-1}{2}}(v)$.
     By \eqref{conjugates},
     $ S_{\frac{k+1}{2}}(v)$ and $S_{\frac{k+1}{2},2}(v)$ are both equal to the conjugate of this.
    Moreover, whenever $j\neq \frac{k-1}{2}, \frac{k}{2}$ or $\frac{k+1}{2}$, 
   we have
        \be\label{Sj(v) for most j}
        S_j(v)=O\bigs(k! T \Psi^k\bigs)
        \ee
     by \eqref{eq:S_j(v) rewritten} and \eqref{eq: Landau term of S_j(v)}.    
     
    Now we can complete the proof of the proposition. Recall that
        \be \label{setup sum Sj(v)}
        \sum_{0 < \g \leq T} \bigs(\Re{\CMcal{P}_X(\g+v)}\bigs)^k=\frac{1}{2^k}\sum_{j=0}^{k}{k\choose j}S_j(v). 
        \ee
    If $k$ is even, then by \eqref{eq:Skby2(v)} and \eqref{Sj(v) for most j}
        \be \notag
        \begin{split}
        \sum_{0 < \g \leq T} \bigs(\Re{\CMcal{P}_X(\g+v)}\bigs)^k 
        =&\, \frac{1}{2^k}{k\choose k/2}\Big((k/2)!N(T) \Psi^{\frac k2}
        +O\bigs(k^2(k/2)!N(T)\Psi^{\frac{k-4}{2}} \, \bigs)\Big) \\ 
        &+O\bigg(\sum_{\substack{0\leq j\leq k\\ j\neq \frac{k}{2}}}\frac{1}{2^k}{k\choose j}k! T\Psi^k\bigg) \\
        =&\, \beta_kN(T)\Psi^{\frac k2}
        +O\bigs(k^2\beta_kN(T)\Psi^{\frac{k-4}{2}}\bigs)
        +O\bigs(k! T \Psi^k\bigs).
        \end{split}
        \ee
    If $k$ is odd, then by \eqref{Sj(v) for most j} and \eqref{setup sum Sj(v)}
        \[
        \begin{split}
        \sum_{0 < \g \leq T} \bigs(\Re{\CMcal{P}_X(\g+v)}\bigs)^k 
        =&\, \frac{1}{2^k} \binom{k}{\frac{k-1}{2}} 
        \Big(S_{\frac{k-1}{2}}(v)+S_{\frac{k+1}{2}}(v)\Big)
        +O\bigg(
        \sum_{\substack{0\leq j \leq k, \\j\neq \frac{k-1}{2},\frac{k+1}{2}}}
        \frac{1}{2^k} \binom{k}{j} k! T\Psi^k\bigg) \\
        =&\, \frac{1}{2^k} \binom{k}{\frac{k-1}{2}} 
        \Big(S_{\frac{k-1}{2}}(v)+\overline{S_{\frac{k-1}{2}}(v)}\Big)
        +O\bigs(k!T\Psi^k\bigs).
        \end{split}
        \]
    Using \eqref{S_(k-1)/2} (and the remark immediately after), we obtain
        \[
        \begin{split}
        \sum_{0 < \g \leq T} \bigs(\Re{\CMcal{P}_X(\g+v)}\bigs)^k
        =& -\frac{1}{2^{k}}{k\choose {\frac{k-1}{2}}}((k+1)/2)!
        \frac{T}{\pi}\frac{\sin(2v\log X)-\sin(v\log 2)}{v}\Psi^{\frac{k-1}{2}} \\
        &+O\Big(\frac{1}{2^{k}}{k\choose {\frac{k-1}{2}}}k^2((k+1)/2)!T\log X
        \Psi^{\frac{k-3}{2}}\Big) \\
        =& -\beta_{k+1} \frac{T}{\pi}
        \Big(\sdfrac{\sin(2v\log X)-\sin(v\log 2)}{v}\Big)\Psi^{\frac{k-1}{2}} 
        +O\bigs(k^2\beta_{k+1}T\log X\Psi^{\frac{k-3}{2}}\bigs) \\
        &+O\bigs(k!T\Psi^k\bigs).
        \end{split}
        \]
     This completes the proof.
\end{proof}

\subsection{Other Moment Calculations}
%

In Lemma \ref{Re log zeta}, we saw that the sum
    $
    \Re\CMcal{P}_X(\g+v)
    $
can be used to approximate the function $\log{|\zeta(\r+z)|}-M_X(\r, z),$ where
    \[
    M_X(\r, z)=m(\r+iv)\Big(\log\Big(\sdfrac{eu\log X}{4}\Big)-\sdfrac{u\log X}{4}\Big).
    \]
In the following result, we provide an estimate for the average difference between the function $\log{|\zeta(\r+z)|}-M_X(\r, z)$ and the sum $\Re{\CMcal{P}_X(\g+v)}.$
  
\begin{prop}\label{Re log zeta error}
Let $T^{\frac{\d}{8k}} \leq X\leq T^{\frac{1}{8k}}$ for $0<\d \leq 1.$
Then for a constant $D$ depending on $\d,$ we have
        \[
        \sum_{0 < \g \leq T}
        \bigs(\log{|\zeta(\r+z)|}-M_X(\r, z)-\Re{\CMcal{P}_X(\g+v)}\bigs)^{k}  
        \ll (Dk)^{2k}N(T).
        \]        
\end{prop}

\begin{proof}\let\qed\relax
    By \eqref{eq:Re log zeta} from Lemma \ref{Re log zeta}
        \[
        \log|\zeta(\r+z)|-M(\r, z)
        -\Re\CMcal{P}_X(\g+v) 
        \ll \sum_{i=1}^{4}r_i(X, \g +v),
        \]
     where the $O$-terms are as given in the statement of the lemma.
    We take the $k$th power of each sides of this equation and sum over $0 < \g \leq T.$ Then the right-hand side is
        \be\label{remainder to k}
        4^k \sum_{i=1}^4 \bigg( \sum_{0 < \g \leq T}r_i(X, \g +v)^k\bigg).
        \ee
    Here we have used the elementary inequality
         \be\label{convexity}
         \Big(\frac 1n \sum_{i=1}^n x_i \Big)^k \leq \frac{1}{n} \sum_{i=1}^n x_i^k,
         \ee 
     which is valid for positive numbers $x_1, x_2, \dots, x_n$ and $k\geq 1.$

    The first two error terms in \eqref{remainder to k} can be estimated in a straightforward manner by applying Lemma \ref{Soundmomentlemma3}. 
    For the first one, we note that by the definition of $w_X(n)$ in \eqref{w_X},
        \begin{dmath*}
        1-w_X(p)=
        \begin{cases}
        0 \quad &\text{if}  \quad p\leq X,\\
        \frac{\log(p/X)}{\log X} \quad  &\text{if} \quad  X\leq p\leq X^2.
        \end{cases}
        \end{dmath*}
    Thus
        \be \notag
        \Big|\frac{1-w_X(p)}{p^{iv}}\Big|^2<\frac{\log^2 p}{\log^2 X}\quad  \text{ for } \quad p\leq X^2.
        \ee
    By the Cauchy-Schwarz inequality and then Lemma \ref{Soundmomentlemma3}, 
        \be\label{1-w}
        \begin{split}
        \sum_{0 < \g \leq T}\Big|\sum_{p\leq X^2}
        \sdfrac{1-w_X(p)}{p^{1/2+i(\g+v)}}\Big|^{k} 
        &\leq \sqrt{N(T)}\bigg(\sum_{0 < \g \leq T}\Big|\sum_{p\leq X^2}\sdfrac{1-w_X(p)}{p^{1/2+i(\g+v)}}\Big|^{2k}\bigg)^{\frac12} \\
        &\ll k!N(T)\Big(\sum_{p\leq X^2}\sdfrac{\log^2{p}}{p\log^2{X}}\Big)^k  
        \ll (ck)^kN(T).
        \end{split}
        \ee
    In the last step, we have used the estimate $\sum_{p\leq X^2} \frac{\log^2{p}}{p}\ll \log^2{X}.$
    To estimate the second error term, we need the following:
        \[
        w_X(p^2)=
        \begin{cases}
        0 \quad &\text{if}  \quad p\leq \sqrt{X},\\
        \frac{\log(p^2/X)}{\log X} \quad  &\text{if} \quad  \sqrt{X}\leq p\leq X.
        \end{cases}
        \]
    This implies
        \[
        \Big|\frac{w_X(p^2)}{p^{1/2+i(\g+2v)}}\Big|^2 
        \leq \dfrac{1}{p} \quad  \text{ for } \quad p\leq X.
        \]
    Again, we apply the Cauchy-Schwarz inequality and then 
    Lemma \ref{Soundmomentlemma3}
    to obtain
        \be\label{w}
        \begin{split}
        \sum_{0 < \g \leq T}\Big|\sum_{p\leq X}\sdfrac{w_X(p^2)}{p^{1+2i(\g+v)}}\Big|^{k}
        &\leq \sqrt{N(T)}\bigg(\sum_{0 < \g \leq T}\Big|\sum_{p\leq X}\sdfrac{w_X(p^2)}{p^{1+2i(\g+v)}}\Big|^{2k}\bigg)^{\frac12} \\
        &\ll (ck)^kN(T)\Big(\sum_{p\leq X}\sdfrac{1}{p^{2}}\Big)^k
        \ll (ck)^kN(T).
        \end{split}
        \ee

    The third error term is handled in Proposition \ref{error term integral},
    and the fourth error term is estimated in Proposition \ref{zero spacing eta v} 
    under the assumption of Montgomery's Pair Correlation Conjecture.
 \end{proof}
     

\begin{prop}\label{error term integral}
    Assume RH and let $\displaystyle X\leq T^{\frac{1}{8k}}.$ Then
        \[
        \sum_{0 < \g \leq T}\bigg(\frac{1}{\log X} \int_{1/2}^\infty X^{\tfrac12-\s}
        \bigg|\sum_{p\leq X^2}\frac{\Lambda_X(p)\log{(Xp)}}{p^{\s+i(\g+v)}}\bigg|\mathop{d\s}\bigg)^{k}
        \ll (ck)^{k/2}N(T).
        \]
\end{prop}

\begin{proof}
    By the Cauchy-Schwarz inequality, 
        \be\label{eq:integral CS}
        \begin{split}
        \sum_{0 < \g \leq T}\bigg(\frac{1}{\log X} \int_{1/2}^\infty X^{\tfrac12-\s}&\Big|\sum_{p\leq X^2}\frac{\Lambda_X(p)\log{(Xp)}}{p^{\s+i(\g+v)}}\Big|\mathop{d\s}\bigg)^{k} \\
        \leq \frac{\sqrt{N(T)}}{(\log X)^k}&
        \bigg(\sum_{0 < \g \leq T}\bigg(\int_{1/2}^\infty X^{\tfrac12-\s}\Big|\sum_{p\leq X^2}\frac{\Lambda_X(p)\log{(Xp)}}{p^{\s+i(\g+v)}}\Big|\mathop{d\s}\bigg)^{2k}\bigg)^{1/2}.
        \end{split}
        \ee
   By H\"{o}lder's inequality,  
        \be
        \begin{split}\label{eq:integral}
        \bigg(\int_{1/2}^\infty X^{\tfrac12-\s}&\Big|\sum_{p\leq X^2}\frac{\Lambda_X(p)\log{(Xp)}}{p^{\s+i(\g+v)}}\Big|\mathop{d\s}\bigg)^{2k}  \\
        &\leq \Big( \int_{1/2}^\infty X^{\tfrac12-\s}\mathop{d\s}\Big)^{2k-1} 
        \int_{1/2}^\infty X^{\tfrac12-\s}\Big|\sum_{p\leq X^2} \frac{\Lambda_X(p)\log{(Xp)}}{p^{\s+i(\g+v)}}\Big|^{2k}\mathop{d\s} \\
        & = \sdfrac{1}{(\log X)^{2k-1}}\int_{1/2}^\infty X^{\tfrac12-\s}\Big|\sum_{p\leq X^2} 
        \frac{\Lambda_X(p)\log{(Xp)}}{p^{\s+i(\g+v)}}\Big|^{2k}\mathop{d\s}.
        \end{split}
        \ee
    Here by Lemma \ref{Soundmomentlemma3}, 
        \be\label{eq:sum prop}
        \sum_{0 < \g \leq T} \Big|\sum_{p\leq X^2}\frac{\Lambda_X(p)\log{(Xp)}}{p^{\s+i(\g+v)}}\Big|^{2k} 
        \ll k!N(T)\Big(\sum_{p\leq X^2}\frac{\Lambda_X^2(p)\log^2{(Xp)}}{p^{2\s}}\Big)^k. 
        \ee
    Since  $\Lambda_X(p)\leq \log p $ and $p\leq X^2,$ we find that
        \[
        \Lambda_X^2(p)\log^2(Xp) \leq \log^2{p}(\log X+\log p)^2 
        \leq 9\log^2{p} \log^2{X}.
        \]    
    Thus
        \[
       \Big(\sum_{p\leq X^2}\sdfrac{\Lambda_X^2(p)\log^2{(Xp)}}{p^{2\s}}\Big)^k
        \ll c^k (\log X)^{2k} \Big(\sum_{p\leq X^2} \frac{\log^2 p}{p^{2\s}}\Big)^k.
        \]
    By Mertens' theorem,
        $
        \sum_{p\leq X^2} \frac{\log^2 p}{p^{2\s}} \ll \log^2{X}
        $
     for any fixed $\s\geq \frac12.$ 
    Thus, the right-hand side of \eqref{eq:sum prop} is at most
    $\ll (ck)^k N(T) (\log X)^{4k}.$
    Inserting this bound into \eqref{eq:integral}, we find that
        \[
        \begin{split}
        \bigg(\int_{1/2}^\infty X^{\tfrac12-\s}&\Big|\sum_{p\leq X^2}\frac{\Lambda_X(p)\log{(Xp)}}{p^{\s+i(\g+v)}}\Big|\mathop{d\s}\bigg)^{2k}  \\
        \ll &\, (ck)^k N(T) (\log X)^{2k+1} \int_{1/2}^\infty X^{\tfrac12-\s}\mathop{d\s}   
        \ll (ck)^k N(T) (\log X)^{2k}.
        \end{split}
        \]
  The claim of the proposition follows from this and \eqref{eq:integral CS}.
\end{proof}


\begin{prop}\label{zero spacing eta v}
   Let $T^{\frac{\d}{8k}} \leq X\leq T^{\frac{1}{8k}}$ for $0<\d \leq 1.$ 
   If Montgomery's Pair Correlation Conjecture is true, then for a constant $D$ depending on $\d$
        \[
        \sum_{0 < \g \leq T}\bigg(\Big( 1+\log^+ \sdfrac{1}{\eta_{\g+v}\log{X}}\Big)\frac{E(X,\g+v)}{\log{X}}\bigg)^k
        \ll (Dk)^{2k}N(T). 
        \]    
Here $\eta_{\g+v}$ is as defined by \eqref{eta}, and $E(X, t)$ is as defined in \eqref{E}. 
\end{prop}

\begin{proof}
    By the Cauchy-Schwarz inequality,
        \be \label{eq:Cauchy Schwarz on zero spacing}
        \begin{split}
        \sum_{0 < \g \leq T}\bigg(\Big( 1+&\log^+ \sdfrac{1}{\eta_{\g+v}\log{X}}\Big) \sdfrac{E(X,\g+v)}{\log X}\bigg)^k \\
        &\leq \frac{1}{(\log X)^{k}}\Big(\sum_{0 < \g \leq T}\Big( 1+\log^+ \sdfrac{1}{\eta_{\g+v}\log{X}} \Big)^{2k}\Big)^{1/2}\Big(\sum_{0 < \g \leq T} \big| E(X,\g+v)\big|^{2k}\Big)^{1/2}.
        \end{split}
        \ee
    First consider the sum
        \[
        \sum_{0 < \g \leq T} \big|E(X,\g+v)\big|^{2k}
        =\sum_{0 < \g \leq T}\Big| \sum_{n\leq X^2}\frac{\Lambda_X(n)}{n^{\s_1+i(\g+v)}}+\log{(\g+v)}\Big|^{2k}.
        \]
    We separate the sum over $n$ into primes and higher powers of primes. Then by \eqref{convexity},
        \be \label{eq:moment of E}
        \begin{split}
        &\sum_{0 < \g \leq T}\Big| \sum_{n\leq X^2}\frac{\Lambda_X(n)}{n^{\s_1+i(\g+v)}}+\log{(\g+v)}\Big|^{2k} \\
        \leq  9^k &\sum_{0 < \g \leq T}\Big| \sum_{p\leq X^2}\frac{\Lambda_X(p)}{p^{\s_1+i(\g+v)}}\Big|^{2k} 
        +9^k\sum_{0 < \g \leq T}\Big| \sum_{\substack{p^\ell \leq X^2, \\ \ell\geq 2}}\frac{\Lambda_X(p^\ell)}{p^{\ell\s_1+i \ell(\g+v)}}\Big|^{2k}
        +9^k\sum_{0 < \g \leq T} (\log{(\g+v)})^{2k}. 
        \end{split}
        \ee
    Since $\s_1=\frac12 +\frac4{\log X},$ it is easy to see that the second sum on the right is $O\bigs((c\log X)^{2k}N(T)\bigs).$
    By Lemma \ref{Soundmomentlemma3},
        \[
        \sum_{0 < \g \leq T}\Big| \sum_{p\leq X^2}\frac{\Lambda_X(p)}{p^{\s_1+i(\g+v)}}\Big|^{2k}
         \ll k! N(T) \Big(\sum_{p\leq X^2} \frac{\Lambda_X(p)^2}{p^{2\s_1}}\Big)^k \ll (ck)^k(\log X)^{2k}N(T).
        \]
   Furthermore, we trivially have
        \[
        \sum_{0 < \g \leq T} (\log{(\g+v)})^{2k} \ll N(T)(\log T)^{2k}.
        \]
    From these estimates and \eqref{eq:moment of E}, we obtain
        \[
        \sum_{0 < \g \leq T} \big|E(X,\g+v)\big|^{2k}
        \ll  c^kN(T)\big(k^k(\log X)^{2k}+(\log T)^{2k}\big).
        \]
     Since $X\geq T^{\d/(8k)},$ this gives 
        \be\label{eq:moment of E 2}
        \sum_{0 < \g \leq T} \big|E(X,\g+v)\big|^{2k}\ll (Dk)^{2k} N(T)  (\log X)^{2k}
        \ee
     for a constant $D$ depending on $\d.$ Next, we treat the expression 
        \be\label{sum log eta}
        \sum_{0 < \g \leq T}\Big( 1+\log^+ \sdfrac{1}{\eta_{\g+v}\log{X}} \Big)^{2k}
        \ee
    from \eqref{eq:Cauchy Schwarz on zero spacing}. Recall that
        \[
        \eta_{\g+v}=\min_{\g'\neq \g+v}|\g'-(\g+v)|.
        \]
   Let $\g \in (0, T].$
   If  $\displaystyle \eta_{\g+v}> \sdfrac{1}{\log X},$ then the contribution of
   $\g$ to the sum is $1.$ If   $\displaystyle \eta_{\g+v}\leq \sdfrac{1}{\log X},$ then there exists a nonnegative integer $j$ such that
        \be\label{eta inequality}
         \frac{e^{-j-1}}{\log X} < \eta_{\g+v} \leq \frac{e^{-j}}{\log X}.
        \ee
    By the inequality on the right-hand side, for some $\g'$ in 
    $\Big(-|v|-\tfrac{1}{\log X}, T+|v|+\tfrac{1}{\log X}\Big]$ other than $\g$ 
        \[
        -v-\frac{e^{-j}}{\log X} \leq \g-\g' \leq -v+\frac{e^{-j}}{\log X}. 
        \]
    By Montgomery's Pair Correlation Conjecture, the number of such ordinates $\g$ and $\g'$ is 
        \[
        \ll N(T)\int_{(-v-e^{-j}/\log X)\log T/(2\pi)}^{(-v+e^{-j}/\log X)\log T/(2\pi)} \left(1-\frac{\sin^2(\pi x)}{(\pi x)^2}\right) \mathop{dx}.
        \]
  The integrand is nonnegative and at most $1,$ so the above is
        \[
         \ll \frac{e^{-j}\log T}{\log X} N(T)
         \ll \frac{k}{\d}e^{-j} N(T) 
        \]
   since $X \geq T^{\d/(8k)}.$ For those $\g$ we have
         \[
        1+\log^+\sdfrac{1}{\eta_{\g+v}\log{X}} < j+2 . 
        \]   
    Thus, the sum in \eqref{sum log eta} is
        \be\label{eq:proof eta}
        \ll \frac{k}{\d} N(T) \sum_{j=0}^\infty \frac{(j+2)^{2k}}{e^j}.
        \ee
    The series is
        \[
        \sum_{j=0}^\infty \frac{(j+2)^{2k}}{e^j} 
        = e^2\sum_{j=2}^\infty \frac{j^{2k}}{e^j} 
        \leq e^2 \int_2^\infty \frac{x^{2k}}{e^x} \mathop{dx}
        \ll \Gamma(2k)
        \ll (2k-1)!. 
        \]
     Hence,
         \[
         \sum_{0 < \g \leq T}\Big( 1+\log^+ \sdfrac{1}{\eta_{\g+v}\log{X}} \Big)^{2k} 
         \ll (Dk)^{2k} N(T)
         \] 
     for a constant $D=D(\d).$ Combining this bound with \eqref{eq:moment of E 2} in \eqref{eq:Cauchy Schwarz on zero spacing}, we complete the proof.
\end{proof}


\subsection{Proof of Theorem \ref{moments of Re log zeta}}
We can now complete the proof of Theorem \ref{moments of Re log zeta}.
Recall that $\displaystyle T^{\tfrac{\d}{8k}}\leq X\leq T^{\tfrac{1}{8k}}$ for some $0<\d \leq 1$.
We write
    \[
    \log{|\zeta(\r+z)|}-M_X(\r, z)  
    = \Re\CMcal{P}_X(\g+v)+ r(X,\g+v).
    \]
Taking the $k$th moment of each side, we obtain
    \be\label{log-M moments}
    \begin{split}
    \sum_{0<\g \leq T} \bigs(\log{|\zeta(\r+z)|}-M_X(\r, z)\bigs)^k  
    &=\sum_{0<\g \leq T} \bigs(\Re\CMcal{P}_X(\g+v)\bigs)^k 
    +\sum_{0 < \g \leq T} \bigs(r(X,\g+v)\bigs)^k
    \\ 
    &+O \bigg( \sum_{j=1}^{k-1}\binom{k}{j} 
     \sum_{0 < \g \leq T}\bigs|\Re{\CMcal{P}_X(\g+v)}\bigs|^j
     \bigs|r(X,\g+v)\bigs|^{k-j}\bigg). \notag 
   \end{split}
   \ee
We write the right-hand side as
     \[
     \sum_{0<\g \leq T} \bigs(\Re\CMcal{P}_X(\g+v)\bigs)^k +A_1+A_2. 
     \]  
By Proposition \ref{Re log zeta error} ,
   \be\label{A 1 bound}
    A_1 \ll (Dk)^{2k} N(T), 
    \ee
 where $D$ depends on $\d$.
To estimate each term in $A_2$, we use the Cauchy-Schwarz inequality to find that
    \begin{align*}
    \sum_{0 < \g \leq T}
    \bigs|\Re\CMcal{P}_X&(\g+v)\bigs|^j
    \bigs|r(X,\g+v)\bigs|^{k-j}  \\
    &\leq \Big(\sum_{0 < \g \leq T} \bigs|\Re{\CMcal{P}_X(\g+v)}\bigs|^{2j}\Big)^{\frac{1}{2}} 
     \Big(\sum_{0 < \g \leq T} \bigs|r(X,\g+v)\bigs|^{2k-2j}\Big)^{\frac{1}{2}}.
    \end{align*}
  By Propositions \ref{moments of Re Dirichlet polyl v} and \ref{Re log zeta error}, and the hypothesis that $k \ll \log\log\log T$, the right-hand side is
    \[
    \ll  \Big(\beta_{2j}N(T)\Psi^j \Big)^{\frac12}\Big(\bigs(D(k-j)\bigs)^{4k-4j}N(T)\Big)^{\frac12}
    \ll  \beta_{2j}^{1/2}(Dk)^{2(k-j)}N(T)\Psi^{j/2},
    \]
 where $\beta_{2j}$ is as defined in \eqref{beta} and $D=D(\d)$. Thus
    \[
    A_2 
    \ll N(T) \sum_{j=1}^{k-1} \binom{k}{j} \beta_{2j}^{1/2} (Dk)^{2(k-j)} \Psi^{j/2}.
    \]
 It can be easily seen by Stirling's approximation that
    \be\label{beta asymptotic}
    \beta_{2j} \sim  c\Big(\frac{j}{e}\Big)^j,
    \ee
so $\displaystyle \beta_{2j}^{1/2} \ll j^{j/2} < k^{j/2}$. Using this, we obtain
    \[
    A_2 \ll (Dk)^{2k} N(T) \sum_{j=1}^{k-1} \binom{k}{j}  \Psi^{j/2}.
    \]
Here, the right-hand side equals   
    \[
    (Dk)^{2k} N(T) \Big\{ \Psi^{k/2}\bigs(1+1/\sqrt{\Psi}\,\bigs)^k-\Psi^{k/2}-1\Big\}.
    \]   
By the mean value theorem of differential calculus,
$(1+x)^k=1+O(k2^kx) $ for $ 0\leq x\leq 1$. 
Hence for a constant $D=D(\d)$,
    \[
    A_2 \ll k(Dk)^{2k} N(T) \Psi^{\tfrac{k-1}{2}}. 
    \]
Combining this with \eqref{log-M moments} and \eqref{A 1 bound}, we find that for $D=D(\d)$
     \begin{align*}
    &\sum_{0<\g \leq T}\bigs(\log{|\zeta(\r+z)|}-M_X(\r, z)\bigs)^k   \\
    =&\, \sum_{0<\g \leq T} \bigs(\Re\CMcal{P}_X(\g+v)\bigs)^k 
    +O\bigs((Dk)^{2k}N(T)\bigs) 
    +O\Big(k(Dk)^{2k}N(T)\Psi^{\tfrac{k-1}{2}}\Big)  \\
    =& \sum_{0<\g \leq T} \bigs(\Re\CMcal{P}_X(\g+v)\bigs)^k
    +O\Big(k(Dk)^{2k}N(T)\Psi^{\tfrac{k-1}{2}}\Big).
    \end{align*}
Now suppose that $k$ is even. Then by Proposition \ref{moments of Re Dirichlet polyl v}, the right-hand side is
    \[
    \beta_k N(T)\Psi^{\tfrac k2}
    +O\Big(k^2\beta_k N(T)\Psi^{\tfrac{k-4}{2}}\Big)
    +O\Big(k(Dk)^{2k}N(T)\Psi^{\tfrac{k-1}{2}}\Big). 
    \]
By \eqref{beta asymptotic}, the second $O$-term may be replaced by
    \[
    O\Big(k\beta_k(Dk)^{\tfrac{3k}{2}}N(T)\Psi^{\tfrac{k-1}{2}}\Big),
    \]
which is larger than the first $O$-term in the above sum. Thus, we obtain
     \[
    \sum_{0<\g \leq T} \bigs(\log{|\zeta(\r+z)|}-M_X(\r, z)\bigs)^k 
    =\beta_k N(T)\Psi^{\tfrac k2}
    + O\Big(D^k k^{\tfrac{3k+2}{2}}\beta_k N(T)\Psi^{\tfrac{k-1}{2}}\Big)
    \]
for even $k$. Arguing similarly when $k$ is odd, we find that
    \begin{align*}
     \sum_{0<\g \leq T} \bigs(\log|\zeta(\r+z)|-M_X(\r, z)\bigs)^k  
     &= -\frac{\beta_{k+1}}{\pi}\frac{\sin(2v\log X)-\sin(v\log 2)}{v}
     T\Psi^{\tfrac{k-1}{2}} 
      \\
     &\quad+O\Big(D^k k^{\tfrac{3k+1}{2}}\beta_{k+1} N(T)\Psi^{\tfrac{k-1}{2}}\Big).
     \end{align*}
     Since 
     \[
     \frac{\sin(2v\log X)-\sin(v\log 2)}{v}
     \ll \log X,
     \]
     the above $k$th moment is
     \[
     \ll D^k k^{\tfrac{3k+1}{2}}\beta_{k+1} N(T)\Psi^{\tfrac{k-1}{2}}.
     \]
This completes the proof of Theorem \ref{moments of Re log zeta}.

    
\section{Proof of Theorem \ref{distr of Re log zeta}}\label{sec:proof}
It follows from Proposition \ref{moments of Re Dirichlet polyl v} that
    \[
    \frac{1}{N(T)}\sum_{0<\g\leq T}\mathbbm{1}_{[a,b]}
    \bigg(\frac{\Re\CMcal{P}_X(\g+v)}{\sqrt{\Psi/2}}\bigg)
    =  \frac{1}{\sqrt{2\pi}}\int_a^b e^{-x^2/2}\mathop{dx}+o(1).
    \]
We obtain a precise error term for this statement in Section \ref{subsection distr of P}. In order to do this, we introduce a random polynomial whose real part has the same Gaussian distribution as the asymptotic distribution of $\Re\CMcal{P}_X(\g+v).$


\subsection{A Random Model for $\Re\CMcal{P}_X(\g+v)$}\label{random model}

Suppose that for each prime $p,$ $\theta_p$ denotes a random variable that is uniformly distributed over $[0, 1],$ and $(\theta_p)_p$ is a sequence of independent and identically distributed random variables. 
We define 
    \[
    \CMcal{P}_v(\underline{\theta}):= \sum_{p\leq X^2}\frac{e^{2\pi i \theta_{p}}}{p^{1/2+iv}},
    \] 
similarly to the polynomial $\CMcal{P}_X(\g+v).$ For now, we take $\displaystyle X\leq T^{\frac{1}{8k}}.$ In order to understand the distribution of 
 $\Re\CMcal{P}_v(\underline{\theta}),$
we first observe that
    \[
    \int_0^1 e^{2\pi i x\theta_p }\mathop{d\theta_p}=
    \begin{cases}
    1 &\text{if} \quad x=0,\\
    0 & \text {for other real } x. 
    \end{cases}
    \]
We also define $\theta_n$ for positive integers $n$ that are not primes. 
For $n >1,$ let $n$ have the prime factorization 
$\displaystyle n=p_1^{\a_1}\dots p_r^{\a_r}.$ Then we set
    \[
    \theta_n:=\a_1\theta_{p_1}+\dots+\a_r\theta_{p_r}.
    \]    
It then follows that
    \[
    \theta_{mn}=\theta_m+\theta_n \quad  \text{for any numbers } \mkern9mu m, n \quad \text{and} \quad \theta_m=\theta_n \quad \text{if and only if} \quad m=n. 
    \]
Furthermore, the last assertion implies that
    \be\label{orthogonality}
    \int_0^1 e^{2\pi i\theta_m} \overline{e^{2\pi i\theta_n}} \mathop{d\underline{\theta}} = 
    \begin{cases}
    1 &\text{if} \quad m=n, \\
    0 &\text{if} \quad m\neq n.
    \end{cases}
    \ee
Here and from now on, $\displaystyle \int_0^1(\dots)\mathop{d\underline{\theta}} $ represents the multidimensional integral $\int_0^1\dots \int_0^1(\dots)  \\ \prod_{p\leq X^2}\mathop{d\theta_p}.$
A consequence of this and the identity $\displaystyle \Re z=\tfrac{z+\overline{z}}{2}$ is
    \[
    \int_0^1 \bigs(\Re\CMcal{P}_v(\theta)\bigs)^k\mathop{d\underline{\theta}}=
    \begin{cases} 
    \beta_k \Psi^{\frac k2} &\text{if} \quad k \text{ is even,}\\
    0 &\text{if} \quad k \text{ is odd}. \\
    \end{cases}
    \]
The coefficients $\beta_k$ are as defined in \eqref{beta} and
$\Psi=\sum_{p\leq X^2} 1/p.$
This means that the polynomial $\displaystyle \Re\CMcal{P}_v(\theta)$ has a Gaussian distribution with mean $0$
and variance $\tfrac{1}{2}\Psi.$ 

Our goal is to understand the distribution of $\displaystyle \Re\CMcal{P}_X(\g+v)$
in relation to the distribution of the random polynomial $\displaystyle \Re\CMcal{P}_v(\theta).$ 
In order to do this, we compare the Fourier transform of $\displaystyle \Re\CMcal{P}_X(\g+v)$ with that of $\displaystyle \Re\CMcal{P}_v(\theta).$

First, we need to express the moments of $\displaystyle \Re\CMcal{P}_X(\g+v)$ in terms of the moments of $\displaystyle \Re\CMcal{P}_v(\theta).$


\begin{lem}\label{lemma 3.4}
    Let $\displaystyle X\leq T^{\frac{1}{8k}}$ where $k\ll \log\log T.$ Then
        \be\label{eq:lemma 3.4 real} 
        \begin{split}
        &\sum_{0 < \g \leq T} \bigs(\Re{\CMcal{P}_X(\g+v)}\bigs)^k  \\
        =& \,N(T)\int_0^1\bigs(\Re{\CMcal{P}_v(\underline{\theta})}\bigs)^k\mathop{d{\underline{\theta}}} 
        -\frac{T}{\pi}\sum_{\substack{q\leq X^2,\\ 1\leq\ell\leq k}}\frac{\log q}{q^{\ell/2}}\int_0^1\bigs(\Re{\CMcal{P}_v(\underline{\theta})}\bigs)^k
        \Re e^{2\pi i\ell\theta_q}\mathop{d{\underline{\theta}}} 
        +O\bigs((ck)^k\sqrt{T}\log^2{T}\bigs).
        \end{split}
        \ee
\end{lem}

\begin{proof}
        By the identity $ \Re z=\sdfrac{z+\overline{z}}{2},$
        \[
        \sum_{0 < \g \leq T}\bigs(\Re{\CMcal{P}_X(\g+v)}\bigs)^k
        =\frac{1}{2^k}\sum_{j=0}^k \binom{k}{j}
        \sum_{0 < \g \leq T}\CMcal{P}_X(\g+v)^j \overline{\CMcal{P}_X(\g+v)}^{k-j}.
        \]
    Thus it suffices to compute the sum over $\g$ in the above for each $j.$ Suppose that $j$ is fixed. By definition, 
        \[
        \begin{split}
        \sum_{0 < \g \leq T}&\CMcal{P}_X(\g+v)^j \overline{\CMcal{P}_X(\g+v)}^{k-j} \\
        &=\sum_{0 < \g \leq T}\Big(\sum_{p\leq X^2}\sdfrac{1}{p^{1/2+i(\g+v)}}\Big)^j \Big(\overline{\sum_{p\leq X^2}\sdfrac{1}{p^{1/2+i(\g+v)}}}\Big)^{k-j}.
        \end{split}
        \]
     By Corollary \ref{cor:moments}, this is 
         \be\label{eq:proof of lemma 3.4}
        \begin{split}
        &N(T)\sum_{n} \frac{a_j(n)}{\sqrt n} \frac{a_{k-j}(n)}{\sqrt n}  \\
        &-\frac{T}{2\pi}\sum_{n}\sum_{m} 
        \frac{a_j(n)}{\sqrt n} \frac{a_{k-j}(m)}{\sqrt m}\Big(\frac{m}{n}\Big)^{iv}
        \bigg\{\frac{\Lambda(m/n)}{\sqrt{m/n}}+\frac{\Lambda(n/m)}{\sqrt{n/m}}\bigg\} \\
         &+O\bigg(\log T \log\log T\Big(\sum_n\sdfrac{a_j(n)}{n}\sum_{n<m}a_{k-j}(m)
        +\sum_m\sdfrac{a_{k-j}(m)}{m}\sum_{m<n} a_j(n)\Big)\bigg) \\
        &+O\bigg(X^{2k}\log^2{T}\Big(\sum_n\sdfrac{a_j(n)^2}{n}+\sum_m\sdfrac{a_{k-j}(m)^2}{m}\Big)\bigg),
        \end{split}
        \ee
where $a_{r}(p_1\dots p_{r})$ denotes the number of permutations of the primes $p_1,\dots, p_{r}.$ Also, in the above equation and throughout this proof,
   we suppose that $n$ is always a product of $j$ primes while $m$ is always a product of $k-j$ primes, where all these primes are of size at most $X^2.$ 
     By \eqref{orthogonality}, the first main term can be written as
     \[
     N(T)\int_0^1\Big(\sum_n\frac{a_j(n)}{n^{1/2+iv}}e^{2\pi i\theta_n}\Big)
        \Big(\overline{\sum_{m}\frac{a_{k-j}(m)}{m^{1/2+iv}}e^{2\pi i\theta_m}}\Big)\mathop{d{\underline{\theta}}}. 
     \]
     We can write this as
       \[
        N(T)\int_0^1\CMcal{P}_v(\underline{\theta})^j \overline{\CMcal{P}_v(\underline{\theta})}^{k-j}\mathop{d{\underline{\theta}}}.
         \]
   Then we see that the first main term in \eqref{eq:lemma 3.4 real} follows from summing the above term over $0\leq j \leq k$ with the binomial coefficients.    
   Again by \eqref{orthogonality}, the second main term in \eqref{eq:proof of lemma 3.4} can be written as
         \[
        \begin{split}
         &\\
        &-\frac{T}{2\pi}
        \int_0^1\Big(\sum_n\sdfrac{a_j(n)}{n^{1/2+iv}}e^{2\pi i\theta_n}\Big)
        \Big(\overline{\sum_m\sdfrac{a_{k-j}(m)}{m^{1/2+iv}}e^{2\pi i\theta_m}}\Big)
          \Big(\sum_{\substack{q\leq X^2,\\1\leq\ell\leq k}}\sdfrac{\log q}{q^{\ell/2}}e^{2\pi i\ell\theta_{q}}\Big)\mathop{d{\underline{\theta}}} \\
        &-\frac{T}{2\pi}\int_0^1\Big(\sum_n\sdfrac{a_j(n)}{n^{1/2+iv}}e^{2\pi i\theta_n}\Big)
        \Big(\overline{\sum_m\sdfrac{a_{k-j}(m)}{m^{1/2+iv}}e^{2\pi i\theta_m}} \Big)
           \Big(\sum_{\substack{q\leq X^2,\\ 1\leq\ell\leq k}}
           \sdfrac{\log q}{q^{\ell/2}} e^{-2\pi i\ell\theta_q} \Big)\mathop{d{\underline{\theta}}}
        \end{split}
        \]
We write the first integral as below 
        \[
        {m_j}'
        := \sum_{\substack{q\leq X^2,\\ 1\leq\ell\leq k}}\frac{\log q}{q^{\ell/2}}
    \int_{0}^{1} {\CMcal{P}_v(\underline{\theta})}^j 
    \overline{\CMcal{P}_v(\underline{\theta})}^{k-j}
         e^{2\pi i\ell\theta_q}\mathop{d{\underline{\theta}}}.
         \]
    Similarly, we write the second integral as
        \[
        {m_j}{''}
        :=\sum_{\substack{q\leq X^2,\\ 1\leq\ell\leq k}}\frac{\log q}{q^{\ell/2}} 
        \int_{0}^{1} \CMcal{P}_v(\underline{\theta})^j 
        \overline{\CMcal{P}_v(\underline{\theta})}^{k-j}
         e^{-2\pi i\ell\theta_q}\mathop{d{\underline{\theta}}}.
        \] 
    We sum these over $j$ with the binomial coefficients and obtain
        \[
        \frac{1}{2^k}\sum_{j=0}^k \binom{k}{j}\bigs( {m_j}'+{m_j}{''} \bigs)
        =2\sum_{\substack{q\leq X^2,\\ 1\leq\ell\leq k}} \sdfrac{\log q}{q^{\ell/2}}
         \int_0^1 \bigs(\Re{\CMcal{P}_v(\underline{\theta})}\bigs)^k \Re e^{2\pi i\ell\theta_q}\mathop{d{\underline{\theta}}}.
        \]
    Thus the second main term of \eqref{eq:lemma 3.4 real} follows.
    
    It is left to bound the error terms in \eqref{eq:proof of lemma 3.4}. By \eqref{estimate1} and \eqref{estimate2}, both of the error terms are $\ll (ck)^k \Psi^k \sqrt[4]{T}\log^2{T}.$ Then the total contribution from $\displaystyle j=0, 1, \dots, k$ is      
        \[
        \ll \frac{1}{2^k}\sum_{j=0}^{k}{{k}\choose{j}} (ck)^k \Psi^k \sqrt[4]{T} \log^2{T} 
        \ll (ck)^k\sqrt{T}\log^2{T} .
        \]
This completes the proof.
\end{proof}

The next two lemmas will be needed in the next section. 


\begin{lem}\label{our lemma 3.16}
    For $\displaystyle X\leq T^{\frac{1}{8k}},$
        \[
        \sum_{0 < \g \leq T}\bigs|\Re{\CMcal{P}_X(\g+v)}\bigs|^k
         \ll (ck\Psi)^{k/2}N(T). 
        \]
\end{lem}

\begin{proof}
    We have
        \[
        \sum_{0 < \g \leq T}\bigs|\Re{\CMcal{P}_X(\g+v)}\bigs|^k 
        \leq \sqrt{N(T)}\Big(\sum_{0 < \g \leq T}\bigs|\CMcal{P}_X(\g+v)\bigs|^{2k}\Big)^{1/2}
        \]
    by the Cauchy-Schwarz inequality.
    Then the result easily follows by Lemma \ref{Soundmomentlemma3},
    which implies that
        \[
        \sum_{0 < \g \leq T}\Big|\sum_{p\leq X^2}\frac{1}{p^{1/2+i(\g+v)}}\Big|^{2k}
        \ll k!N(T)\Big(\sum_{p\leq X^2}\sdfrac{1}{p}\Big)^k
        =k!\Psi^k N(T).
        \]
\end{proof}


\begin{lem}\label{Tsang lemma 3.4}
    Suppose $k$ is a nonnegative integer. For $\displaystyle 2\leq X \leq T^{\frac{1}{2k}} ,$ we have
        \[
        \int_0^1\bigs|\Re{\CMcal{P}_v(\underline{\theta}})\bigs|^k \mathop{d{\underline{\theta}}}
        \ll (ck\Psi)^{k/2}.
        \]
\end{lem}

\begin{proof}
For $v=0,$ this is (3.10) in \cite{Tsang}. For $v\neq 0,$ the result can be proven in a similar way by using \eqref{orthogonality}. 
\end{proof}


\subsection{The Fourier Transform of $\Re\CMcal{P}_X(\g+v)$}

We use the polynomial approximation 
    \be\label{eq:exponential}
    e^{ix}=\sum_{0\leq k<K}\frac{(ix)^k}{k!}+O\bigg(\frac{|x|^K}{K!}\bigg) \quad \text{for} \quad x\in\mathbb{R}, 
    \ee
to obtain an approximation to the Fourier transform of the polynomial $\Re\CMcal{P}_X(\g+v).$ We define the following parameters
    \be\label{Omega K}
    \Omega=\Psi(T)^2, \quad  K=2\lfloor \Psi(T)^6\rfloor, 
\quad X \leq T^{\tfrac{1}{16\Psi(T)^6}},
    \ee
where 
    \[
    \Psi(T)=\sum_{p\leq T} \frac{1}{p}.
    \]    
The following lemma relates the Fourier transform of $\Re\CMcal{P}_X(\g+v)$  to the Fourier transform of $\Re\CMcal{P}_v(\theta).$


\begin{lem}\label{fourier}
        Let $\displaystyle X  \leq T^{\tfrac{1}{16\Psi(T)^6}}.$ Then for $0\leq \omega \leq \Omega,$
        \[
        \begin{split}
        \sum_{0 < \g \leq T}&\exp\bigs(2\pi i \omega\Re{\CMcal{P}_X(\g+v)}\bigs) \\
        &=\,N(T)\int_0^1 \exp
        \bigs(2\pi i \omega \Re{\CMcal{P}_v(\underline{\theta})}\bigs)\mathop{d\underline{\theta}}
        -\frac{T}{\pi}\sum_{\substack{q\leq X^2,\\ 1\leq\ell\leq K}}
        \frac{\log q}{q^{\ell/2}}\int_0^1\exp\bigs(2\pi i\omega\Re{\CMcal{P}_v(\underline{\theta})}\bigs)\Re e^{2\pi i\ell\theta_q}\mathop{d\underline{\theta}} \\
        &\quad+O\bigg(N(T) \frac{\omega}{2^K}\bigg).
        \end{split}
        \]
\end{lem}

\begin{proof}    
    By \eqref{eq:exponential}, 
        \be\label{eq:exponential expansion}
        \begin{split}
        \sum_{0 < \g \leq T} &\exp\bigs(2\pi i \omega\Re{\CMcal{P}_X(\g+v)}\bigs) \\
        &=\sum_{0\leq k< K}\frac{(2\pi i \omega)^k}{k!}
        \sum_{0 < \g \leq T}\bigs(\Re{\CMcal{P}_X(\g+v)}\bigs)^k
        +O\bigg(\sdfrac{(2\pi \omega)^K}{K!}\sum_{0 < \g \leq T}
        \bigs|\Re{\CMcal{P}_X(\g+v)}\bigs|^K\bigg).
        \end{split}
        \ee
 By Lemma \ref{our lemma 3.16}, the $O$-term is 
        \[
        \begin{split} 
        &\ll N(T)\frac{(2\pi\omega)^K}{K!}(cK\Psi)^{K/2} \\
        &\ll N(T) \omega \frac{(2\pi e)^K\omega^{K-1} }{K^K}(cK\Psi)^{K/2}
        \ll N(T) \omega \bigg(\sdfrac{c\, \Omega\sqrt{\Psi}}{\sqrt{K}}\bigg)^{K}  
        \ll N(T)\frac{\omega}{2^{K}},
        \end{split}
        \]
 where the final estimate follows from \eqref{Omega K}.
    Now we apply Lemma \ref{lemma 3.4} to the main term on the right-hand side of \eqref{eq:exponential expansion} and obtain
        \be\label{eq:fourier} 
        \begin{split}
        \sum_{0\leq k< K}\frac{(2\pi i \omega)^k}{k!}\sum_{0 < \g \leq T}\bigs(\Re\CMcal{P}_X(\g+v)\bigs)^k 
        &=N(T)\sum_{0\leq k< K}\frac{(2\pi i \omega)^k}{k!}\int_0^1 \bigs(\Re{\CMcal{P}_v(\underline{\theta})}\bigs)^k \mathop{d\underline{\theta}}  \\
        &-\frac{T}{\pi}
         \sum_{0\leq k< K}\frac{(2\pi i \omega)^k}{k!} 
         \sum_{\substack{q\leq X^2,\\ 1\leq\ell\leq k}} \frac{\log q}{q^{\ell/2}}\int_0^1\bigs(\Re{\CMcal{P}_v(\underline{\theta})}\bigs)^k \Re e^{2\pi i\ell\theta_q}\mathop{d{\underline{\theta}}}   \\
         &\quad +O\bigg(\sqrt T\log^2{T}\sum_{1\leq k< K}\frac{(2\pi \omega)^k}{k!}(ck)^k\bigg). \\
        &= A_1+ A_2+A_3 . 
        \end{split}
        \ee
    Note that the sum in the error term starts from $k=1$ because we can assume that there is no error term in \eqref{eq:lemma 3.4 real} when $k=0.$
    
    We start by treating $A_1.$ Via a reverse use of \eqref{eq:exponential}, we retrieve the Fourier transform of $\Re\CMcal{P}_v(\underline{\theta}).$
        \[
        \begin{split}
        A_1
        = N(T)\sum_{0\leq k< K}&\frac{(2\pi i \omega)^k}{k!}\int_0^1 \bigs(\Re{\CMcal{P}_v(\underline{\theta})}\bigs)^k \mathop{d\underline{\theta}}  \\
        &=N(T)\int_0^1 \exp\bigs(2\pi i \omega\Re{\CMcal{P}_v(\underline{\theta})\bigs)} \mathop{d\underline{\theta}}
        +O\bigg(N(T)\frac{(2\pi\omega)^K}{K!}\int_0^1\bigs|\Re{\CMcal{P}_v(\underline{\theta})}\bigs|^K \mathop{d\underline{\theta}} \bigg).
        \end{split}
        \]
    By Lemma \ref{Tsang lemma 3.4}, the error term here is 
        \[
        \ll  N(T)\frac{(2\pi \omega)^K}{K!}(cK\Psi)^{K/2}, 
        \]
    which is 
       \[
       \ll N(T)  \frac{\omega(2\pi e)^K\omega^{K-1}}{K^K} (cK\Psi)^{K/2}
       \ll N(T) \omega \bigg(\frac{c\, \Omega\sqrt{\Psi}}{\sqrt{K}}\bigg)^K
       \ll N(T) \frac{\omega}{2^{K}}
       \]   
    by \eqref{Omega K}.    
    Hence, 
        \be\label{A1 fourier}
        A_1
        =N(T)\int_0^1 \exp\bigs(2\pi i \omega\Re{\CMcal{P}_v(\underline{\theta})}\bigs) \mathop{d\underline{\theta}} +O\bigg(N(T) \frac{\omega}{2^{K}}\bigg).
        \ee
        
   We proceed to the estimation of $A_2$ in \eqref{eq:fourier}. 
   First, observe that we may extend the inner sum in $A_2$ from 
   $1\leq \ell \leq k$ to $1\leq \ell \leq K.$ 
   To see this, note that by the binomial theorem
        \[
        \int_0^1\bigs(\Re{\CMcal{P}_v(\underline{\theta})}\bigs)^k e^{2\pi i\ell\theta_q}\mathop{d{\underline{\theta}}}
        = \frac{1}{2^k}\int_0^1 \sum_{j=0}^k \binom{k}{j}
        \bigg(\sum_{p\leq X^2}\frac{e^{2\pi i\theta_p}}{p^{1/2+iv}}\bigg)^{j}
       \bigg(\overline{\sum_{p\leq X^2}\frac{e^{2\pi i\theta_p}}{p^{1/2+iv}}}\bigg)^{k-j}e^{2\pi i\ell\theta_q}
        \mathop{d{\underline{\theta}}}.
        \]
By comparing the number of primes when this is multiplied out and using  \eqref{orthogonality}, it is clear that unless $k-j=j+\ell$ for some $0\leq j\leq k,$ the right-hand side equals zero.
This implies $1\leq \ell \leq k.$ 
The same is true for the integral   
       \[
       \int_0^1\bigs(\Re{\CMcal{P}_v(\underline{\theta})}\bigs)^k e^{-2\pi i\ell\theta_q}\mathop{d{\underline{\theta}}}.
       \]
Thus, the integral in $A_2$ is $0$ for $\ell> k.$ 
We may therefore extend the sum over $\ell$ up to $K$ in $A_2,$ and write   
    \be\label{A2 fourier}
    \begin{split}
    A_2
    =&-\frac{T}{\pi}\sum_{\substack{q\leq X^2,\\ 1\leq\ell\leq K}} 
     \frac{\log q}{q^{\ell/2}} \sum_{0\leq k< K}\frac{(2\pi i \omega)^k}{k!} 
    \int_0^1 \bigs(\Re{\CMcal{P}_v(\underline{\theta})}\bigs)^k
    \Re e^{2\pi i\ell\theta_q}\mathop{d\underline{\theta}} \\
    =& -\frac{T}{\pi}\sum_{\substack{q\leq X^2,\\ 1\leq\ell\leq K}}
     \frac{\log q}{q^{\ell/2}}  \int_0^1
     \bigg( \sum_{0\leq k< K} 
     \frac{\bigs(2\pi i \omega\Re{\CMcal{P}_v(\underline{\theta})}\bigs)^k}{k!}  \bigg)
     \Re e^{2\pi i\ell\theta_q}\mathop{d\underline{\theta}}  \\
    =&  -\frac{T}{\pi}\sum_{\substack{q\leq X^2,\\ 1\leq\ell\leq K}}
     \frac{\log q}{q^{\ell/2}}M_\ell (q),
    \end{split}
    \ee
say.
Our goal now is to show that we may replace the sum over $k$
in $M_\ell(q)$ by $\exp\bigs(2\pi i\omega\Re{\CMcal{P}_v(\underline{\theta})}\bigs)$
at the cost of a very small error term.
We do this in different ways for $\ell=1$ and $\ell>1.$ 

Suppose first that $\ell>1.$ Then by \eqref{eq:exponential}
    \[
   \begin{split}
   \sum_{0\leq k< K}\frac{(2\pi i \omega)^k}{k!} 
       \int_0^1\bigs(\Re{\CMcal{P}_v(\underline{\theta})}\bigs)^k \Re e^{2\pi i\ell\theta_q}\mathop{d{\underline{\theta}}}
   =&\int_0^1\exp\bigs(2\pi i\omega\Re{\CMcal{P}_v(\underline{\theta})}\bigs)
   \Re e^{2\pi i\ell\theta_q}\mathop{d\underline{\theta}} \\
     &
       + O\bigg(\frac{(2\pi \omega)^K}{K!} 
       \int_0^1\bigs|\Re{\CMcal{P}_v(\underline{\theta})}\bigs|^K  \mathop{d{\underline{\theta}}}\bigg). 
      \end{split}
     \]
By Lemma \ref{Tsang lemma 3.4} and an estimate similar to some above, the error term is
    \[
    \ll \frac{(2\pi\omega\sqrt{cK\Psi})^K}{K!}
    \ll \omega \bigg(\frac{c\Omega^2 \Psi}{K}\bigg)^K
    \ll \frac{\omega}{2^K}.
    \]
Thus for $\ell>1,$
    \be\label{M_ell}
    M_\ell(q)= \int_0^1 \exp\bigs(2\pi i\omega\Re{\CMcal{P}_v(\underline{\theta})}\bigs) \Re e^{2\pi i\ell\theta_q}\mathop{d\underline{\theta}}
    +O\bigg(\frac{\omega}{2^K}\bigg).
    \ee
Now assume that $\ell=1.$ We write   
   \[
   \begin{split}
    \int_0^1   \exp &\bigs(2\pi i\omega\Re{\CMcal{P}_v(\underline{\theta})}\bigs)
    \Re e^{2\pi i  \theta_q}\mathop{d\underline{\theta}} \\
     =&\sum_{0\leq k< K}\frac{(2\pi i \omega)^k}{k!} 
       \int_0^1\bigs(\Re{\CMcal{P}_v(\underline{\theta})}\bigs)^k \Re e^{2\pi i    \theta_q}\mathop{d{\underline{\theta}}}
       +\sum_{k\geq  K}\frac{(2\pi i \omega)^k}{k!} 
       \int_0^1\bigs(\Re{\CMcal{P}_v(\underline{\theta})}\bigs)^k \Re e^{2\pi i   \theta_q}\mathop{d{\underline{\theta}}} .
     \end{split}
     \]     
The first term is $M_1(q)$ so we may write this as
       \be\label{M1+R1}
       \begin{split}     
           \int_0^1 \exp\bigs(2\pi i\omega\Re{\CMcal{P}_v(\underline{\theta})}\bigs)
           \Re e^{2\pi i  \theta_q}\mathop{d\underline{\theta}} 
            =  M_1(q) +  R_1(q). 
        \end{split}
         \ee    
 We now estimate  $R_1(q)$ which we rewrite as
    \be\label{R_1(q)}
    \begin{split}
   R_1(q)=& \,
        \frac12 \sum_{k\geq K}\frac{(2\pi i \omega)^k}{k!} 
       \int_0^1\bigs(\Re{\CMcal{P}_v(\underline{\theta})}\bigs)^k e^{2\pi i\theta_q}\mathop{d{\underline{\theta}}}
           + \frac12  \sum_{k\geq K}\frac{(2\pi i \omega)^k}{k!} 
       \int_0^1\bigs(\Re{\CMcal{P}_v(\underline{\theta})}\bigs)^k e^{-2\pi i\theta_q}\mathop{d{\underline{\theta}}}\\
       =& \, R_1^{'}(q)+R_1{''}(q).
       \end{split}
       \ee   
By the binomial theorem, the integral in $R_1^{'}(q)$ is
        \[
        \frac{1}{2^k} \sum_{j=0}^k \binom{k}{j}  
         \int_0^1 \bigg(\sum_{p\leq X^2}\frac{e^{2\pi i\theta_p}}{p^{1/2+iv}}\bigg)^{j}
       \bigg(\overline{\sum_{p\leq X^2}\frac{e^{2\pi i\theta_p}}{p^{1/2+iv}}}\bigg)^{k-j}e^{2\pi i\theta_q}
        \mathop{d{\underline{\theta}}}.
        \]
    By \eqref{orthogonality}, the integral on the right-hand side is $0$ unless $j+1=k-j,$ that is, $j=\frac{k-1}{2}.$ In that case, we have $q_1\dots q_{(k+1)/2}=p_1\dots p_{(k-1)/2}q$ for some primes $p_1,\dots,p_{(k-1)/2}, q_1,\dots, q_{(k+1)/2}.$
We then can see that the above is equal to 
    \[
    \frac{1}{q^{1/2-iv}}
    \frac{1}{2^k}\binom{k}{(k-1)/2} 
   \sum_{p_1,\dots,p_{(k-1)/2}  \leq X^2}\frac{a_{(k-1)/2}(p_1\dots p_{(k-1)/2})
   a_{(k+1)/2}( p_1\dots p_{(k-1)/2}q)}{p_1\dots p_{(k-1)/2}},
     \]
   where $a_r(p_1p_2\dots p_r)$ denotes the number of permutations of $p_1,p_2,\dots, p_r.$ Note that  
     \[
     a_{(k+1)/2}( p_1\dots p_{(k-1)/2}q)
     \leq \frac{(k+1)a_{(k-1)/2}( p_1\dots p_{(k-1)/2})}{2}, 
     \]
    and also
    \[
    \begin{split}
    \sum_{p_1,\dots,p_{(k-1)/2}  \leq X^2}\frac{a^2_{(k-1)/2}(p_1\dots p_{(k-1)/2})}{p_1\dots p_{(k-1)/2}} 
    \leq &\, \bigs((k-1)/2\bigs)! \sum_{p_1,\dots,p_{(k-1)/2}  \leq X^2}\frac{a_{(k-1)/2}(p_1\dots p_{(k-1)/2})}{p_1\dots p_{(k-1)/2}} \\
    =&\, \bigs((k-1)/2\bigs)! \Big(\sum_{p\leq X^2}\frac{1}{p}\Big)^{(k-1)/2}
    \leq (ck\Psi)^{\frac{k-1}{2}}.
    \end{split}
    \]
    This shows that  
     \[
      \int_0^1\bigs(\Re{\CMcal{P}_v(\underline{\theta})}\bigs)^k e^{2\pi i\theta_q}\mathop{d{\underline{\theta}}}
      \ll \frac{1}{\sqrt q} k(ck\Psi)^{\frac{k-1}{2}}. 
     \]
    Thus, using \eqref{Omega K}, we see that  
        \[
        \begin{split}
        R_1^{'}(q)
        \ll &\,  \frac1{\sqrt q} \sum_{k\geq K} k \frac{(2\pi \omega )^k}{k!}(ck\Psi)^{\frac{k-1}{2}} 
        \ll \,    \frac{ \omega }{\sqrt q} \sum_{k\geq K} \frac{(2\pi c\, \omega  \sqrt{k\Psi})^{k-1}}{(k-1)!}  \\
        \ll & \frac{ \omega }{\sqrt q}  \sum_{k\geq K} \Big(\frac{c\,\Omega\sqrt{ \Psi}  }{\sqrt{k} }\Big)^{k-1}
        \ll  \frac{\omega}{2^K \sqrt q}.
        \end{split}
        \]
In a similar way it follows that  $R_1^{''}(q) \ll   \omega/(2^K\sqrt q).$
Hence by \eqref{R_1(q)},
$ R_1(q) \ll \omega/(2^K\sqrt q).$
By \eqref{M1+R1} we therefore have
    \[
     M_1(q) 
     =\int_0^1   \exp\bigs(2\pi i\omega\Re{\CMcal{P}_v(\underline{\theta})}\bigs)
     \Re e^{2\pi i  \theta_q}\mathop{d\underline{\theta}} 
    +O\bigg(\frac{\omega}{2^K \sqrt q}\bigg).
     \]

Combining this with \eqref{M_ell} in  \eqref{A2 fourier}, we find that
    \[
    \begin{split}
     A_2  
     =&-\frac{T}{\pi} \sum_{\substack{q\leq X^2,\\ 1\leq\ell\leq K}} 
     \frac{\log q}{q^{\ell/2}} \; 
        \int_0^1   \exp{  \bigs(2\pi i\omega\Re{\CMcal{P}_v(\underline{\theta})}\bigs)}\Re e^{2\pi i \ell \theta_q}\mathop{d\underline{\theta}}  \\
        & + O\bigg(\frac{T\omega}{2^K} \sum_{\substack{q\leq X^2,\\ 2<\ell\leq K}} \frac{\log q}{q^{\ell/2}}\bigg) 
        +O\bigg(\frac{T\omega}{2^K} \sum_{q\leq X^2} \frac{\log q}{q}\bigg) .
       \end{split}
      \]
The sums in the error terms here are both $\ll \log X \ll \log T,$ so we see that
     \be\label{A2 fourier final}
     \begin{split}
      A_2 
       =&-\frac{T}{\pi} \sum_{\substack{q\leq X^2,\\ 1\leq\ell\leq K}}
       \frac{\log q}{q^{\ell/2}} \; 
        \int_0^1 \exp\bigs(2\pi i\omega\Re{\CMcal{P}_v(\underline{\theta})}\bigs)
        \Re e^{2\pi i \ell \theta_q}\mathop{d\underline{\theta}} 
        +O\bigg(\frac{N(T)\omega}{2^K}\bigg) .
      \end{split}
      \ee 
    
Next, we estimate $A_3$ in \eqref{eq:fourier}. Clearly 
        \[
        \begin{split}
        A_3 
        \ll \sqrt{T}\log^2{T} \sum_{1\leq k< K}\frac{(2\pi\omega)^k}{k!}(ck)^k
        &\ll \omega \, \Omega^K  \sqrt{T}\log^2{T} \, \sum_{1\leq k< K}\frac{(ck)^k}{k!} \\
        &\ll  \omega \, \Omega^K\sqrt{T}\log^2{T} \,  e^{cK}.
        \end{split}
        \]
By  \eqref{Omega K}, we see that $K\log(2e^c\Omega)\leq \sdfrac 1{10} \log T,$ say,
hence $(2e^c\Omega)^K\ll T^{\tfrac{1}{10}}.$ 
This gives
        \[
        A_3  \ll   \frac{N(T)\omega}{2^K}.
        \]
Combining this with \eqref{A1 fourier}  and \eqref{A2 fourier final}      
in \eqref{eq:fourier}, we obtain the assertion of the lemma.

 \end{proof}


\subsection{Beurling-Selberg Functions}
In this section, we introduce Beurling-Selberg functions which can be used to transmit   information about Fourier transform to obtain information about distribution function.

Let $a$ and $b$ be real numbers with $a< b.$ The indicator function $\mathbbm{1}_{[a,b]}$ is defined as 
    \[
    \mathbbm{1}_{[a,b]}(x)=
    \begin{cases}
    1 &\text{if} \quad x\in [a, b], \\
    0 &\text{else}.
    \end{cases}
    \]
It can also be written as
    \be\label{chi}
    \mathbbm{1}_{[a,b]}(x)=\frac12\sgn(x-a)-\frac12\sgn(x-b)+\frac{\delta_a(x)}{2}+\frac{\delta_b(x)}{2}.
    \ee
Here, $\sgn$ denotes the signum function given by
    \[
    \sgn(x)=
    \begin{cases}
    1 &\text{if} \quad x>0, \\
    -1 &\text{if} \quad x<0, \\
    0 &\text{if} \quad x=0,
    \end{cases}
    \]
and $\delta_a$ is the Dirac delta function at $a.$   

For a parameter $\Omega > 0,$ a Beurling-Selberg function is given by
    \be\label{eq:F}
    F_{\Omega}(x)
    =\Im \int_0^\Omega G\Big(\frac \omega \Omega\Big)\exp{(2\pi ix\omega)}\frac{\mathop{d\omega}}{\omega},
    \ee
where
    \[
    G(u)= \frac{2u}{\pi}+2u(1-u)\cot{(\pi u)} \quad \text{for} \quad u\in[0, 1]. 
    \]   
It is easily seen that $G(u)$ is a positive valued, differentiable function on $[0, 1].$    
We also know that 
    \be\label{eq:sgn}
    \sgn{(x)}
    =F_{\Omega}(x)
    +O\bigg(\sdfrac{\sin^2(\pi\Omega x)}{(\pi\Omega x)^2}\bigg).
    \ee
A proof of this result can be found in \cite[pp. 26--29]{Tsang}. 


Using \eqref{chi} and \eqref{eq:sgn}, the indicator function can be approximated by a differentiable function.
    \be\label{eq:chi F}
    \mathbbm{1}_{[a,b]}(x)
    =\frac12F_\Omega(x-a)-\frac12F_\Omega(x-b)
    +O\bigg(\sdfrac{\sin^2(\pi\Omega(x-a))}{(\pi\Omega(x-a))^2}\bigg)
    +O\bigg(\sdfrac{\sin^2(\pi\Omega(x-b))}{(\pi\Omega(x-b))^2}\bigg).
    \ee


\subsection{The Distribution of $\Re{\CMcal{P}_X(\g+v)}$}\label{subsection distr of P}

The aim of this section is to prove the following result.

\begin{prop}\label{distr of chi Re P v}
    Let $\Psi(T)=\sum_{p\leq T}\frac{1}{p}$ and $\displaystyle X = T^{\frac{1}{16\Psi(T)^6}}.$ Then
    \[
    \frac{1}{N(T)}\sum_{0<\g\leq T}\mathbbm{1}_{[a,b]}\bigs(\Re\CMcal{P}_X(\g+v)\bigs)
    =  \frac{1}{\sqrt{2\pi}}\int_{a/\sqrt{\tfrac12\log\log T}}^{b/\sqrt{\tfrac12\log\log T}} 
       e^{-x^2/2}\mathop{dx}
    +O\bigg(\sdfrac{\log{\Psi(T)}}{\Psi(T)}\bigg).
    \]
\end{prop}

\begin{proof}
        We will prove that
        \be\label{eq:sum F}
        \begin{split}
         \mathcal F
        :=\sum_{0 < \g \leq T}F_\Omega\bigs(\Re\CMcal{P}_X(\g+v)-a\bigs)
        =\frac{N(T)}{\sqrt{2\pi}}\int_{-\infty}^\infty\sgn{\bigg(x-\sdfrac{a}{\sqrt{\Psi/2}}\bigg)}e^{-x^2/2}\mathop{dx}
        +O\bigg(\frac{N(T)}{\Psi^2}\bigg),
        \end{split}
        \ee
        and 
        \be\label{eq:sum sin}
        \mathcal S
        :=\sum_{0 < \g \leq T}\frac{\sin^2\big(\pi \Omega(\Re{\CMcal{P}_X(\g+v)}-a)\big)}{\big(\pi \Omega(\Re{\CMcal{P}_X(\g+v)}-a)\big)^2} =O\bigg( \frac{N(T)}{\Psi^2}\bigg), 
        \ee
        and that the same results hold when we replace $b$ for $a.$
        Then by \eqref{eq:chi F}, it will follow that
        \[
         \sum_{0<\g\leq T}\mathbbm{1}_{[a,b]}\bigs(\Re\CMcal{P}_X(\g+v)\bigs)
          =\frac{N(T)}{\sqrt{2\pi}}\int_{a/\sqrt{\Psi/2}}^{b/\sqrt{\Psi/2}} 
          e^{-x^2/2}\mathop{dx}
          +O\bigg(\frac{N(T)}{\Psi^2}\bigg).
         \]
       By our choice of $X,$ we have 
        \[
        \Psi=\log\log T+O\bigs(\log \Psi(T)\bigs). 
        \] 
       Thus the right-hand side is
       \[
       \frac{N(T)}{\sqrt{2\pi}}
       \int_{a/\sqrt{\tfrac12\log\log T}}^{b/\sqrt{\tfrac12\log\log T}}
       e^{-x^2/2}\mathop{dx}+O\bigg(N(T)\sdfrac{\log{\Psi(T)}}{\Psi(T)}\bigg). 
       \]
        Since 
       $
       \frac{N(T)}{\Psi^2} \ll N(T)\frac{\log{\Psi(T)}}{\Psi(T)}, 
       $
       the proof of the proposition will then be complete. 
       
      We start with the proof of \eqref{eq:sum sin}. For this, we take advantage of the  identity 
        \[
        \frac{\sin^2(\pi \Omega x)}{(\pi \Omega x)^2}
        =\frac{2}{\Omega^2}\int_0^{\Omega}(\Omega-\omega)
        \cos(2\pi x\omega)\mathop{d\omega},
        \]
      and find that
        \[
        \mathcal S
        =\frac{2}{\Omega^2}\int_0^{\Omega}(\Omega-\omega)
        \Re{\sum_{0 < \g \leq T}\exp\bigs(2\pi i\omega(\Re\CMcal{P}_X(\g+v)-a)\bigs)}\mathop{d\omega}.
        \]
    By Lemma \ref{fourier}, we then see that
        \begin{align}\label{sin term}
        \mathcal S
        \ll & \,  \frac{N(T)}{\Omega^2}\int_0^\Omega (\Omega-\omega)\bigg|\int_0^1\exp\bigs(2\pi i\omega\Re{\CMcal{P}_v(\underline{\theta})}\bigs)\mathop{d\underline{\theta}}\bigg|
        \mathop{d\omega} \notag \\
        & +\frac{T}{\Omega^2}\sum_{\substack{q\leq X^2,\\ 1\leq\ell \leq K}}\frac{\log q}{q^{\ell/2}}\int_0^\Omega (\Omega-\omega)\bigg|
         \int_0^1\exp\bigs(2\pi i\omega\Re{\CMcal{P}_v(\underline{\theta})}\bigs) \Re e^{2\pi i\ell\theta_q}\mathop{d\underline{\theta}}
         \bigg|\mathop{d\omega} \notag  \\
        &+ \, N(T)\frac{\Omega}{2^K} \notag \\
        =& \, \mathcal S_1+\mathcal S_2+O\bigg(N(T)\frac{\Omega}{2^K}\bigg),
        \end{align}
say.        

    We first consider $\mathcal S_1.$ By the definition of $\Re\CMcal{P}_v(\underline{\theta})$ and the definition of the Bessel function of order $0,$ that is, 
    \[
    J_0(z)=\int_0^1\exp\bigs(iz\cos(2\pi\theta)\bigs)\mathop{d\theta},
    \]
   it follows that
        \be\label{exp int 1}
        \begin{split}
        \int_0^1\exp\bigs(2\pi i\omega\Re{\CMcal{P}_v(\underline{\theta})}\bigs)\mathop{d\underline{\theta}}
        &=\int_0^1\prod_{p\leq X^2}\exp\bigg(\frac{2\pi i\omega \cos(2\pi\theta_p-v\log p)}{\sqrt{p}}\bigg)\mathop{d\underline{\theta}} \\
        &=\prod_{p\leq X^2}J_0\Big(\frac{2\pi \omega}{\sqrt p}\Big).
        \end{split}
       \ee 
Then by Lemma 4.2 in \cite{Tsang}, we find that 
     \be\label{eq:Bessel}
       \int_0^1\exp\bigs(2\pi i\omega\Re{\CMcal{P}_v(\underline{\theta})}\bigs)\mathop{d\underline{\theta}}
       \ll e^{ -c\Psi\omega^2}. 
       \ee   
Thus
     \[
     \mathcal S_1      
     \ll   \frac{N(T)}{\Omega^2}\int_0^\Omega
     (\Omega-\omega)e^{-c\Psi\omega^2}\mathop{d\omega} 
     \leq \frac{N(T)}{\Omega }\int_0^\infty e^{-c\Psi\omega^2}\mathop{d\omega} 
     \ll  \frac{N(T)}{\Omega \sqrt \Psi} .
      \]
Now we estimate $\mathcal S_2.$ Let 
        \[
        \mathcal J
        = \mathcal J (q, \ell, \omega)
        =\int_0^1\exp\bigs(2\pi i\bigs(\omega\Re{\CMcal{P}_v(\underline{\theta})+\ell\theta_q}\bigs)\bigs) \mathop{d\underline{\theta}}.
        \]
The integral with respect to $\underline{\theta}$ in $\mathcal S_2$ is 
        \[
         \frac{\mathcal J (q, \ell, \omega)+\mathcal J (q, -\ell, \omega)}{2}=
         \int_0^1\exp\bigs(2\pi i\omega\Re{\CMcal{P}_v(\underline{\theta})}\bigs)
         \Re e^{2\pi i\ell\theta_q}\mathop{d\underline{\theta}}.
        \]
We start with estimating $\mathcal{J}.$ 
Using the independence of the variables $\{\theta_p\},$ we may separate the term corresponding to the prime $q$ and write
        \begin{align*}
        \mathcal J
        =\int_0^1 &\exp\bigg(2\pi i\Big(\omega\frac{\cos(2\pi\theta_q-v\log q)}{\sqrt q}+\ell\theta_q\Big)\bigg)\mathop{d{\theta_q}} \\
        &\cdot \prod_{\substack{p\leq X^2\\ p\neq q}}\int_0^1 \exp\Big(2\pi i\omega\frac{\cos(2\pi\theta_p-v\log p)}{\sqrt p}\Big)\mathop{d{\theta_p}}.
        \end{align*}
Here, we recall the definition of the Bessel function of the first kind and of order $\ell$:
    \be\label{Jell integral}
    J_\ell(z)
    =(-i)^\ell\int_0^1 \exp\bigg(2\pi i \Big(\sdfrac{z\cos(2\pi\theta)}{2\pi}+\ell\theta\Big)\bigg)\mathop{d\theta}.
    \ee        
Then in the above expression for $\mathcal J,$ each factor corresponding to a prime $p\neq q$ is $\displaystyle J_0\Big(\sdfrac{2\pi \omega}{\sqrt p}\Big).$ For the remaining term corresponding to the prime $q,$ we apply the change of variable  $\theta_q \to \theta_q+\sdfrac{v\log q}{2\pi},$ and find that it equals
       \[
        \int_{\tfrac{v\log q}{2\pi}}^{1+\tfrac{v\log q}{2\pi}}
         \exp\bigg(2\pi i\Big(\omega\frac{\cos(2\pi\theta_q)}{\sqrt q}+\ell\Big(\theta_q+\frac{v\log q}{2\pi}\Big)\Big)\bigg)
        \mathop{d{\theta_q}}.
        \]
The integrand has period $1,$ so we see from \eqref{Jell integral} that this equals   
        \[
        (iq^{iv})^\ell J_{\ell}\Big(\frac{2\pi\omega}{\sqrt q}\Big).
        \]
Hence
        \be\label{J}
        \mathcal J
        =\mathcal J (q,\ell, \omega)
        =(iq^{iv})^\ell J_{\ell}\Big(\frac{2\pi\omega}{\sqrt q}\Big)\prod_{\substack{p\leq X^2,\\p\neq q}}J_0\Big(\frac{2\pi\omega}{\sqrt p}\Big).
        \ee
One can similarly prove that
        \be\label{J minus ell}
        \mathcal J (q,-\ell, \omega)
        =(iq^{-iv})^\ell J_{\ell}\Big(\frac{2\pi\omega}{\sqrt q}\Big)\prod_{\substack{p\leq X^2,\\p\neq q}}J_0\Big(\frac{2\pi\omega}{\sqrt p}\Big).
        \ee      
We estimate $\mathcal J$ in two different ways, depending on the size of $\omega.$ 
First, we note that if $|z| \leq 1$ and $\ell$ is a nonnegative integer, then
    \be\label{Bessel estimate}
    J_\ell(2z)=\frac{z^\ell}{\ell!}e^{\textstyle{-\frac{z^2}{\ell+1}}}\, \Big(1+O(z^4)\Big).
    \ee
This can be proven by using the series expansion (see \cite[p. 15]{Watson})
    \be\label{eq:series for Bessel}
    J_\ell(z)=\sum_{j=0}^\infty \frac{(-1)^j}{j!(\ell+j)!}\Big(\frac{z}{2}\Big)^{2j+\ell}.
    \ee 
From the above series, we see that \eqref{Bessel estimate} holds for $z=0.$ For $z\neq 0,$ together with the series expansion of $e^{z^2/(\ell+1)},$ we have
   \begin{align*}
   J_\ell(2z)\frac{\ell!}{z^\ell}e^{z^2/(\ell+1)}
   &= \bigg(\sum_{j=0}^\infty \frac{(-1)^j \ell !}{j!(\ell+j)!}z^{2j}\bigg) \bigg( \sum_{k=0}^\infty 
   \frac{1}{k!(\ell+1)^k}z^{2k} \bigg) \\
   &= \sum_{j=0}^\infty \sum_{k=0}^\infty 
   \frac{(-1)^j \ell !}{j!(\ell+j)! k!(\ell+1)^k} z^{2j+2k}.
   \end{align*}
   In the double series, we see that the term for $(j,k)=(0, 0)$ is $1,$ 
   and that the term for $(j,k)=(0, 1)$ cancels the term for $(j,k)=(1, 0).$ 
   Thus for $|z| \leq 1,$ the double series is 
   \begin{align*}
    =\sum_{j=1}^\infty \sum_{k=1}^\infty 
    \frac{(-1)^j \ell !}{j!(\ell+j)! k!(\ell+1)^k} z^{2j+2k} 
    &\leq z^4 \sum_{j=1}^\infty  \frac{\ell !}{j!(\ell+j)!} \sum_{k=1}^\infty 
    \frac{1}{k!(\ell+1)^k} \\
    &\ll z^4.
    \end{align*}    
This proves \eqref{Bessel estimate}.    

Now, note that if $\omega \in[0,\sqrt 2/\pi],$ then $\frac{\pi\omega}{\sqrt{p}}\leq 1$ for all primes $p,$ and we can apply \eqref{Bessel estimate} to estimate the Bessel functions in \eqref{J}. This gives
        \begin{align*}
        \mathcal J
        &=(iq^{iv})^\ell \frac{(\pi\omega)^\ell}{\ell!q^{\ell/2}}e^{-\tfrac{\pi^2 \omega^2}{q(\ell+1)}}\Big(1+O\Big(\frac{\omega^4}{q^2}\Big)\Big)
        \prod_{\substack{p\leq X^2,\\p\neq q}} e^{-\tfrac{\pi^2 \omega^2}{p}}\Big(1+O\Big(\frac{\omega^4}{p^2}\Big)\Big)\\
        &= \frac{(iq^{iv}\pi\omega)^\ell}{\ell!q^{\ell/2}}e^{-\pi^2\omega^2\left(\Psi-\tfrac{\ell}{(\ell+1)q}\right)}\Big(1+O\bigs(\omega^{4}\,\bigs)\Big).
        \end{align*}
Introducing the notation $\Psi_q=\Psi-\sdfrac{\ell}{(\ell+1)q},$
we may write 
     \be\label{product Bessels 1}
        \mathcal J
        =\frac{(iq^{iv}\pi\omega)^\ell}{\ell!q^{\ell/2}}e^{-\pi^2\Psi_q\omega^2}\Big(1+O\bigs(\omega^{4}\,\bigs)\Big) 
        \quad \text{for} \quad \omega \in[0,\sqrt{2}/\pi].
        \ee
 Now suppose that $\omega\in [\sqrt 2/\pi, \Omega].$
Then from the proof of Lemma 4.2 in \cite{Tsang}, it is not hard to see that
        \[
        \prod_{\substack{p\leq X^2, \\p\neq q}} 
        J_0\Big(\frac{2\pi\omega}{\sqrt p}\Big)
        \ll e^{-c\Psi\omega^2}
        \]
for some positive constant $c.$ 
For the Bessel function corresponding to the prime $q$ in \eqref{J}, we use the series expansion in \eqref{eq:series for Bessel} to obtain
        \begin{align*}
        J_\ell \Big(\frac{2\pi\omega}{\sqrt q}\Big)
        &=\sum_{j=0}^\infty 
        \frac{(-1)^j}{j!(j+\ell)!}\Big(\frac{\pi\omega}{\sqrt q}\Big)^{2j+\ell}  \\
        &\leq \Big(\frac{\pi\omega}{\sqrt q}\Big)^{\ell}\sum_{j=0}^\infty \frac{(\pi^2\omega^2)^j}{2^j j!(j+\ell)!} 
        \leq \frac{(\pi\omega)^\ell}{\ell !q^{\ell/2}}\sum_{j=0}^\infty 
        \frac{(\sqrt 2\pi\omega)^{2j}}{(2j)!}
        \leq \frac{(\pi\omega)^\ell}{\ell !q^{\ell/2}} e^{\sqrt 2\pi\omega}.
        \end{align*}
In the next-to-last inequality we have used the bound 
        \be\label{binom ineq}
        2^j j!(j+\ell)! \geq \frac{(2j)!\ell !}{2^j},
        \ee
which holds for all integers $j, \ell \geq 0.$ 
To see this, first note that from $\binom{j+\ell}{j} \geq 1,$
we have $(j+\ell)!\geq  j! \ell!.$
Secondly, $\binom{2j}{j} \leq 2^{2j},$ so $2^j j! \geq \sdfrac{(2j)!}{2^{j}j!}.$
Combining these inequalities, we obtain \eqref{binom ineq}. Then it follows that for $\omega\in [\sqrt{2}/\pi, \Omega]$
        \[
        \mathcal J
        \ll \frac{(\pi\omega)^\ell}{\ell !q^{\ell/2}}  e^{-c\Psi\omega^2+\sqrt{2}\pi\omega}.
        \]     
Combining this with \eqref{product Bessels 1}, we obtain
       \be\label{product Bessels 4}
        \mathcal J(q, \ell, \omega) 
       \ll \frac{(\pi\omega)^\ell}{\ell !q^{\ell/2}}  e^{-c\Psi\omega^2+\sqrt{2}\pi\omega}
        \, \quad \text{ for }  \omega \in[0, \Omega].
        \ee
By \eqref{J minus ell}, the same bound holds for $ \mathcal J(q, -\ell, \omega).$      
We use this bound to estimate $\mathcal S_2$ in \eqref{sin term} as
        \begin{align*}
        \mathcal S_2
        \ll &\,  \frac{T}{\Omega^2}\sum_{1\leq\ell \leq K} \frac{\pi^\ell}{\ell!}
        \sum_{q\leq X^2}  \frac{\log q}{q^{\ell}}\int_0^\Omega
        (\Omega-\omega)\omega^\ell e^{-c\Psi\omega^2+\sqrt 2\pi\omega}\mathop{d\omega} \\
        \ll &\, \frac{T}{\Omega} \sum_{q\leq X^2}  \frac{\log q}{q}
       \int_0^\Omega   \bigg( \sum_{1\leq\ell \leq K} \frac{(\pi\omega)^\ell}{\ell!}\bigg)
       e^{-c\Psi\omega^2+\sqrt 2\pi\omega}\mathop{d\omega}
        \\
        \ll &\, \frac{T}{\Omega}\sum_{q\leq X^2}  \frac{\log q}{q}
        \int_0^\Omega e^{-c\Psi\omega^2+(\sqrt 2+1)\pi\omega}\mathop{d\omega}
        \ll \frac{T\log X}{\Omega \sqrt \Psi}\ll  \frac{N(T)}{\Omega \sqrt \Psi}.
        \end{align*} 
  Combining our estimates for  $\mathcal S_1$ and $\mathcal S_2$    
in \eqref{sin term}, we obtain
       \[
        \mathcal S
        \ll \frac{N(T)}{\Omega \sqrt \Psi} + \frac{N(T) \Omega}{2^K}.
        \]
 By \eqref{Omega K}, this is $O\left(\sdfrac{N(T)}{\Psi^2}\right),$ so this proves \eqref{eq:sum sin}.   

    We proceed to prove \eqref{eq:sum F}.    
    Substituting $x=\Re{\CMcal{P}_X(\g+v)}-a$ in \eqref{eq:F} and summing both sides of the equation over $\g,$ we obtain
        \[
        \mathcal F
        =\Im \int_0^ \Omega G\Big(\frac{\omega}{\Omega}\Big)e^{-2\pi i a\omega}\sum_{0 < \g \leq T}\exp\bigs(2\pi i\omega\Re\CMcal{P}_X(\g+v)\bigs)\frac{\mathop{d\omega}}{\omega}. 
        \]
By Lemma \ref{fourier}, this becomes
        \begin{align*}
        \mathcal F
        &=N(T)\Im \int_0^\Omega G\Big(\frac{\omega}{\Omega}\Big)e^{-2\pi i a\omega}\int_0^1\exp\bigs(2\pi i\omega\Re \CMcal{P}_v(\underline{\theta})\bigs)\mathop{d\underline{\theta}}
        \frac{\mathop{d\omega}}{\omega}  \\
        &-\sdfrac{T}{\pi}\sum_{\substack{q\leq X^2,\\ 1\leq\ell \leq K}}\sdfrac{\log q}{q^{\ell/2}} \Im \int_0^{\Omega}G\Big(\sdfrac{\omega}{\Omega}\Big)e^{-2\pi ia\omega}
        \int_0^1\exp\bigs(2\pi i\omega\Re{\CMcal{P}_v(\underline{\theta})}\bigs) \Re e^{2\pi i\ell\theta_q}\mathop{d\underline{\theta}}\frac{\mathop{d\omega}}{\omega} \\
         &+O\bigg(\sdfrac{N(T)}{2^K}\int_0^\Omega \Big| G\Big(\frac{\omega}{\Omega}\Big)e^{-2\pi ia\omega}\Big| \mathop{d\omega} \bigg) 
        =\, \mathcal F_1+\mathcal F_2+\mathcal F_3. 
        \end{align*}       
\noindent We estimate these three terms in turn. 
By \eqref{exp int 1},
        \[
        \mathcal F_1
        = N(T)\Im\int_0^{\Omega}G\Big(\frac{\omega}{\Omega}\Big)e^{-2\pi ia\omega}\prod_{p\leq X^2}J_0\Big(\frac{2\pi\omega}{\sqrt p}\Big)\frac{\mathop{d\omega}}{\omega}.
        \]
The imaginary part of the integral here has been evaluated by Tsang (see~\cite[pp. 34--35]{Tsang}). Using his result, we obtain
        \[
        \mathcal F_1
        =\frac{N(T)}{\sqrt{2\pi}} \int_{-\infty}^\infty 
        \sgn\bigg(x-\frac{a}{\sqrt{\Psi/2}}\bigg)e^{-x^2/2}\mathop{dx}
        +\, O\bigg(\frac{N(T)}{\Psi^2}\bigg).
        \]
To treat $\mathcal F_2,$ recall the notation 
     \[
     \mathcal J(q, \ell, \omega)= \int_0^1\exp\bigs(2\pi i\bigs(\omega\Re{\CMcal{P}_v(\underline{\theta})+\ell\theta_q}\bigs)\bigs) \mathop{d\underline{\theta}}.
     \]  
Then we may write 
        \begin{align*}
        \mathcal F_2
         =&-\frac{T}{2\pi}\sum_{\substack{q\leq X^2,\\ 1\leq\ell \leq K}}\frac{\log q}{q^{\ell/2}}\Im\int_0^\Omega G\Big(\frac{\omega}{\Omega}\Big)e^{-2\pi ia\omega}
          \mathcal J(q, \ell, \omega) \frac{\mathop{d\omega}}{\omega} \\
         &-\frac{T}{2\pi}\sum_{\substack{q\leq X^2,\\ 1\leq\ell \leq K}}\frac{\log q}{q^{\ell/2}}\Im\int_0^\Omega G\Big(\frac{\omega}{\Omega}\Big)e^{-2\pi ia\omega}
          \mathcal J(q, -\ell, \omega)\frac{\mathop{d\omega}}{\omega}  
         =  \mathcal F'_2+\mathcal F^{''}_2,
         \end{align*} 
say. Now by \eqref{product Bessels 4}, 
        \[
        \begin{split}
         \mathcal F'_2 
        =&- \frac{T}{2\pi} 
        \sum_{\substack{q\leq X^2,\\ 1\leq\ell\leq K}}\frac{\log q}{q^{\ell/2}}
       \;  \Im\int_0^\Omega G\Big(\frac{\omega}{\Omega}\Big)e^{-2\pi ia\omega}\,
       \mathcal J(q, \ell, \omega) \frac{\mathop{d\omega}}{\omega} \\
       \ll& \, T \sum_{\substack{q\leq X^2,\\ 1\leq\ell\leq K}}
        \frac{\pi ^\ell \log q}{\ell! q^{\ell}}
         \int_0^\Omega G\Big(\frac{\omega}{\Omega}\Big)
        \omega^{\ell-1} e^{-c\Psi\omega^2+\sqrt 2\pi\omega} \mathop{d\omega}.
        \end{split}
        \]      
Since $G$ is bounded on $[0, 1],$ we see that
        \begin{align*}
        \mathcal F'_{2}      
        & \ll \, T   \sum_{q\leq X^2}\frac{\log q}{q} 
        \sum_{1\leq\ell \leq K}  \frac{\pi^\ell}{\ell!}
         \int_0^\Omega \omega^{\ell-1}
         e^{-c\Psi\omega^2+\sqrt 2\pi\omega}\mathop{d\omega} \\
      &\ll    T   \sum_{q\leq X^2}\frac{\log q}{q}  
        \int_0^\Omega \sum_{1 \leq \ell \leq K} \frac{(\pi \omega)^{\ell-1} }{(\ell-1)!}
      e^{-c\Psi\omega^2+\sqrt 2\pi\omega}\mathop{d\omega}  \\
      & \ll   T   \sum_{q\leq X^2}\frac{\log q}{q}  
      \int_0^\Omega 
      e^{-c\Psi\omega^2+(\sqrt 2\pi+\pi)\omega}\mathop{d\omega}.   
      \end{align*}  
The integrand is bounded and the sum over $q$ is $\ll \log X,$ so by \eqref{Omega K},
        \[
         \mathcal F'_{2} 
         \ll T\, \Omega \log X
         \ll \frac{N(T)}{\Psi(T)^4} .
        \]     
Similarly, $\mathcal F^{''}_2 \ll  \frac{N(T)}{\Psi(T)^4}.$ Finally, since $G$ is a bounded function on $[0, 1],$ 
         \[
         \mathcal F_3  
          \ll  \frac{N(T) \Omega}{2^K} 
          \ll \frac{N(T)}{\Psi^2}
         \] 
by \eqref{Omega K}. Recall that we set $\mathcal F=\sum_{1\leq i \leq 3} \mathcal F_i.$ Combining our estimates for $\mathcal F_1,  \mathcal F_2$ and $\mathcal F_3,$ we obtain \eqref{eq:sum F}.
         
    \end{proof}


\subsection{Completing the Proof of Theorem \ref{distr of Re log zeta}}
\label{proof of thm Re log zeta}

Let $\displaystyle X=T^{\tfrac{1}{16\Psi(T)^6}}$ where $\Psi(T)=\sum_{p\leq T} p^{-1},$ and $0<u \ll \frac{1}{\log T}$ and $|v| = O\bigs(\tfrac{1}{\log X}\bigs).$ 
Define    
    \[
    r(X,\g+v)=\log{|\zeta(\r+z)|}-M_X(\r, z)-\Re\CMcal{P}_X(\g+v),
    \]
where $M_X(\r, z)$ is as defined in \eqref{mean} and $\CMcal{P}_X(\g+v)$ is as given in \eqref{P defn}. We also set $\displaystyle Y= T^{\tfrac{1}{8k}}$ where we choose $k=\lfloor \log{\Psi(T)}\rfloor .$ We similarly define
    \[
    r(Y,\g+v)=\log{|\zeta(\r+z)|}-M_Y(\r, z)-\Re\CMcal{P}_Y(\g+v),
    \]
    where
    \[
    M_Y(\r, z)=m(\r+iv)\Big(\log\Big(\sdfrac{eu\log Y}{4}\Big)-\sdfrac{u\log Y}{4}\Big)
    \]
    and
    \[
    \CMcal{P}_Y(\g+v)=\sum_{p\leq Y^2} \frac{1}{p^{1/2+i(\g+v)}}.
    \]
 Observe that
        \[
        \mathbbm{1}_{[a,b]}\bigs(\Re\CMcal{P}_X(\g+v)\bigs)
        =\mathbbm{1}_{[a,b]}\bigs(\log{|\zeta(\r+z)|}-M_X(\r, z)\bigs)
        \]
    unless 
        \[
        \bigs| r(X,\g+v)\bigs|  > \bigs|\Re\CMcal{P}_X(\g+v)-a\bigs|
        \quad \text{or}\quad 
         \bigs| r(X,\g+v)\bigs|  > \bigs|\Re\CMcal{P}_X(\g+v)-b\bigs|.
        \]
    To estimate the number of such exceptional $\g,$ we define the set
        \[
        A_a=\Big\{0< \g \leq T : \bigs|r(X,\g+v)\bigs| \mkern9mu > \mkern9mu \bigs|\Re\CMcal{P}_X(\g+v)-a\bigs|\Big\},
        \]
     and a similar set $A_b.$  
    Let $\g\in A_a.$ Trivially, for any positive constant $C$ we either have $\bigs| r(X,\g+v)\bigs| \leq 2Ck^2$ or $\bigs|r(X,\g+v)\bigs| > 2Ck^2.$
     By the definition of $A_a,$ the first case implies that $\bigs|\Re\CMcal{P}_X(\g+v)-a\bigs| < 2Ck^2.$ 
    Then by Proposition \ref{distr of chi Re P v} and our choice of $k,$
        \[
        \#\Big\{\g\in A_a : a-2Ck^2 < \Re\CMcal{P}_X(\g+v) < a+2Ck^2\Big\} \ll N(T)\sdfrac{\log^2 {\Psi(T)}}{\sqrt{\Psi(T)}}.
        \]
    On the other hand, by Chebyshev's inequality, 
        \be\label{Chebyshev}
        \#\Big\{\g\in A_a:\bigs| r(X,\g+v)\bigs| > 2Ck^2\Big\}
         \leq \frac{1}{(2Ck^2)^{2k}}\sum_{0<\g\leq T} \bigs| r(X,\g+v)\bigs|^{2k}. 
        \ee
     Hence, we need to estimate moments of $r(X,\g+v).$ Since $X < Y,$ 
     \begin{align*}
     & r(X,\g+v)\\
     = &\, 
      r(Y,\g+v)+\bigs(M_Y(\r,z)-M_X(\r,z)\bigs)+\bigs(\Re\CMcal{P}_Y(\g+v)-\Re\CMcal{P}_X(\g+v)\bigs)\\
     =&\, r(Y,\g+v)
     +m(\r+iv)\Big(\log\Big(\sdfrac{\log Y}{\log X}\Big)-u\sdfrac{\log(Y/X)}{4}\Big) 
     +\Re\sum_{X^2 < p \leq Y^2}\frac{1}{p^{1/2+i(\g+v)}}\\
     =&\, A_1+A_2+A_3 ,
     \end{align*}
     say.
    Then by \eqref{convexity}, we have
    \be\label{moment of r}
     \sum_{0<\g\leq T}\bigs|r(X,\g+v)\bigs|^{2k}
    \ll c^k \bigg( \sum_{0<\g\leq T}\bigs|A_1\bigs|^{2k} +
     \sum_{0<\g\leq T}\bigs|A_2\bigs|^{2k}
     + \sum_{0<\g\leq T}\bigs|A_3\bigs|^{2k}\bigg). 
    \ee
     Here, Proposition \ref{Re log zeta error} gives
     \be\label{moment of A1 prop}
     \sum_{0<\g\leq T} \bigs|A_1\bigs|^{2k} \ll (ck)^{4k} N(T). 
     \ee
     For the $(2k)$th moment of $A_2,$ note that 
     \be\label{moment of m}
     \begin{split}
     A_2&= \bigs(-\log(8k)+\log(16\Psi(T)^6)\bigs)m(\r+iv) \\
     &\ll m(\r+iv)\bigs(\log k+\log\log\log T \bigs)
     \ll k\, m(\r+iv),
     \end{split}
     \ee
     where we used our choice for $k$ and Mertens' theorem $\Psi(T) = \log\log T+O(1).$ 
     Thus, it suffices to study the $(2k)$th moment of $m(\r+iv).$ 
     For each positive integer $j,$ let $N_j(T)$ denote the number of zeros $\r$ with its ordinate in $(0, T]$ and with multiplicity $j,$ that is, $m(\r)=j.$ By a result of Korolev (see Theorem A in ~\cite{Korolev}), it is known that
        \[
         N_j(T) \ll e^{-cj}N(T),
        \]
where $c$ is an absolute constant. Clearly, $N(T+|v|)\ll N(T)$ since $|v| \ll \sdfrac{1}{\log X}.$ Thus, the number of zeros of the form $\r+iv$ with $\g\in(0, T]$ and $m(\r+iv)=j$ is
     \[
     \ll e^{-cj}N(T). 
     \]
Then by \eqref{moment of m}, 
     \[
      \sum_{0<\g\leq T} \bigs|A_2\bigs|^{2k}
      \ll (ck)^{2k}  \sum_{0<\g\leq T} \bigs| m(\r+iv)\bigs|^{2k}
      \ll  (ck)^{2k}\sum_{j=1}^\infty \frac{j^{2k}}{e^{c j}}\, N(T). 
     \]
In a similar way to our estimation of the series in 
     \eqref{eq:proof eta}, the above series can be seen to be $\ll (ck)^{2k}.$ Hence
     \be\label{moment of m A2}
      \sum_{0<\g\leq T} \bigs|A_2\bigs|^{2k} \ll (ck)^{4k} N(T).
     \ee
Finally, we estimate moments of $A_3.$ We easily have
      \[
      \sum_{0<\g\leq T} \bigs|A_3\bigs|^{2k}
      \ll \sum_{0<\g\leq T}\Big|\sum_{X^2<p\leq Y^2} \frac{1}{p^{1/2+i(\g+v)}}\Big|^{2k}. 
      \] 
By Lemma \ref{Soundmomentlemma3}, the right-hand side is
     \[
     \ll k! N(T)\Big( \sum_{X^2< p\leq Y^2} \frac{1}{p}\Big)^k, 
     \]
Then by Mertens' theorem and our choices for $X$ and $Y,$ this is
    \[
    (ck)^k N(T) \bigs(\log(\log Y/\log X) \bigs) ^k \ll (ck)^k N(T) (\log k+\log\log\log T)^k, 
    \] 
and so
    \[
     \sum_{0<\g\leq T} \bigs|A_3\bigs|^{2k}
     \ll (ck)^{2k} N(T).
    \]
Together with \eqref{moment of A1 prop} and \eqref{moment of m A2}, we substitute this into \eqref{moment of r} and obtain
    \[
    \sum_{0<\g\leq T}\bigs| r(X,\g+v)\bigs|^{2k} \ll (ck)^{4k} N(T). 
    \]
We now go back to \eqref{Chebyshev}, which, together with the above estimate, implies 
    \[
    \#\Big\{\g\in A_a:\bigs| r(X,\g+v)\bigs| > 2Ck^2\Big\}
    \ll \frac{(ck)^{4k}}{(2Ck^2)^{2k}}N(T) 
    \]
for some constant $c>0.$ 
If we choose $C$ such that $C\geq c^2,$ then
        \[
        \#\Big\{\g \in A_a:\bigs|r(X,\g+v)\bigs| > 2Ck^2 \Big\} 
        \leq \frac{N(T)}{4^k}
        \ll N(T)\frac{\log^2{\Psi(T)}}{\sqrt{\Psi(T)}},
        \]
     where we used $k=\lfloor \log{\Psi(T)}\rfloor .$ Thus $ \# A_a=O\Big(N(T)\sdfrac{\log^2{\Psi(T)}}{\sqrt{\Psi(T)}}\Big).$ Similarly, $ \# A_b=O\Big(N(T)\sdfrac{\log^2{\Psi(T)}}{\sqrt{\Psi(T)}}\Big).$ This proves that
        \[
        \sum_{0<\g \leq T}\mathbbm{1}_{[a,b]}\bigs(\Re\CMcal{P}_X(\g+v)\bigs) 
        =\sum_{0<\g\leq T}\mathbbm{1}_{[a,b]}\bigs(\log{|\zeta(\r+z)|}-M_X(\r, z)\bigs)
         +O\bigg(N(T)\frac{\log^2 \Psi(T)}{\sqrt{\Psi(T)}}\bigg). 
         \]
With the result of Proposition \ref{distr of chi Re P v}, the left-hand side is 
        \[
         \frac{N(T)}{\sqrt{2\pi}}\int_{a/\sqrt{\tfrac12\log\log T}}^{b/\sqrt{\tfrac12\log\log T}} 
           e^{-x^2/2}\mathop{dx}
           +O\bigg(N(T)\frac{\log\Psi(T)}{\Psi(T)}\bigg). 
        \]
By replacing $a$ with $a\sqrt{\tfrac12\log\log T}$ and $b$ with $b\sqrt{\tfrac12\log\log T},$ we obtain
         \[
         \sum_{0<\g\leq T}\mathbbm{1}_{[a,b]}\bigg(\frac{\log{|\zeta(\r+z)|}-M_X(\r, z)}{\sqrt{\tfrac12\log\log T}}\bigg) 
         =\frac{N(T)}{\sqrt{2\pi}}\int_a^b e^{-x^2/2}\mathop{dx}
               +O\bigg(N(T)\frac{\log^2 \Psi(T)}{\sqrt{\Psi(T)}}\bigg).
         \]
Noting that $\Psi(T)=\log\log T+O(1)$ by Mertens' theorem, we complete the proof of Theorem \ref{distr of Re log zeta}.   
       
 \begin{rem*}      
    One can see from the above discussion that for $u$ as small as $O\Big( \frac{\log\log\log T}{\log T}\Big),$ we can replace $M_X(\r,z)$ in the statement of Theorem \ref{distr of Re log zeta} by 
       \[
       M(\r,z)=m(\r+iv)\Big(\log\Big(\sdfrac{eu\log T}{4}\Big)-\sdfrac{u\log T}{4}\Big). 
       \]  
\end{rem*}       

\section{Proof of Theorem~\ref{thm: log zeta'}}
\label{proof of thm Re log zeta'}

Our proof of Theorem~\ref{thm: log zeta'} is very similar to that of
Theorem~\ref{distr of Re log zeta}. We are, however, required to assume that all the zeros of the Riemann zeta-function are simple. Another difference from Theorem~\ref{distr of Re log zeta} is that the error resulting from the spacing between consecutive (distinct) zeros will be estimated under the weaker assumption of
Hypothesis $\mathscr{H}_\a$ instead of Montgomery's Pair Correlation Conjecture. 

We start with the formula in Corollary \ref{Re log zeta`}. This gives
       \be\label{eq:difference}
       \log{\frac{|\zeta^{(m(\r))}(\r)|}{\big((e\log{X})/4\big)^{m(\r)}}}-
       \Re\CMcal{P}_X(\g)
        =\sum_{i=1}^{4}  r_i(X, \g),
        \ee
where the terms on the right-hand side are given as    
        \begin{align*}
         r_1(X, \g)=
        \bigg|\sum_{p\leq X^2}\frac{1-w_X(p)}{p^{1/2+i\g}}\bigg| \, , & \quad
       \quad r_2(X, \g)=\bigg|\sum_{p\leq X}\frac{w_X(p^2)}{p^{1+2i\g}}\bigg|\, ,  \\[4pt] 
        r_3(X, \g)= \frac{1}{\log X} \int_{1/2}^\infty X^{\frac 1 2-\s}&\bigg|\sum_{p\leq X^2}\frac{\Lambda_X(p)\log{(Xp)}}{p^{\s+i\g}}\bigg|\mathop{d\s},
        \end{align*}
and  
         \[  
   r_4(X, \g)
   =\bigg(1+\log^{+}\Big(\frac{1}{\eta_{\g}\log X }\Big)\bigg)\frac{E(X,\g)}{\log X}. 
         \]    
In the following proposition, we estimate the moments of the left-hand side of \eqref{eq:difference}.


\begin{prop}\label{Re log zeta' error}
Assume RH and Hypothesis $\mathscr H_\a$ for some $\a\in(0,1].$
Let $\, T^{\tfrac{\d}{8k}} \leq X\leq T^{\tfrac{1}{8k}},$ where $k$ is a positive integer and $0<\d \leq 1$ is fixed.
Then for a constant $A$ depending on $\a$ and $\d,$ we have
      \[
     \sum_{0< \g \leq T} 
       \bigg(\log{\frac{|\zeta^{(m(\r))}(\r)|}{\big((e\log{X})/4\big)^{m(\r)}}}-\Re\CMcal{P}_X(\g)\bigg)^k
       \ll (Ak)^{2k}N(T).
       \] 
\end{prop}

\begin{proof}
    By \eqref{eq:difference} and then the inequality \eqref{convexity}, 
        \[
       \sum_{0< \g \leq T} 
       \bigg(\log{\frac{|\zeta^{(m(\r))}(\r)|}{\big((e\log{X})/4\big)^{m(\r)}}}-\Re\CMcal{P}_X(\g)\bigg)^k
        = O\bigg(c^k\sum_{i=1}^{4} \sum_{0< \g \leq T}  r_i(X, \g)^k\bigg). 
        \]
    The error terms which correspond to $i=1$ and $i=2$ have already been shown to be $\ll (ck)^k N(T)$ in \eqref{1-w} and \eqref{w}, respectively. 
    The third error term is absorbed in this bound by the result of Proposition \ref{error term integral}. 
    It remains to show that 
      \[
         \sum_{0 < \g \leq T}
         \bigg(\Big(1+\log^+ \sdfrac{1}{\eta_{\g}\log{X}}\Big)\frac{E(X,\g)}{\log X}\bigg)^{k}
         \ll (Ak)^{2k}N(T). 
      \]
We will estimate this conditionally on Hypothesis $\mathscr{H}_\a.$ By \eqref{eq:moment of E 2} with $v=0,$ we have 
        \[
        \sum_{0 < \g \leq T} \big|E(X,\g)\big|^{2k}\ll (Dk)^{2k} N(T)  (\log X)^{2k}
        \]
for a constant $D$ depending on $\d.$ It suffices to show that for a constant $A$ depending on $\a$ and $\d,$
        \be\label{eq:moment of eta}
        \sum_{0 < \g \leq T}\Big( 1+\log^+ \sdfrac{1}{\eta_{\g}\log{X}} \Big)^{2k} \ll (Ak)^{2k} N(T). 
        \ee  
Recall the definition
        \[
        \eta_{\g}=\min_{\g' \neq \g}|\g'-\g|.
        \]
    In the case when $\displaystyle \eta_{\g} >\frac{1}{\log X},$ the term on the left-hand side of \eqref{eq:moment of eta} will be just $1$ and such terms will not contribute more than $N(T)$ to the sum. 
To consider larger terms, we assume $\displaystyle \eta_{\g} \leq \frac{1}{\log X}.$ 
Then for a nonnegative integer $j$ 
        \[
         \frac{e^{-j-1}}{\log X} < \eta_{\g} \leq \frac{e^{-j}}{\log X},
        \]
    and so
        \[
        1+\log^+\sdfrac{1}{\eta_\g\log{X}} < j+2 . 
        \]
    By Hypothesis $\mathscr H_\a$ and the comment immediately following it, the number of such $\g$ is at most
        \[
        \#\Big\{0 <  \g \leq T: \eta_\g \leq \sdfrac{e^{-j}}{\log X} \Big\} \ll 
        \min\Big\{  1, \Big(\sdfrac{e^{-j}\log T}{\log X}\Big)^\a \Big\}N(T).
        \]
    Note that $\displaystyle\sdfrac{e^{-j}\log T}{\log X}\geq 1$ for $j\leq \log(\log T/\log X),$ so for $\displaystyle j\leq \log\Big(\sdfrac{8k}{\d}\Big).$
    The contribution of $\g$ corresponding to such $j$ to the sum in \eqref{eq:moment of eta} is at most
        \[
        N(T)\sum_{j \leq \log(8k/{\d})} (j+2)^{2k} \ll (D\log k)^{2k+1} N(T),
        \]
     where $D$ is a constant depending on $\d.$
    Thus, the sum over $\g$ is
        \begin{align*}
        \sum_{0 < \g \leq T}\Big( 1+\log^+ \sdfrac{1}{\eta_\g\log{X}} \Big)^{2k} 
        &\ll N(T) \Big(\sdfrac{\log T}{\log X}\Big)^{\a} \sum_{j=0}^\infty \frac{(j+2)^{2k}}{e^{j\a}} +(D\log k)^{2k+1} N(T) \\
        &\ll \Big(\frac{k}{\d}\Big)^\a N(T)\sum_{j=0}^\infty \frac{(j+2)^{2k}}{e^{j\a}}+(D\log k)^{2k+1} N(T).
        \end{align*}
    By using a similar argument to the one we applied to the series in \eqref{eq:proof eta}, we see that the series on the last line is $\ll (Ak)^{2k}$ for a constant $A$ depending on $\a.$ Hence
      \[
      \sum_{0 < \g \leq T}\Big( 1+\log^+ \sdfrac{1}{\eta_\g\log{X}} \Big)^{2k} 
      \ll  (Ak)^{2k} N(T),
      \]
    where $A$ denotes a constant depending on both $\a$ and $\d.$ 
\end{proof}

\begin{proof}[Proof of Theorem~\ref{thm: log zeta'}]
Let $\displaystyle X = T^{\tfrac{1}{16\Psi(T)^6}}$ where $\Psi(T)=\sum_{p\leq T} p^{-1}.$
By Proposition \ref{distr of chi Re P v} for $v=0,$ we have
    \begin{align*}
    \frac{1}{N(T)}\#\bigg\{0 < \g \leq T: \frac{\Re\CMcal{P}_X(\g)}{\sqrt{\tfrac12\log\log T}}\in [a, b]\bigg\} 
    = \frac{1}{\sqrt{2\pi}}\int_a^b e^{-x^2/2}\mathop{dx}
    +O\bigg(\frac{\log{\Psi(T)}}{\Psi(T)}\bigg).
    \end{align*}
We assume that $m(\r)=1$ for all zeros $\r.$
We quickly see that one can prove the following result by using
a simpler version of the discussion in Section \ref{proof of thm Re log zeta}.
        \be\label{proof of thm log zeta'}
        \begin{split}
        \frac{1}{N(T)}\#\bigg\{0 < \g \leq T: 
        & \frac{\log\sdfrac{|\zeta'(\r)|}{(e\log X)/4}}{\sqrt{\tfrac12\log\log T}}\in [a, b] \bigg\}  \\
        =&\frac{1}{\sqrt{2\pi}}\int_a^b e^{-x^2/2}\mathop{dx}
        +O\bigg(\frac{(\log\log\log T)^2}{\sqrt{\log\log T}}\bigg).
        \end{split}
        \ee
Now, note that there is an absolute constant $c$ for which
        \[
        \bigg|\log\frac{|\zeta'(\r)|}{(e\log X)/4}
        -\log\frac{|\zeta'(\r)|}{\log T}\bigg|
        \leq c\log\log\log T 
        \quad \text{ for all} \, \,\,  \r.   
        \]  
Then for a suitable constant $c,$ 
        \begin{align*}
        \#\bigg\{0 < \g \leq T:  \frac{\log\sdfrac{|\zeta'(\r)|}{(e\log X)/4}}{\sqrt{\tfrac12\log\log T}} \in   
         \bigg[a+c&\sdfrac{\log\log\log T}{\sqrt{\log\log T}}, b-c\sdfrac{\log\log\log T}{\sqrt{\log\log T}}\bigg] \bigg\}
        \\ 
         \leq
         \, &\#\bigg\{0 < \g \leq T: \frac{\log\sdfrac{|\zeta'(\r)|}{\log T}}{\sqrt{\tfrac12\log\log T}}\in 
        [a,b] \bigg\},
        \end{align*}
and
        \begin{align*}
         \#\bigg\{0 < \g \leq& T: \frac{\log\sdfrac{|\zeta'(\r)|}{\log T}}{\sqrt{\tfrac12\log\log T}}\in [a, b] \bigg\}
        \\ 
        \leq  
        &\, \#\bigg\{0 < \g \leq T: \frac{\log\sdfrac{|\zeta'(\r)|}{(e\log X)/4}}{\sqrt{\tfrac12\log\log T}}\in 
        \bigg[a-c\sdfrac{\log\log\log T}{\sqrt{\log\log T}}, b+c\sdfrac{\log\log\log T}{\sqrt{\log\log T}}\bigg] \bigg\}.
        \end{align*}
By \eqref{proof of thm log zeta'},  the left-hand side of the first inequality and 
the right-hand side of the second inequality are both 
        \[
        \frac{N(T)}{\sqrt{2\pi}}\int_a^b e^{-x^2/2}\mathop{dx}
        +O\bigg(N(T)\frac{(\log\log\log T)^2}{\sqrt{\log\log T}}\bigg)
        +O\bigg(N(T)\frac{\log\log\log T}{\sqrt{\log\log T}}\bigg). 
        \]
Combining the inequalities, we get an estimate which is essentially the statement \eqref{proof of thm log zeta'}, but with $\displaystyle \log\Big(\sdfrac{|\zeta'(\r)|}{(e\log X)/4}\Big)$ replaced by $\displaystyle \log\Big(\sdfrac{|\zeta'(\r)|}{\log T}\Big).$ This is the claim of Theorem~\ref{thm: log zeta'}. 
\end{proof}


\section*{Final Comments}
    Similarly to the proof of Selberg's central limit theorem in \cite{Tsang},
    it would be possible to generalize the results in this paper to hold for $\g$ within an interval $[T, T+H],$ where $H$ satisfies $T^\theta < H \leq T$ with $\theta > 1/2.$
    As another generalization of our work, 
    one can consider the value distribution of the sequence $(\log{L(\r, \chi)})$ for a fixed primitive $L$-function $L(s, \chi).$
    In fact, central limit theorems as in Theorem \ref{distr of Re log zeta} and Theorem \ref{distr of Im log zeta} hold in this case,
and they are conditional on the Generalized Riemann hypothesis, a zero-spacing hypothesis between nontrivial zeros of $\zeta(s)$ and nontrivial zeros of $L(s,\chi),$ and the assumption that
nontrivial zeros of $L(s, \chi)$ never coincide with any of the nontrivial zeros of $\zeta(s).$ This can be proven by suitable modifications of the results in this paper.




\begin{thebibliography}{BM}

\normalsize
\baselineskip=16pt
\bibitem{Billingsley} Billingsley, P. : \emph{Probability and measure}. Anniversary edition. With a foreword by Steve Lalley and a brief biography of Billingsley by Steve Koppes. Wiley Series in Probability and Statistics. John Wiley \& Sons, Inc., Hoboken, NJ, 2012. xviii+624 pp.
\bibitem{BJ} Bohr, H. and Jessen, B. : \emph{\"{U}ber die Werteverteilung der Riemannschen Zetafunktion}. Acta Math., \textbf{54}, (1), 1--35, (1930)
\bibitem{BH} Bombieri, E. and Hejhal, D. A. : \emph{On the distribution of zeros of linear combinations of Euler products}. Duke Math. J., \textbf{80}, 821--862, (1995)
\bibitem{GonekLandaulemma1}  Gonek, S. M. : \emph{A formula of Landau and mean values of $\zeta(s)$}. Topics in analytic number theory (Austin, Tex., 1982) 
92--97, Univ. Texas Press, Austin, TX, 1985
\bibitem{GonekLandaulemma2}  Gonek, S. M. : \emph{An explicit formula of Landau and its applications to the theory of the zeta-function}. Contemporary Mathematics 143, 395--413, Amer. Math. Soc., Providence, RI, 1993
\bibitem{H} Hejhal, D. A. : \emph{On the distribution of $\log{|\zeta'(\frac 1 2+it)|}$}. Number theory, trace formulas and discrete groups (Oslo, 1987), 343--370, Academic Press, Boston, MA, 1989
\bibitem{Kirila}  Kirila, S. J. : \emph{Discrete Moments and Linear Combinations of L-functions}. Thesis (Ph.D.)--University of Rochester, 2018. 83 pp. 
\bibitem{Korolev} Korol\"{e}v, M. A. :  \emph{On multiple zeros of the {R}iemann zeta function}. Izv. Ross. Akad. Nauk Ser. Mat. \textbf{70}, no. 3, 3--22, (2006)
\bibitem{Landau} Landau, E. : \emph{\"{U}ber die Nullstellen der Zetafunktion}. German, Math. Ann., \textbf{71}, no. 4, 548--564, (1912)
\bibitem{L} Lester, S.J. : \emph{a-Points of the Riemann zeta-function on the critical line}. Int. Math. Res. Not. IMRN, no. 9, 2406--2436, (2015)
\bibitem{Montgomery73}  Montgomery, H. L. : \emph{The pair correlation of zeros of the zeta-function}. Analytic number theory (Proc. Sympos. Pure Math., \textbf{XXIV}, St. Louis Univ., St. Louis, Mo., 1972), pp. 181--193. Amer. Math. Soc., Providence, R.I., 1973
\bibitem{Selberg1944}  Selberg, A. : \emph{On the remainder in the formula for $N(T),$ the number of zeros of $\zeta(s)$ in the strip $0<t<T$}. Norske Vid. Akad. Oslo. I. 1944, no. 1, 27 pp., (1944)
\bibitem{S1946Archiv} Selberg, A. : \emph{Contributions to the theory of the Riemann zeta-function}. Archiv for Mathematik og Naturvidenskab B. 48, no. 5, 89--155, (1946)
\bibitem{Selberg92} Selberg, A. : \emph{Old and new conjectures and results about a class of Dirichlet series}. Proceedings of the Amalfi Conference on Analytic
Number Theory (Maiori, 1989), 367--385, Univ. Salerno, Salerno, (1992)
\bibitem{SUnpublished} Selberg, A. : \emph{On the value distribution of derivatives of zeta-functions}. https://publications.ias.edu/selberg/section/2483
\bibitem{Soundararajan09} Soundararajan, K. : \emph{Moments of the Riemann zeta-function}. Annals of Mathematics, \textbf{170}, 981--993, (2009)
\bibitem{T} Titchmarsh, E. C. : \emph{The theory of the Riemann zeta-function}.
Second edition. Edited and with a preface by D. R. Heath-Brown. The Clarendon Press, Oxford University Press, New York, 1986. x+412 pp.
\bibitem{Tsang} Tsang, K. : \emph{The Distribution of the Values of the Riemann Zeta-function}. Thesis (Ph.D.)--Princeton University. 1984. 189 pp. 
\bibitem{Watson} Watson, G. N. : \emph{A Treatise on the Theory of Bessel Functions}. Cambridge University Press, Cambridge, England; The Macmillan Company, New York, 1944. vi+804 pp.
\end{thebibliography}
 \end{document}